\documentclass{article} 

\usepackage[T1]{fontenc}
\usepackage{emptypage}
\usepackage{centernot}
\usepackage{array}
\usepackage{pst-grad}
\usepackage{pst-arrow}
\usepackage{pst-text}

\usepackage{latexsym}
\usepackage{amsfonts}
\usepackage{amssymb}
\usepackage{amsmath}
\usepackage{amscd}
\usepackage{eucal}
\usepackage{amsthm}

\usepackage{pstricks}
\usepackage{pst-node}
\usepackage{verbatim} 
\usepackage{graphics}
\usepackage{graphicx}
\usepackage{color}   
\usepackage{framed}

\usepackage{bbm}
\usepackage{pst-plot}
\usepackage{pstricks-add}

\usepackage{xcolor}
\usepackage{mdframed}

\def\1{\ensuremath{\mathbbm{1}}}%
\def\2{\ensuremath{\mathbbm{2}}}%

\newcommand{\R}{\ensuremath{\mathbb R}}

\newcommand{\E}{\ensuremath{\bb{E}}}

\newcommand{\cC}{{\cl{C}}}
\newcommand{\cD}{{\cl{D}}}

\newcommand{\cF}{{\cl{F}}}

\newcommand{\cO}{{\cl{O}}}

\newcommand{\nowdot}{\psset{unit=1mm}\pscircle*(0,0){0.3}}

\newcommand{\eps}{\endpspicture}
\newcommand{\ps}{\pspicture}

\newcommand{\eh}{Eckmann--Hilton}

\newcommand{\bmc}{braided monoidal category}

\newcommand{\dd}{doubly-degenerate}

\newcommand{\iso}{\cong}

\newcommand{\scr}{\scriptsize}

\renewcommand{\:}{\colon}

\newcommand{\bs}{\bigskip}

\newcommand{\bicats}{\ensuremath{\cat{Bicat}_s}}

\newcommand{\ddbicatscat}{doubly-degenerate \bicats-category}

\newcommand{\noi}{\noindent}

\newcommand{\cat}[1]{\ensuremath{\textrm{\bfseries {\upshape {#1}}}}}

\newcommand{\cl}[1]{\ensuremath{\mathcal {#1}}}
\newcommand{\bb}[1]{\ensuremath{\mathbb {#1}}}
\newcommand{\ed}{\end{document}}
\newcommand{\bq}{\begin{quote}}
\newcommand{\eq}{\end{quote}}
\newcommand{\bc}{\begin{center}}
\newcommand{\ec}{\end{center}}
\newcommand{\bmp}{\noi\begin{minipage}}
\newcommand{\emp}{\end{minipage}}

\newcommand{\bfr}{\begin{flushright}}
\newcommand{\efr}{\end{flushright}}

\newcommand{\lra}{\tra}

\newcommand{\dtimes}{
\pspicture(4.2,2)
\rput(2,1){
\rput(0.3,0.3){$\times$}
\rput(0,0){$\times$}}
\endpspicture}

\newcommand{\ul}{\underline}
\newcommand{\ol}{\overline}
\newcommand{\wt}{\widetilde}

\newcommand{\hs}[1]{\hspace*{#1em}}

\newcommand{\tra}{{\psset{unit=0.1cm,nodesep=0pt} \pspicture(8,0)
\pcline{->}(1,1.1)(7,1.1) \endpspicture}}

\newcommand{\mtra}{{\psset{unit=0.1cm,nodesep=0pt} \pspicture(10,0)
\pcline{->}(1,1.1)(9,1.1) \endpspicture}}

\newcommand{\ltra}{{\psset{unit=0.1cm,nodesep=0pt} \pspicture(15,0)
\pcline{->}(1.5,1.4)(13.5,1.4) \endpspicture}}

\newcommand{\tramap}[1]{{\psset{unit=0.1cm,nodesep=0pt,labelsep=1pt} \pspicture(8,4)
\pcline{->}(1,1)(7,1)\naput[npos=0.45]{\ensuremath{\scriptstyle{#1}}} \endpspicture}}

\newcommand{\tmap}{\tramap}

\newcommand{\mtlmap}[1]{{\psset{unit=0.1cm,nodesep=0pt,labelsep=1pt} \pspicture(10,4)
\pcline{<-}(1,1)(9,1)\naput[npos=0.6]{\ensuremath{\scriptstyle{#1}}} \endpspicture}}

\newcommand{\mtmap}[1]{{\psset{unit=0.1cm,nodesep=0pt,labelsep=1pt} \pspicture(10,4)
\pcline{->}(1,1)(9,1)\naput[npos=0.45]{\ensuremath{\scriptstyle{#1}}} \endpspicture}}

\newcommand{\mmmmmtmap}[1]{{\psset{unit=0.1cm,nodesep=0pt,labelsep=1pt} \pspicture(12,4)
\pcline{->}(1,1)(11,1)\naput[npos=0.45]{\ensuremath{\scriptstyle{#1}}} \endpspicture}}

\newcommand{\ltramap}[1]{{\psset{unit=0.1cm,nodesep=0pt,labelsep=1pt} \pspicture(15,4)
\pcline{->}(1.5,1.1)(13.5,1.1)\naput{\ensuremath{\scriptstyle{#1}}} \endpspicture}}

\newcommand{\ltmap}{\ltramap}

\newcommand{\vltmap}[1]{{\psset{nodesep=0pt,labelsep=1pt} \pspicture(20,4)
\pcline{->}(1.5,1.1)(18.5,1.1)\naput{\ensuremath{\scriptstyle{#1}}} \endpspicture}}

\newcommand{\tmapsto}{{\psset{unit=0.1cm,nodesep=0pt} \pspicture(8,0) 
\pcline{|->}(1,1.2)(7,1.2) \endpspicture}}

\newcommand{\myps}{
\begin{small}
\pspicture
}

\newcommand{\emyps}{
\endpspicture
\end{small}
}

\usepackage{hyperref}

\newlength{\currentindent}
\newlength{\mpt}
\numberwithin{equation}{section}

\theoremstyle{plain}

\newtheorem{theorem}{Theorem}[section]

\newtheorem{proposition}[theorem]{Proposition}
\newtheorem{prop}[theorem]{Proposition}

\theoremstyle{definition}

\newtheorem{definition}[theorem]{Definition}

\newtheorem{example}[theorem]{Example}

\newtheorem{examples}[theorem]{Examples}
\newtheorem{nonexample}[theorem]{Non-example}
\newtheorem{remark}[theorem]{Remark}
\newtheorem{remarks}[theorem]{Remarks}
\newtheorem{exercise}[theorem]{Exercise}
\newtheorem{note}[theorem]{Note}
\newtheorem{question}[theorem]{Question}
\newtheorem{questions}[theorem]{Questions}
\newtheorem{algorithm}[theorem]{Algorithm}
\newtheorem{method}[theorem]{Method}
{ \end{sf}\end{framed}\end{minipage}
\end{center}}

\renewcommand{\arraystretch}{1.1} 

\psset{
unit=1mm,
linewidth=0.7pt,
arrowlength=1.1, 
arrowsize=3pt 2,
arrowinset=0.7,
doublesep=1pt,
labelsep=2pt,
nodesep=2pt, 
dash=2pt 2pt}
  
\begin{document}


\title{Weak vertical composition}

\author{Eugenia Cheng \\  School of the Art Institute of Chicago \\E-mail: info@eugeniacheng.com \\[12pt]
Alexander S. Corner\\
Sheffield Hallam University\\
E-mail: alex.corner@shu.ac.uk
}


\maketitle


\begin{abstract}
We study semi-strict tricategories in which the only weakness is in vertical composition.  We construct these as categories enriched in the category of bicategories with strict functors, with respect to the cartesian monoidal structure.  As these are a form of tricategory it follows that doubly-degenerate ones are braided monoidal categories.  We show that this form of semi-strict tricategory is weak enough to produce all braided monoidal categories. That is, given any braided monoidal category $B$ there is a doubly-degenerate ``vertically weak'' semi-strict tricategory whose associated braided monoidal category is braided monoidal equivalent to $B$. 
\end{abstract}


\setcounter{tocdepth}{2}
\tableofcontents


\section*{Introduction}
\addcontentsline{toc}{section}{Introduction}

It is well-known that every strict 2-category is equivalent to a weak one \cite{mp1}, but that the analogous result for 3-categories does not hold \cite{gps1}. Rather, coherence for weak 3-categories (tricategories) needs more nuance.  One way of viewing this is that we need to take account of possible braidings that arise and cannot be strictified into symmetries.  The original coherence result of Gordon--Power--Street \cite{gps1} says, essentially, that every tricategory is equivalent to one in which everything is strict except interchange.  The intuition is that ``braidings arise from weak interchange''.  However, from close observation of how the Eckmann--Hilton argument works, Simpson \cite{sim3} conjectured that weak units would be enough, and this result was proved for the case $n=3$ by Joyal and Kock \cite{jk1}.  Their result involves a weak unit $I$ in an otherwise completely strict monoidal 2-category.  They showed that the category $\cat{End}(I)$ of endomorphisms on $I$ is naturally a braided monoidal category, and that every braided monoidal category is equivalent to $\cat{End}(I)$ for some monoidal 2-category.  Regarding this as a (degenerate) 3-category, this means that everything in the 3-category is strict except horizontal units.

In this work we will address a third case, in which everything is strict except vertical composition, that is composition along bounding 1-cells; this amounts to considering categories strictly enriched in the category of bicategories and strict functors, with respect to cartesian product. We write this category as \bicats. On the one hand this might be regarded as a peculiar mixture of weakness and strictness, but as pointed out in \cite{lac5, lac6} the strict functors make for a much better behaved category---unlike the category involving weak functors, it is complete and cocomplete.  The category \bicats\ is further studied in \cite{bak1}.  Although these properties of \bicats\ are not our primary motivation for using strict functors, they may result in useful consequences of our main theorem.

Note that we need fully weak vertical composition, not just weak units; Kock \cite{koc3} proved that strict associativity in both the horizontal and vertical directions yields commutativity.

It follows from the result for general tricategories that any doubly-degenerate \bicats-category ``is'' naturally a braided monoidal category; that is, its single hom-category of 2-cells and 3-cells has the structure of a braided monoidal category with the monoidal structure given by vertical composition and braiding constructed from a weak \eh\ argument.  We will show that every braided monoidal category is equivalent to one arising in this way. The proof closely follows the idea of Joyal and Kock's, using clique constructions.  Joyal and Kock use train track diagrams to give just enough ``rigidity'' to the structure of points in 3-space, and they describe this as preventing the points from being able to simply commute past each other via an Eckmann--Hilton argument. We are aiming for a different axis of strictness and so instead of points in $\R^2$ with cliques arising from train track diagrams, we use use points embedded in the interior of $I^2$ (where $I$ denotes the unit interval) with cliques arising from horizontal reparametrisations.  The idea is that the fundamental groupoid of the configuration space of points in $I^2$ is naturally a doubly-degenerate tricategory with weak horizontal and vertical composition but strict interchange.  We can make horizontal composition strict by imposing an equivalence relation, but we follow Joyal and Kock in implementing this using cliques.  

Then, starting with a braided monoidal category $B$ we show how to construct a doubly-degenerate ``vertically weak'' tricategory $\Sigma B$ whose associated braided monoidal category is braided monoidal equivalent to $B$.  Analogously to \cite{jk1} the 2-cells will be configurations of points in $I^2$ labelled by objects of $B$, such as
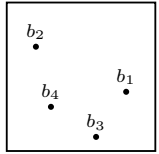
\[\psset{unit=2mm}
\pspicture(-5,-5)(5,5)

\psframe(-5,-5)(5,5)

\rput(3,-1){
\pscircle*(0,0){0.2}
\rput(0,1){\scr $b_1$}
}

\rput(-3,2){
\pscircle*(0,0){0.2}
\rput(0,1){\scr $b_2$}
}

\rput(1,-4){
\pscircle*(0,0){0.2}
\rput(0,1){\scr $b_3$}
}

\rput(-2,-2){
\pscircle*(0,0){0.2}
\rput(0,1){\scr $b_4$}
}


\endpspicture\]
subject to the appropriate horizontal equivalence relation. Then vertical composition is given by stacking the boxes vertically and scaling them equally; as for concatenation of paths in a space, this is only weakly associative.  Horizontal composition is strict because of the equivalence relation.

As in \cite{jk1} the morphisms (that is, 3-cells of the doubly-degenerate tricategory) are constructed via cliques.  The idea is that we want to interpret the configuration of labelled points as a tensor product in $B$, and then take morphisms between those objects in $B$, but there is no consistent way to do that. We can interpret a single point labelled by $a$ as the object $a$:
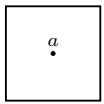
\[\renewcommand{\nowdot}{\psset{unit=1.1mm}\pscircle*(0,0){0.3}}
\psset{unit=0.8mm}
\ps(0,0)(16,16)

\rput(0,0){

\psframe(0,0)(16,16)
\pnode(8,8){x1}

\rput(8,8){\rnode{a1}{\nowdot}
\rput(0,2){\scr $a$}
}

}

\eps\]

\noi and a vertical pair as shown below is then ``obviously'' a vertical tensor product of the two singletons
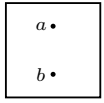
\[\renewcommand{\nowdot}{\psset{unit=1.1mm}\pscircle*(0,0){0.3}}
\psset{unit=0.8mm}
\ps(0,0)(16,16)

\rput(0,0){

\psframe(0,0)(16,16)
\pnode(8,8){x1}

\rput(8,12){\rnode{a1}{\nowdot}
\rput(-2,0){\scr $a$}
}

\rput(8,4){\rnode{b1}{\nowdot}
\rput(-2,0){\scr $b$}
}

}

\eps\]
so it can be interpreted as $a \otimes b$.  However with more than two objects it is unclear what parenthesisation we should take.  Worse, this configuration
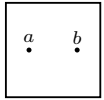
\[\renewcommand{\nowdot}{\psset{unit=1.1mm}\pscircle*(0,0){0.3}}
\psset{unit=0.8mm}
\ps(0,0)(16,16)

\rput(0,0){

\psframe(0,0)(16,16)
\pnode(8,8){x1}

\rput(4,8){\rnode{a1}{\nowdot}
\rput(0,2){\scr $a$}
}

\rput(12,8){\rnode{b1}{\nowdot}
\rput(0,2){\scr $b$}
}

}

\eps\]
appears to be a horizontal composite in the doubly-degenerate tricategory, but it is then not clear if we should interpret it as $a \otimes b$ or $b \otimes a$ in $B$.  Whichever choice we make for constructing 3-cells we will not get strict interchange -- interchange will need to invoke the braiding in $B$.   

We address all these issues by following \cite{jk1} and using cliques.  First we consider the free braided monoidal category on the objects of $B$, and embed that as configurations of points in $I^2$ entirely on the central vertical line.  Then, for a general labelled configuration of points, instead of picking one interpretation as a tensor product in $B$, we consider the clique of all such interpretations; moreoever, each interpretation must be equipped with a braid recording a path from the configuration in $I^2$ to the vertical configuration of points corresponding to that particular tensor product in $B$.  Clique maps are then those maps in $B$ that are just coherence maps (and braidings) commuting with the ``linearising'' braids.  We then take clique maps between those as the 3-cells in our doubly-degenerate tricategory.  The non-strictness of the interchange is then absorbed into the cliques.

Joyal and Kock are starting from a different framework of train tracks, but once we set up our framework to replace the train tracks the rest of the construction and proof is very similar.  The main construction and result of the paper are then as follows.

\subsubsection*{Main construction}

Given a braided monoidal category $B$ we define a doubly-degenerate vertically weak tricategory $\Sigma B$ whose 2-cells are certain cliques of configurations of points of $I^2$ labelled by objects of $B$, and whose 3-cells are pulled back from $B$ via a clique construction.

\subsubsection*{Main theorem}

Given any braided monoidal category $B$, the underlying braided monoidal category of $\Sigma B$ is braided monoidal equivalent to $B$.

\bs

There are several critical subtleties to this, which is why we have to leave vertical associativity weak but can use fully doubly-degenerate structures, where Joyal and Kock were able to have all associativity strict but could not use fully doubly-degenerate structures.  In future work we will extend the result to totalities, exhibiting a biequivalence of appropriate bicategories.   This method generalises, by omitting the ``slide'' cliques, to prove the corresponding result for doubly degenerate Trimble 3-categories \cite{tri1}; this will also be in future work.  In fact, our original aim was to make the construction for Trimble 3-categories, but in doing so we realised that we could alter the construction slightly to make it work for the vertically weak tricategory case.  

The structure of the paper is as follows.  In Section~\ref{one} we define the type of semi-strict tricategories we will be studying, which are categories enriched in the category of bicategories and strict functors. We will characterise the doubly-degenerate ones as certain categories with two monoidal structures, one weak and one strict, satisfying strict interchange, and show that each one has an associated braided monoidal category. Note that these categories with two monoidal structures are a particularly strict form of 2-monoidal category \cite{am1}. In general in a 2-monoidal category interchange is only required to be lax; it is known that braided monoidal categories can be built from 2-monoidal categories where interchange is weak, but in this work we deal with the case where interchange is strict.

In Section~\ref{cliques} we give some background on cliques that will be needed for the main construction and proof; none of this section is new.  In Section~\ref{three} we introduce the labelled configuration spaces of points that we will use for the main construction, together with their braided monoidal structure. The new content in this section is the definition of slide cliques, which we use to ensure that our horizontal tensor product is strict. In Section~\ref{mainconstruction} we give the main construction, which starts with any braided monoidal category $B$ and produces a doubly-degenerate \bicats-category $\Sigma B$ from it. In Section~\ref{maintheorem} we prove the main theorem, which is that the braided monoidal category associated with $\Sigma B$ is braided monoidal equivalent to $B$, showing that all braided monoidal categories arise from doubly-degenerate \bicats-categories.  In Section~\ref{future} we give a brief account of future work.

\subsubsection*{How to read this paper quickly}

To read this paper quickly, experts may proceed by reading as follows:

\begin{enumerate}

\item The characterisation of doubly-degenerate \bicats-categories in 
Proposition~\ref{dd-char}.

\item The definition of slide cliques in Section~\ref{slidemaps}.

\item The construction of $\Sigma B$: the underlying category (Section~\ref{sigmacategory}), the tensor products (Section~\ref{sigmatensor}) and interchange (Section~\ref{sigmainterchange}).

\item The main theorem, stated as Theorem~\ref{theoremmain}.

\end{enumerate}

\subsubsection*{Terminology conventions}

We will use the general terminology convention where ``strict'' means that coherence constraints are identities and ``weak'' means they are isomorphisms (which are called strong by some authors).  All our braided monoidal categories have a weak monoidal structure.


\section{Doubly-degenerate \bicats-categories}
\label{one}

In this section we will set up the framework of the semi-strict tricategories we will be studying.  First we perform a dimension shift to characterise them as categories with two monoidal structures, and we then define the underlying braided monoidal category of such a structure.

\subsection{\bicats-categories}

We write \bicats\ for the category of bicategories and strict functors between them. \bicats\ has finite products; we will study categories enriched in \bicats\ with respect to the cartesian monoidal structure. (We will not consider any other monoidal structure on \bicats.)  Note that we are using standard enrichment in a monoidal category, not any kind of weak enrichment. We will refer to composition along bounding $k$-cells as $k$-composition, but we will also refer to 0-composition as horizontal, and 1-composition as vertical.

Thus a \bicats-category is a form of semi-strict tricategory with
\begin{itemize}
\item a set of 0-cells, and hom-bicategories giving 1-cells, 2-cells and 3-cells,
\item 2-composition coming from the vertical composition in the hom-bicategories, thus it is strictly associative and unital,
\item 1-composition coming from the horizontal composition in the hom-bicategories, thus it is weakly associative and weakly unital,
\item 0-composition coming from the composition in the enriched category, thus it is strictly associative and strictly unital,
\item interchange between 1-composition and 2-composition is strict as it comes from interchange in the hom-bicategories,
\item interchange between 0-composition and any other kind is strict as it comes from the functoriality of morphisms in \bicats, which is strict.
\end{itemize}
We will also refer to these as ``vertically weak tricategories'', by which we mean that the only weakness is in the vertical composition. 

\subsection{Doubly-degenerate \bicats-categories by dimension shift}

A doubly-degenerate \bicats-category has only one 0-cell and only one 1-cell.  As usual we perform a ``dimension shift'' and look at the 2-cells and 3-cells as a category with extra structure.  In this case we have a category with two monoidal structures:

\begin{itemize}

\item a strict ``horizontal'' monoidal structure, coming from 0-composition of the original tricategory, which we will write as $a|b$, and

\item a weak ``vertical'' monoidal structure, coming from 1-composition, which we will write as $\frac{a}{b}$.

\end{itemize}

\noi These satisfy strict interchange, which with this notation can then be written as 
\[\setlength{\arraycolsep}{2.7pt}
\renewcommand{\arraystretch}{1.4}
\begin{array}{ccc}
a & \psline(0,-0.8)(0,3.3) & b \\
\hline 
c & \psline(0,-1)(0,2.8) & d
\end{array}
\hs{1}
= \hs{1}
\setlength{\arraycolsep}{5pt}
\begin{array}{c|c}
a & b \\[-8pt]
\psline(-1.6,0.5)(1.4,0.5) & \psline(-1.4,0.5)(1.6,0.5) \\[-8pt]
c & d 
\end{array}
\]

\noi In general there are three types of weakness in a tricategory:

\begin{itemize}
\item 0-composition (horizontal),
\item 1-composition (vertical), and
\item interchange between them.
\end{itemize}

\noi but in the semi-strict structures we are studying only vertical composition is weak. The following characterisation is straightforward but it will help later for us to make it explicit.

\begin{prop}\label{dd-char}
A doubly degenerate \bicats-category is precisely a category equipped with
two monoidal structures, one weak and one strict, satisfying strict interchange.
\end{prop}

\subsection{Underlying braided monoidal category}

Classical coherence for tricategories \cite{gps1} shows that we can make everything strict except interchange.  Joyal and Kock \cite{jk1} study making everything strict except horizontal composition (in fact everything is strict except horizontal units).  We are studying the remaining case, where everything is strict except vertical composition.  We will show that this situation is ``weak enough'' to produce braided monoidal categories in the doubly-degenerate case. 

We will follow the methods and techniques of \cite{jk1} quite closely.  So first, for any doubly-degenerate \bicats-category $A$ we exhibit its associated braided monoidal category $UA$.  $A$ has an underlying category of 2-cells and 3-cells which by abuse of notation we also write as $A$.  Since it is, among other things, a form of doubly-degenerate tricategory, it becomes a braided monoidal category with respect to the monoidal structure coming from 1-composition, which in this case is a weak monoidal structure.  The braiding then comes from the Eckmann--Hilton argument in the usual way (see for example \cite{cg3}), but simplified by the horizontal composition being strict. Note that there are two possible choices of orientation for the braiding, depending on which way round the left and right units are applied. We will pick the following convention:


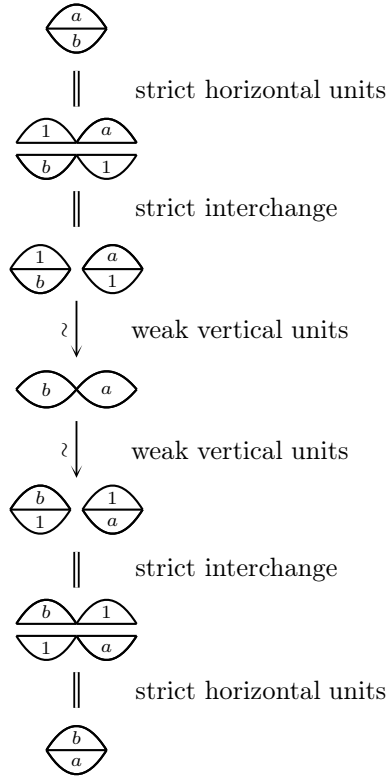
\[\ps(0,-20)(30,90)

\rput(10,0){
\psset{unit=0.4mm,arrowsize=1.5pt 
2,arrowlength=1.4,linewidth=0.8pt,labelsep=2pt, nodesep=0pt}
\pspicture(0,0)(40,60)


%
%
%

\rput(0,240){\rnode{a6}{\ps(0,0)
\rput(-10,0){
\parabola[linecolor=black](0,0)(10,-8)
\parabola(0,0)(10,8)

\parabola[linecolor=black](0,0)(10,8)
\parabola(0,0)(10,-8)
\psline(0,0)(20,0)

\rput(10,4){\scr $a$}
\rput(10,-4){\scr $b$}}
\eps}}

\rput(0,200){\rnode{a7}{\ps(0,0)
\rput(-20,0){\parabola[linecolor=black](0,-2)(10,-10)
\parabola(0,2)(10,10)
\parabola(0,-2)(10,-10)
\parabola(20,2)(30,10)

\rput(0,2){\rput(10,4){\scr $1$}}
\rput(0,-2){\rput(10,-4){\scr $b$}}

\parabola[linecolor=black](20,2)(30,10)
\parabola(20,-2)(30,-10)

\rput(20,2){\rput(10,4){\scr $a$}}
\rput(20,-2){\rput(10,-4){\scr $1$}}

\psline(0,2)(20,2)
\psline(20,2)(40,2)

\psline(0,-2)(20,-2)
\psline(20,-2)(40,-2)}
\eps}}

\rput(0,160){\rnode{a8}{\ps(0,0)
\rput(-22,0){
\parabola[linecolor=black](0,0)(10,-8)
\parabola(0,0)(10,8)
\parabola(0,0)(10,-8)
\psline(0,0)(20,0)

\rput(0,0){\rput(10,4){\scr $1$}}
\rput(0,0){\rput(10,-4){\scr $b$}}

\rput(24,0){ 
\parabola(0,0)(10,8)

\parabola[linecolor=black](0,0)(10,8)
\parabola(0,0)(10,-8)
\psline(0,0)(20,0)

\rput(0,0){\rput(10,4){\scr $a$}}
\rput(0,0){\rput(10,-4){\scr $1$}}
}}\eps}}

\rput(0,120){\rnode{a9}{\ps(0,0)
\rput(-20,0){
\parabola[linecolor=black](0,0)(10,6)
\parabola(0,0)(10,6)

\parabola[linecolor=black](0,0)(10,-6)
\parabola(0,0)(10,-6)

\parabola[linecolor=black](20,0)(30,6)
\parabola(20,0)(30,6)

\parabola[linecolor=black](20,0)(30,-6)
\parabola(20,0)(30,-6)

\rput(0,0){\rput(10,0){\scr $b$}}
\rput(20,0){\rput(10,0){\scr $a$}}}\eps}}

%


\rput(0,80){\rnode{a10}{\ps(0,0)
\rput(-22,0){
\parabola[linecolor=black](0,0)(10,8)
\parabola(0,0)(10,8)
\parabola(0,0)(10,-8)
\psline(0,0)(20,0)

\rput(0,0){\rput(10,4){\scr $b$}}
\rput(0,0){\rput(10,-4){\scr $1$}}

\rput(24,0){
\parabola(0,0)(10,8)

\parabola[linecolor=black](0,0)(10,-8)
\parabola(0,0)(10,-8)
\psline(0,0)(20,0)

\rput(0,0){\rput(10,4){\scr $1$}}
\rput(0,0){\rput(10,-4){\scr $a$}}}}
\eps}}


%

\rput(0,40){\rnode{a11}{\rput(-20,0){
\parabola[linecolor=black](0,2)(10,10)
\parabola(0,2)(10,10)
\parabola(0,-2)(10,-10)
\parabola(20,2)(30,10)

\rput(0,2){\rput(10,4){\scr $b$}}
\rput(0,-2){\rput(10,-4){\scr $1$}}

\parabola[linecolor=black](20,-2)(30,-10)
\parabola(20,-2)(30,-10)
\rput(20,2){\rput(10,4){\scr $1$}}
\rput(20,-2){\rput(10,-4){\scr $a$}}

\psline(0,2)(20,2)
\psline(20,2)(40,2)

\psline(0,-2)(20,-2)
\psline(20,-2)(40,-2)}}}

%


\rput(0,0){\rnode{a12}{\rput(-10,0){
\parabola[linecolor=black](0,0)(10,8)
\parabola(0,0)(10,8)

\rput(0,0){\rput(10,4){\scr $b$}}
\rput(0,0){\rput(10,-4){\scr $a$}}

\parabola[linecolor=black](0,0)(10,-8)
\parabola(0,0)(10,-8)
\psline(0,0)(20,0)}}}

\endpspicture
}


\psset{nodesep=16pt, labelsep=20pt}

\ncline[doubleline=true,nodesep=16pt]{-}{a6}{a7} \naput{strict horizontal units}
\ncline[doubleline=true, arrows=-]{a7}{a8} \naput{strict interchange}
\ncline[nodesep=12pt]{->}{a8}{a9} \nbput[labelsep=2pt]{\rotatebox{90}{\scr $\sim$}} \naput{weak vertical units}
\ncline[nodesep=12pt]{->}{a9}{a10} \nbput[labelsep=2pt]{\rotatebox{90}{\scr $\sim$}} \naput{weak vertical units}
\ncline[doubleline=true, arrows=-]{a10}{a11} \naput{strict interchange}
\ncline[doubleline=true,nodesep=16pt]{-}{a11}{a12} \naput{strict horizontal units}

\endpspicture\]

The question then is whether all braided monoidal categories arise from a \ddbicatscat\ in this way. We follow Joyal and Kock \cite{jk1} and show that given any \bmc\ $B$ there is a \dd\ \bicats-category $\Sigma B$ such that $U\Sigma B$ is braided monoidal equivalent to $B$. This is evidently only an object-level correspondence, rather than dealing with the totalities of such structures, but hints at $U$ and $\Sigma$ giving an equivalence between the totalities; a full equivalence requires further theory so we defer it to a future work.


\section{Preliminaries on cliques}
\label{cliques}

Our construction follows that of \cite{jk1} in using cliques, so we include the definitions we need here; none of this is new.  As we have found the existing literature to be somewhat piecemeal (following historical developments), we judge that there is broader benefit to some exposition here.

\subsection{Cliques and their morphisms}

Cliques were mentioned briefly in \cite{js2}, referred to again with a little more development in \cite{js1} and then the theory was developed further in \cite{jk1}; they are also studied in \cite{mak3} under the name ana-objects.  Essentially, a clique is a more categorically satisfactory way of implementing equivalence relations in which certain isomorphic objects are to be considered ``the same''.  Instead of identifying the objects in question, we assemble them into cliques, which are, essentially, collections of objects with uniquely specified isomorphisms between them.  The cliques are then considered to be objects of a new category, so that the entire clique is now considered to be a single object, without us having to quotient out by anything.  In a sense, the power of the theory of cliques comes from the way morphisms are handled, which is typically a very unsatisfactory aspect of performing quotients on a category.

So a clique in a category is essentially a family of objects equipped with uniquely specified isomorphisms between them.  The subtlety is that objects might be repeated, hence we use an indexing category to specify the objects of a clique.  Sometimes that subtlety is not needed so it is not so crucial to emphasise the indexing category, hence in the earlier definitions (for example \cite{js2}) the indexing category is not emphasised, but in later works (for example \cite{jk1}) the indexing category is prominent.  The definition in \cite{js2} is as follows, and this is often how we will think of cliques in practice.

\begin{definition}
A \emph{clique} in a category $\cC$ is a non-empty family 
\[\{x_j \ | \ j \in J\}\] 
of objects of $\cC$ together with a family
\[\{ x_{jk} \: x_j \lra x_k \ | \ (j, k) \in J \times J \}\]
of maps such that $x_{jj} = 1$ and $x_{jk}x_{kl} = x_{jl}$ (so in particular $x_{jk} = x_{kj}^{-1})$).  The isomorphisms $x_{jk}$ are called \emph{connecting isomorphisms}.

A morphism of cliques 
\[ f\: \{x_j \ | \ j \in J \} \ltra \{y_k \ | \ k \in K \}\]
is a family of maps
\[f_{jk} \: x_j \lra y_k\]
such that the following diagram commutes for all $(j,j') \in J^2$, $(k,k') \in K^2$:
\[
\psset{unit=0.1cm,labelsep=3pt,nodesep=3pt}
\pspicture(-10,-10)(10,10)



\rput(-10,8){\rnode{a1}{$x_j$}}  
\rput(10,8){\rnode{a2}{$y_k$}}  
\rput(-10,-8){\rnode{a3}{$x_{j'}$}}  
\rput(10,-8){\rnode{a4}{$y_{k'}$}}  

\ncline{->}{a1}{a2} \naput{{\scriptsize $f_{jk}$}} 
\ncline{->}{a3}{a4} \nbput{{\scriptsize $f_{j'k'}$}} 
\ncline{->}{a1}{a3} \nbput{{\scriptsize $u_{jj'}$}} 
\ncline{->}{a2}{a4} \naput{{\scriptsize $u_{kk'}$}} 

\endpspicture
\]

\end{definition}

The following definitions of clique and clique morphism are equivalent to the above but are expressed more abstractly; these definitions are given in \cite{jk1} and the greater reference to the indexing category facilitates the more advanced constructions.  The idea is to define a clique as a functor from an indiscrete category.

%
%
%



\begin{definition}
Given a set (or collection) $J$ let $\ol{J}$ denote the groupoid whose object set is $J$ and whose morphism set is $J \times J$, with the two projections giving source and target.  If $J$ is non-empty then $\ol{J}$ is contractible.  A \emph{clique} in a category $\cC$ is a functor $\ol{J} \tra \cC$ for some non-empty $J$.  
\end{definition}

%
%


\begin{remark} \label{repeatobjects}
A collection of objects in $\cC$ together with unique specified isomorphisms between them is a clique, but in general a clique may contain multiple copies of the same object.   In the former case the indexing category can be suppressed (as it is essentially the same as the objects of the clique itself), but in the latter case it cannot.  The slide cliques we use will be of the former type, but the cliques we use in the main construction will crucially need to be the latter.
\end{remark}

\begin{definition}
A \emph{morphism of cliques} $\{x_j \ | \ j \in J\} \tra \{y_k \ | \ k \in K\}$ is a natural transformation

\[\psset{unit=0.08cm,labelsep=0pt,nodesep=3pt}
\pspicture(20,22)


\rput[B](0,20){\Rnode{a1}{$\ol{J \times K}$}} 
\rput[B](20,20){\Rnode{a2}{$\ol{K}$}} 

\rput[B](0,0){\Rnode{b1}{$\ol{J}$}}   
\rput[B](20,0){\Rnode{b2}{$\cC$}}  

\psset{nodesep=3pt,labelsep=2pt,arrows=->}
\ncline{a1}{a2}\naput{{\scriptsize $$}} 
\ncline{b1}{b2}\nbput{{\scriptsize $x$}} 
\ncline{a1}{b1}\nbput{{\scriptsize $$}} 
\ncline{a2}{b2}\naput{{\scriptsize $y$}} 

\rput(10.5,11.5){
\psset{labelsep=1.5pt,nodesep=0pt}
\pnode(-3,-3){c1}
\pnode(3,3){c2}
\ncline[doubleline=true,arrowinset=0.6,arrowlength=0.8,arrowsize=0.5pt 2.1]{c1}{c2} \nbput[npos=0.4]{{\scriptsize $$}}
}

\endpspicture\]
%
%
%
%
%
%
%
\end{definition}

\begin{remark}
Note that a clique morphism is completely determined by specifying any one of its components, and moreover any morphism $x_j \tra y_k$ (for any $j$ and $k$) will specify a clique map.  The subtlety is that clique maps specified by

\[\ps(-10,-2)(10,6)

\rput[B](-8,6){\Rnode{a1}{$x_j$}}
\rput[B](8,6){\Rnode{a2}{$y_k$}}
\rput[B](-8,0){\Rnode{b1}{$x_{j'}$}}
\rput[B](8,0){\Rnode{b2}{$y_{k'}$}}

\ncline{->}{a1}{a2} \naput{\scr $f$}
\ncline{->}{b1}{b2} \naput{\scr $f'$}

\eps\]
represent the same clique map whenever the square involving connecting isomorphisms commutes.  In practice this is typically how we will specify clique maps (by a single component). 
\end{remark}

\begin{definition}
We write $\tilde{\cC}$ for the category of cliques in $\cC$ and clique morphisms.
\end{definition}

For composition, note that when clique maps are specified by a single component they must be composed via the appropriate connecting isomorphisms.  That is, given cliques
\[\{x_j \ | \ j \in J\}, \hs{0.5} \{y_k \ | \ k \in K\}, \hs{0.5} \{z_l \ | \ l \in L\}\]
and clique maps specified by the following components
\[x_j \tmap{f} y_k, \hs{0.5} y_{k'} \tmap{g} z_l\]
the composite may be given by the following component:
\[
\psset{unit=0.1cm,labelsep=2pt,nodesep=3pt}
\pspicture(0,-1)(44,18)


\rput[B](0,15){\Rnode{a1}{$x_j$}}  
\rput[B](20,15){\Rnode{a2}{$y_k$}}  

\rput[B](20,0){\Rnode{b2}{$y_{k'}$}}  
\rput[B](40,0){\Rnode{b3}{$z_l$}}  

\ncline{->}{a1}{a2} \naput{{\scriptsize $f$}} 

\ncline{->}{b2}{b3} \nbput{{\scriptsize $g$}} 

\ncline{->}{a2}{b2} \naput{{\scriptsize $y_{kk'}$}} \nbput{\rotatebox{90}{\scr $\sim$}}

\endpspicture
\]

\noi as seen by the composite naturality square:

\[
\psset{unit=0.1cm,labelsep=2pt,nodesep=3pt}
\pspicture(0,-3)(44,19)


\rput[B](0,16){\Rnode{a1}{$x_j$}}  
\rput[B](20,16){\Rnode{a2}{$y_k$}}  
\rput[B](40,16){\Rnode{a3}{$z_l$}}  

\rput[B](0,0){\Rnode{b1}{$x_j$}}  
\rput[B](20,0){\Rnode{b2}{$y_{k'}$}}  
\rput[B](40,0){\Rnode{b3}{$z_l$}}  

\ncline{->}{a1}{a2} \naput{{\scriptsize $f$}} 
\ncline{->}{b1}{b2} \nbput{{\scriptsize $y_{kk'} \circ f$}} 

\ncline{->}{a2}{a3} \naput{{\scriptsize $g\circ y_{kk'}$}} 
\ncline{->}{b2}{b3} \nbput{{\scriptsize $g$}} 

\ncline{->}{a1}{b1} \nbput{{\scriptsize $x_{jj} = 1$}} \naput{\rotatebox{90}{\scr $\sim$}} 
\ncline{->}{a2}{b2} \naput{{\scriptsize $y_{kk'}$}} \nbput{\rotatebox{90}{\scr $\sim$}} 
\ncline{->}{a3}{b3} \naput{{\scriptsize $z_{ll}=1$}} \nbput{\rotatebox{90}{\scr $\sim$}} 

\endpspicture
\]

Note that this means we might have two clique maps specified by identity morphisms, but their composite is not specified by an identity because we have picked up a non-trivial connecting isomorphism in the middle.

\begin{remarks}\label{cliqueremarks}
Some observations about identities and isomorphisms in the category of cliques will help us later.  First note that the identity clique map 
\[ \{x_j \ |\ j \in J\} \ltra \{x_j \ |\ j \in J\} \]
has as its components the connecting isomorphisms, thus it may be represented by an identity component on any object of the clique
\[ x_j \mmmmmtmap{1_{x_j}} x_j \]
or, more generally, if objects $x_j$ and $x_k$ are used to represent the source and target, then the identity clique map is represented by the connecting isomorphism
\[ x_j \mmmmmtmap{x_{jk}} x_k.\]
Thus an identity clique map can be specified by a non-identity morphism; conversely a non-identity clique map can be specified by an identity morphism, as it might be a morphism between different cliques that happen to have an object in common.

Also note that an isomorphism in the category of cliques is a clique morphism each of whose components is an isomorphism; it is sufficient for any individual component to be an isomorphism.  In particular, cliques that are isomorphic in the category of cliques do not have to have isomorphic indexing categories.

\end{remarks}


\noi We will not need the following result but include it for completeness.

\begin{proposition}{\cite{jk1}}
There is a canonical equivalence of categories
\[\cC \tra \wt{\cC}\]
given by sending an object $x \in \cC$ to the singleton clique on the object $x$.
\end{proposition}

\subsection{Cliques in braided monoidal categories}

Eventually we will be taking cliques in a braided monoidal category, and will need to know that a canonical braided monoidal structure is induced on the category of cliques.  The structure is exactly as expected, with only the notation being slightly cumbersome; in practice we will mostly suppress the notation.

If $(\cC, \otimes, I)$ is a monoidal category, then there is a canonical monoidal structure on the category category of cliques $\wt{\cC}$, given in \cite{jk1}.  Its tensor product $\wt{\otimes}$ is defined pointwise. That is, given cliques $x$ and $y$ indexed by $J$ and $K$ respectively, the tensor product $x \wt{\otimes} y$ is indexed by $J \times K$ and is defined as follows:

\begin{itemize}

\item objects: $(x \wt{\otimes} y)_{j,k} := x_j \otimes y_k$
\item connecting isomorphisms: $(x \wt{\otimes} y)_{(j,k)(p,q)} := x_{jp} \otimes y_{kq}$

\end{itemize}

\noi The unit $\wt{I}$ is the singleton clique $\ast \tra I$.  Note that even if $\cC$ is a strict monoidal category, the monoidal structure induced on $\wt{\cC}$ will be weak as it involves the weakness of cartesian products of the indexing sets.

Furthermore, if $\cC$ has a braiding then there is an induced braiding on $\wt{\cC}$; the braiding
\[x \wt{\otimes} y \mtra y \wt{\otimes} x\]
has components given by the braiding in $\cC$
\[x_j \otimes y_k \mtra y_k \otimes x_j.\]

\subsection{Induced functors between clique categories}
\label{inducedfunctors}

As in \cite{jk1} we will need two constructions on functors. The first is the lowershriek, which can also be thought of as a direct image.

\begin{definition}(Lowershriek) Given a functor
\[F \: \cC \tra \cD\]
we define a functor on clique categories
\[F_! \: \wt{\cC} \tra \wt{\cD}\]
as follows.  Given a clique $\ol{J} \tmap{x} \cC$, $F_!x$ is the clique
\[\ol{J} \tmap{x} \cC \tmap{F} \cD\]
This extends to a functor in the obvious way.
\end{definition}

The next construction is of inverse image cliques and takes more technical build-up. The idea is that given a clique in $\cD$ we can attempt to take its ``essential pre-image'' under a functor $\cC \tmap{F} \cD$, which might mean all the objects in $\cC$ whose image under $F$ is isomorphic to an object in the clique in question in $\cD$.  However, to make this algebraic we should record the isomorphisms in question.  To ensure this essential pre-image is non-empty we need $F$ to be essentially surjective, and to ensure that the resulting structure is a clique we need $F$ to be full and faithful.  To express the construction precisely we make use of the 2-fibred product.
\[\psset{unit=0.08cm,labelsep=0pt,nodesep=3pt}
\pspicture(0,1)(20,23)


\rput[B](0,20){\Rnode{a1}{$\cC \dtimes_{\hs{-0.2}\cD} \ol{J}$}} 
\rput[B](20,20){\Rnode{a2}{$\ol{J}$}} 

\rput[B](0,0){\Rnode{b1}{$\cC$}}   
\rput[B](20,0){\Rnode{b2}{$\cD$}}  

\psset{nodesep=3pt,labelsep=2pt,arrows=->}
\ncline{a1}{a2}\naput{{\scriptsize $$}} 
\ncline{b1}{b2}\nbput{{\scriptsize $F$}} 
\ncline{a1}{b1}\nbput{{\scriptsize $$}} 
\ncline{a2}{b2}\naput{{\scriptsize $x$}} 

\rput(10.5,10.5){
\psset{labelsep=1.5pt,nodesep=0pt}
\pnode(3,3){c1}
\pnode(-3,-3){c2}
\ncline[doubleline=true,arrowinset=0.6,arrowlength=0.8,arrowsize=0.5pt 2.1]{c1}{c2} \nbput[npos=0.4]{{\scriptsize $\sim$}}
}

\endpspicture\]
which we will elucidate after the definition.

\begin{definition}(Inverse image clique \cite{jk1})
Given an equivalence of categories $F: \cC \tmap{\sim} \cD$ we define a functor on clique categories
\[F^*: \wt{\cD} \mtmap{\sim} \wt{\cC}\]
as follows.  Suppose $x \: \ol{J} \tra \cD$ is a clique.  Then the \emph{inverse image clique} $F^*x$ is the clique:
\[\cC \dtimes_{\hs{-0.2}\cD} \hs{0.1}\ol{J} \mtra \cC\]
where $\cC \dtimes_{\hs{-0.2}\cD} \hs{0.1}\ol{J}$ denotes a 2-fibred product and the functor above is the projection to $\cC$. This extends to a functor $F^*$ which is also an equivalence of categories.
\end{definition}

\begin{remarks}

As in \cite{jk1} we will take the objects of $\cC \dtimes_{\hs{-0.2}\cD} \ol{J}$ to be triples $(c, j, \gamma)$ where

\begin{itemize}
\item $c\in \cC$, 
\item $j \in J$,  and
\item $\gamma: x_j \tmap{\sim} Fc \in \cD$.
\end{itemize}
(In fact in \cite{jk1} $\gamma$ is given in the inverse direction, but we will use this direction as it better matches our intuition for the braid diagrams we will draw.)

A morphism 
\[(c_0, j_0, \gamma_0) \tra (c_1, j_1, \gamma_1) \ \in \cC \dtimes_{\hs{-0.2}\cD} \ol{J}\]
is then a morphism 
\[c_0 \tmap{w} c_1 \in \cC\]
such that the following diagram commutes:

\[
\psset{unit=0.1cm,labelsep=3pt,nodesep=3pt}
\pspicture(0,-4)(20,22)



\rput(20,18){\rnode{a1}{$Fc_0$}}  
\rput(0,18){\rnode{a2}{$x_{j_0}$}}  
\rput(20,0){\rnode{a3}{$Fc_1$}}  
\rput(0,0){\rnode{a4}{$x_{j_1}$}}  

\ncline{->}{a2}{a1} \naput{{\scriptsize $\gamma_0$}} 
\nbput[labelsep=2pt]{\scr $\sim$}
\ncline{->}{a4}{a3} \nbput{{\scriptsize $\gamma_1$}} 
\naput[labelsep=1pt]{\scr $\sim$}
\ncline{->}{a1}{a3} \naput{{\scriptsize $Fw$}} 
\ncline{->}{a2}{a4} \nbput{{\scriptsize $x_{j_0j_1}$}} 
\naput{\rotatebox{90}{\scr $\sim$}}
\endpspicture
\]

%
%
%
%
%
%
%
Note that $F$ being essentially surjective ensures that this 2-fibred product is non-empty, and $F$ being full and faithful ensures that it is contractible, so $F^*x$ is indeed a clique.
\end{remarks}

An example may help at this point.

\begin{example}
Take $J = 1$ so $x:\ol{J} \lra \cD$ is a singleton clique, on an object $d$, say.  As $F$ is essentially surjective we know there exists an object $c$ with an isomorphism $d \tmap{\gamma} Fc$.

Now $F^*x$ is indexed by $\cC \dtimes_{\hs{-0.2}\cD} 1$ whose objects are all the pairs $(c, \gamma)$ with
\[\gamma\: d \tmap{\sim} Fc\]
Given any two such objects $(c, \gamma)$ and $(c', \gamma')$ we have a composite
\[Fc \mtmap{\gamma^{-1}} d \mtmap{\gamma'} Fc'\]
and as $F$ is full and faithful we know there is a unique morphism $c \tmap{w} c'$ making the following diagram commute
\[
\psset{unit=0.1cm,labelsep=3pt,nodesep=3pt}
\pspicture(0,-3)(20,21)



\rput(20,18){\rnode{a1}{$Fc$}}  
\rput(0,18){\rnode{a2}{$d$}}  
\rput(20,0){\rnode{a3}{$Fc'$}}  
\rput(0,0){\rnode{a4}{$d$}}  

\ncline{->}{a2}{a1} \naput{{\scriptsize $\gamma$}} 
\nbput[labelsep=2pt]{\scr $\sim$}
\ncline{->}{a4}{a3} \nbput{{\scriptsize $\gamma'$}} 
\naput[labelsep=1pt]{\scr $\sim$}
\ncline{->}{a1}{a3} \naput{{\scriptsize $Fw$}} 
\ncline{->}{a2}{a4} \nbput{{\scriptsize $1$}} 
\naput{\rotatebox{90}{\scr $\sim$}}
\endpspicture
\]
So $\cC \dtimes_{\hs{-0.2}\cD} 1$ is indeed non-empty, with a unique isomorphism between any pair of objects.  The clique $F^*x$ is obtained by projecting onto just the $\cC$ component.  Note that an object $c$ may appear more than once, if there are distinct isomorphisms
\[\gamma, \gamma'\: d \tmap{\sim} Fc\]
so we cannot suppress the indexing category. 
\end{example}

Thus the inverse image clique can be thought of as an ``algebraic essential inverse image'' where ``algebraic'' refers to the fact that we specify the isomorphisms exhibiting objects to be in the essential inverse image. This will be a key construction enabling us to keep track of all the possible ways of interpreting configurations of points as tensor products, and to take into account when different ways have been invoked.

\begin{remark}
Note that while $F_!$ and $F^*$ are functors, we will only use their action on objects in this work. We will also not use the fact that $F^*$ is an equivalence, but have included this for completeness.
\end{remark}




\section{Labelled configuration spaces}
\label{three}


In this section we will set up the framework of labelled configuration spaces that we will use for our main construction.  This is analogous to the ``train track'' construction of \cite{jk1}.  As Joyal and Kock point out, we want to use configurations of points in 2-space, but with some rigidity to ensure that points can't just move around each other and thus commute.  In our case, we need just the right kind of texture to maintain weak vertical composition but allow horizontal composition to be strict. We start with configurations of points in a unit square, but this would \emph{a priori} be weak in both directions; we then invoke a clique construction to make horizontal composition strict.  

\subsection{Labelled configurations of points}

Where \cite{jk1} uses configurations of points in $\R^2$, we use configurations of points in the interior of $I^2$ where $I$ is the unit interval $[0,1]$.  These can then be stacked horizontally and vertically and reparametrised, to produce 0-composition and 1-composition.  This is weak in \emph{both} directions (but with strict interchange) and naturally forms a \dd\ Trimble-tricategory; we will study this in a future work. In the present work we will eventually use cliques to make horizontal composition strict.  



Given a set $\cO$, let $C_n(I^2, \cO)$ denote the space of configurations of $n$ distinct points in the interior of $I^2$, each labelled by an element of $\cO$.  For example:

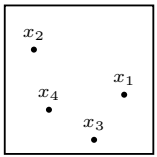
\[\psset{unit=2mm}
\pspicture(-5,-5)(5,5)

\psframe(-5,-5)(5,5)

\rput(3,-1){
\pscircle*(0,0){0.2}
\rput(0,1){\scr $x_1$}
}

\rput(-3,2){
\pscircle*(0,0){0.2}
\rput(0,1){\scr $x_2$}
}

\rput(1,-4){
\pscircle*(0,0){0.2}
\rput(0,1){\scr $x_3$}
}

\rput(-2,-2){
\pscircle*(0,0){0.2}
\rput(0,1){\scr $x_4$}
}


\endpspicture\]
We are interested in the disjoint union
\[C(I^2, \cO) := \coprod_{n \geq 0} C_n(I^2, \cO).\]
Thus $C(I^2, \cO)$ is the space of functions $S \tra \cO$ where $S$ is a finite subset of $(0,1)^2$.  When $\cO=1$ (the singleton set), the points are effectively unlabelled so $C(I^2, \cO)$ is the standard space of configurations of points in a square, which we will write $C(I^2)$; its fundamental groupoid $\Pi_1\big(C(I^2)\big)$ is equivalent to the braid category (the free braided monoidal category on one object).  Similarly, $\Pi_1\big(C(I^2, \cO)\big)$ is equivalent to the free braided monoidal category on $\cO$, and we can put a braided monoidal structure on the former to make this into a braided monoidal equivalence.   As we will be relying on this structure we will now elucidate it further. 

\subsection{Braided monoidal structure}
\label{vertmonconfigs}

We will consider the following braided monoidal structure on $\Pi_1\big(C(I^2, \cO)\big)$.  The tensor product is given by ``stacking'' boxes vertically and scaling them linearly so that each box has equal height inside the resulting unit square. For example given the following objects $X$ and $Y$:

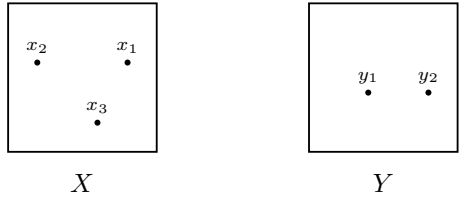
\[\psset{unit=2mm}
\pspicture(-5,-8)(25,5)

\rput(0,0){
\psframe(-5,-5)(5,5)
\rput(0,-7){$X$}

\rput(3,1){
\pscircle*(0,0){0.2}
\rput(0,1){\scr $x_1$}
}

\rput(-3,1){
\pscircle*(0,0){0.2}
\rput(0,1){\scr $x_2$}
}

\rput(1,-3){
\pscircle*(0,0){0.2}
\rput(0,1){\scr $x_3$}
}



}

\rput(20,0){

\psframe(-5,-5)(5,5)
\rput(0,-7){$Y$}

\rput(-1,-1){
\pscircle*(0,0){0.2}
\rput(0,1){\scr $y_1$}
}

\rput(3,-1){
\pscircle*(0,0){0.2}
\rput(0,1){\scr $y_2$}
}




}

\endpspicture\]
the tensor product $X \otimes Y$ is given by the following configuration (where the dotted line is just there to show where the two boxes were ``concatenated''):
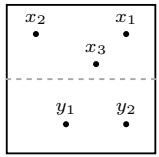
\[\psset{unit=2mm}
\pspicture(-5,-5)(5,5)

\psframe(-5,-5)(5,5)
\psline[linecolor=gray!70!white,linestyle=dashed](-5,0)(5,0)

\rput(0,2.5){

\rput(3,0.5){
\pscircle*(0,0){0.2}
\rput(0,1){\scr $x_1$}
}

\rput(-3,0.5){
\pscircle*(0,0){0.2}
\rput(0,1){\scr $x_2$}
}

\rput(1,-1.5){
\pscircle*(0,0){0.2}
\rput(0,1){\scr $x_3$}
}

}

\rput(0,-2.5){

\rput(-1,-0.5){
\pscircle*(0,0){0.2}
\rput(0,1){\scr $y_1$}
}

\rput(3,-0.5){
\pscircle*(0,0){0.2}
\rput(0,1){\scr $y_2$}
}




}

\endpspicture\]

\noi Note that this tensor product is only weakly associative and unital, and the unit is the empty box.

\begin{remark}
We find that a formal expression for this is easy to write down but adds little, so we leave it to the reader.
\end{remark}

Morphisms in this category are braids.   So far we have a weak monoidal category.  A braiding for this monoidal category is given by braiding multiple strands past each other.  As in Section~\ref{one} we can express this via an Eckmann--Hilton type construction, and we just have to pick a consistent orientation. We will use the following orientation:

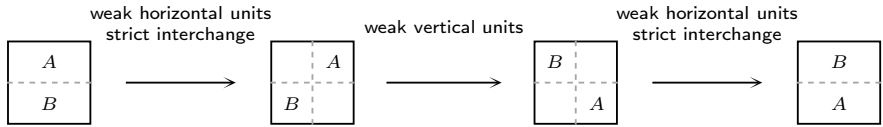
\[\psset{unit=0.7mm}
\pspicture(0,-10)(150,10)

\rput(0,0){\rnode{a1}{
\psframe(-8,-8)(8,8)

\rput(0,4){\scr $A$}

\rput(0,-4){{\scr $B$}}

\psline[linestyle=dashed, linecolor=gray!70!white](-8,0)(8,0)
}}


\rput(50,0){\rnode{a2}{\psframe(-8,-8)(8,8)

\rput(4,4){{\scr $A$}}

\rput(-4,-4){{\scr $B$}}

\psline[linestyle=dashed, linecolor=gray!70!white](-8,0)(8,0)
\psline[linestyle=dashed, linecolor=gray!70!white](0,-8)(0,8)}}


\rput(100,0){\rnode{a3}{\psframe(-8,-8)(8,8)
\rput(4,-4){{\scr $A$}}
%
\rput(-4,4){{\scr $B$}}
%
\psline[linestyle=dashed, linecolor=gray!70!white](-8,0)(8,0)
\psline[linestyle=dashed, linecolor=gray!70!white](0,-8)(0,8)
}}


\rput(150,0){\rnode{a4}{\psframe(-8,-8)(8,8)
\rput(0,-4){{\scr $A$}}
%
\rput(0,4){{\scr $B$}}
%
\psline[linestyle=dashed, linecolor=gray!70!white](-8,0)(8,0)
%
}}

\psset{nodesep=29pt, labelsep=4pt}

\ncline[arrows=->]{a1}{a2} \naput{\scr\sf \parbox{10em}{\bc weak horizontal units\\ strict interchange\ec}}
\ncline[arrows=->,nodesep=28pt]{a2}{a3} \naput[labelsep=18pt]{\scr\sf weak vertical units}
\ncline[arrows=->]{a3}{a4} \naput{\scr\sf \parbox{10em}{\bc weak horizontal units\\ strict interchange\ec}}

\endpspicture\]

\noi We will represent this in braid diagrams as follows:


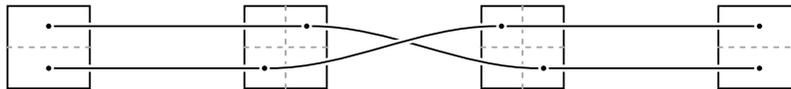
\[\psset{unit=0.7mm}
\pspicture(0,-10)(135,10)

\rput(0,0){\rnode{a1}{
\psframe(-8,-8)(8,8)

\rput(0,4){\rnode{aa1}{\pscircle*(0,0){0.5}}}

\rput(0,-4){\rnode{bb1}{\pscircle*(0,0){0.5}}}

\psline[linestyle=dashed, linecolor=gray!70!white](-8,0)(8,0)
}}


\rput(45,0){\rnode{a2}{\psframe(-8,-8)(8,8)

\rput(4,4){\rnode{aa2}{\pscircle*(0,0){0.5}}}

\rput(-4,-4){\rnode{bb2}{\pscircle*(0,0){0.5}}}

\psline[linestyle=dashed, linecolor=gray!70!white](-8,0)(8,0)
\psline[linestyle=dashed, linecolor=gray!70!white](0,-8)(0,8)}}


\rput(90,0){\rnode{a3}{\psframe(-8,-8)(8,8)
\rput(4,-4){\rnode{aa3}{\pscircle*(0,0){0.5}}}
\rput(-4,4){\rnode{bb3}{\pscircle*(0,0){0.5}}}
\psline[linestyle=dashed, linecolor=gray!70!white](-8,0)(8,0)
\psline[linestyle=dashed, linecolor=gray!70!white](0,-8)(0,8)
}}


\rput(135,0){\rnode{a4}{\psframe(-8,-8)(8,8)
\rput(0,-4){\rnode{aa4}{\pscircle*(0,0){0.5}}}

\rput(0,4){\rnode{bb4}{\pscircle*(0,0){0.5}}}
\psline[linestyle=dashed, linecolor=gray!70!white](-8,0)(8,0)
%
}}

\psset{border=1pt}

\nccurve[angleA=0,angleB=180]{aa1}{aa2}
\nccurve[angleA=0,angleB=180]{aa2}{aa3}
\nccurve[angleA=0,angleB=180]{aa3}{aa4}

\nccurve[angleA=0,angleB=180]{bb1}{bb2}
\nccurve[angleA=0,angleB=180]{bb2}{bb3}
\nccurve[angleA=0,angleB=180]{bb3}{bb4}

%

\endpspicture\]

\noi The braiding therefore amounts to the following braid:


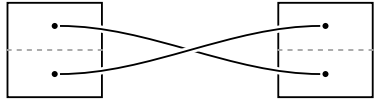
\[\psset{unit=0.8mm}
\pspicture(0,-10)(45,10)

\rput(0,0){\rnode{a1}{
\psframe(-8,-8)(8,8)

\rput(0,4){\rnode{aa1}{\pscircle*(0,0){0.5}}}

\rput(0,-4){\rnode{bb1}{\pscircle*(0,0){0.5}}}

\psline[linestyle=dashed, linecolor=gray!70!white](-8,0)(8,0)
}}


%
%
%
%
%
%

\rput(45,0){\rnode{a4}{\psframe(-8,-8)(8,8)
\rput(0,-4){\rnode{aa4}{\pscircle*(0,0){0.5}}}

\rput(0,4){\rnode{bb4}{\pscircle*(0,0){0.5}}}
\psline[linestyle=dashed, linecolor=gray!70!white](-8,0)(8,0)
%
}}

\psset{border=1pt}

\nccurve[angleA=0,angleB=180]{aa1}{aa4}

\nccurve[angleA=0,angleB=180]{bb1}{bb4}

%

\endpspicture\]

\noi Note that if we represented this more three-dimensionally it might look like the diagram below, but we will continue to use the ``flat'' presentation.


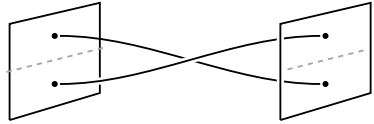
\[\psset{unit=0.8mm}
\pspicture(0,-10)(45,10)

\rput(0,0){\rnode{a1}{

\psline(-7.5,-10)(7.5,-5.5)(7.5,9.5)(-7.5,6)(-7.5,-10)

\rput(0,4){\rnode{aa1}{\pscircle*(0,0){0.5}}}

\rput(0,-4){\rnode{bb1}{\pscircle*(0,0){0.5}}}

\psline[linestyle=dashed, linecolor=gray!70!white](-8,-2)(8,2)
}}


%
%
%
%
%
%

\rput(45,0){\rnode{a4}{
%
\rput(0,-4){\rnode{aa4}{\pscircle*(0,0){0.5}}}

\rput(0,4){\rnode{bb4}{\pscircle*(0,0){0.5}}}
%

\nccurve[angleA=0,angleB=180]{aa1}{aa4}

\psset{border=1pt}

\nccurve[angleA=0,angleB=180]{bb1}{bb4}

\psline[linestyle=dashed, linecolor=gray!70!white](-8,-2)(8,2)

\psset{border=1pt}

\psline(-7.5,-10)(7.5,-5.5)(7.5,9.5)(-7.5,6)(-7.5,-10)

%
}}

\psset{border=1pt}


%

\endpspicture\]

\subsection{Equivalence with free braided monoidal category}
\label{equivfree}

As in \cite{jk1} we will need to embed the free braided monoidal category on a set $\cO$ into the above braided monoidal category of points labelled by $\cO$.  As we are dealing with braided monoidal categories whose monoidal structure is weak, we will take $\cF\cO$ to be the free weak \bmc\ on the object set $\cO$.  So the objects are parenthesised words in $\cO$ and a formal unit, and the morphisms are braids labelled by objects of $\cO$.  We will use the universal property of $\cF\cO$ to induce a braided monoidal functor
\[\cF\cO \mtmap{F} \Pi_1\big(C(I^2, \cO)\big)\]
so we just have to define a map from the set $\cO$ to the set of objects of $\Pi_1\big(C(I^2, \cO)\big)$.  We do this by sending each element $x \in \cO$ to the configuration of a single point in the centre of $I^2$, labelled by $x$.  We will refer to this as a ``singleton box'' which we can draw as:
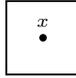
\[\pspicture(-5,-5)(5,5)

\psframe(-5,-5)(5,5)
\pscircle*(0,0){0.5}
\rput(0,2){\scr $x$}

\endpspicture\]
Note that the induced functor $F$ is strictly monoidal.  

It may help to sketch more of the details:

\begin{itemize}

\item The empty word is mapped to the empty box, that is, the configuration of 0 points in $I^2$.

\item The word $(x_1, x_2)$ is the tensor product $x_1 \otimes x_2$ so must be mapped to the configuration obtained from two singleton boxes stacked vertically and reparametrised equally.  Thus we have the following, where, again,  the dotted line merely depicts where the two boxes were joined together.

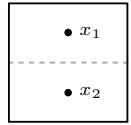
\[\pspicture(-8,-8)(8,8)

\psline[linestyle=dashed, linecolor=gray!70!white](-8,0)(8,0)

\psframe(-8,-8)(8,8)
\pscircle*(0,4){0.5}
\rput(3,4){\scr $x_1$}

\pscircle*(0,-4){0.5}
\rput(3,-4){\scr $x_2$}

\endpspicture\]

\item If $X$ and $Y$ are (parenthesised) words, then $X \otimes Y$ must be mapped to $FX$ and $FY$ stacked vertically and reparametrised equally:

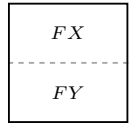
\[\pspicture(-8,-8)(8,8)

\psline[linestyle=dashed, linecolor=gray!70!white](-8,0)(8,0)

\psframe(-8,-8)(8,8)
%

\rput(0,4){\scr $FX$}
\rput(0,-4){\scr $FY$}

\endpspicture\]
\end{itemize}

\noi Thus all the configurations in the image of $F$ have all their points in a vertical line running down the center of the box.  (This is not sufficient to be in the image, but is necessary.)

On morphisms, $F$ sends a braid between words to the corresponding braid connecting the associated labelled configurations, according to the orientation for the braiding we fixed when defining the braided monoidal structure of $\Pi_1\big(C(I^2, \cO)\big)$.


\begin{proposition}
The functor $\cF\cO \mtmap{F} \Pi_1\big(C(I^2, \cO)\big)$ is a braided monoidal equivalence of braided monoidal categories.
\end{proposition}

\begin{proof}
Coherence for braided monoidal categories tells us that for any words $X,Y$ in the objects of $\cO$, a morphism $X \tra Y$ in $\cF\cO$ is precisely given by a braid connecting the words, with consistent labels on the strands.  Such a braid also precisely defines a morphism $FX \tra FY$ in  $\Pi_1\big(C(I^2, \cO)\big)$ so $F$ is full and faithful.

For essential surjectivity note that given any object in $\Pi_1\big(C(I^2, \cO)\big)$, that is a labelled configuration of points, we can make a path (necessarily an isomorphism) to a vertical configuration, appropriately spaced to be in the image of $F$.

The functor $F$ is a braided monoidal functor by construction, thus it is a braided monoidal equivalence.
\end{proof}

%
%

\subsection{Horizontal slide maps}
\label{slidemaps}


The weak monoidal structure we are considering on $\Pi_1\big(C(I^2, \cO)\big)$ is in the vertical direction.  Eventually we also need a notion of horizontal tensor product, but we will need it to be strict.  To deal with this we introduce a notion of ``slide map'' to allow the points to ``slide'' horizontally freely.  We can picture this idea like an abacus (as compared with the train-tracks of \cite{jk1}): there are horizontal ``rails'' on which points can slide freely, but they cannot swap places on a rail, and they cannot move to a different vertical height.

There are two equivalent ways to make this construction: by an equivalence relation, or using cliques.  We will follow \cite{jk1} and use cliques; this approach enables an efficacious treatment of morphisms.

Write $C_n(I^2)$ for the configuration space of $n$ distinct points in the interior of $I^2$, and consider its fundamental groupoid, the category $\Pi_1\big(C_n(I^2)\big)$. So objects are configurations of $n$ points in the interior of $I^2$, and morphisms are homotopy classes of paths between them.  

\begin{definition}
We define a \emph{slide path} between points of $C_n(I^2)$ to be any path between configurations that keeps the $y$ coordinate of every point fixed.  We define a \emph{slide map} to be a homotopy class of a slide path; these are thus particular morphisms in the category $\Pi_1\big(C_n(I^2)\big)$. We say that two configurations are \emph{slide-equivalent} if they are isomorphic in this category via a slide map.  More generally for the labelled version a slide map in $\Pi_1\big(C(I^2, \cO)\big)$ is a map whose underlying unlabelled map is a slide map in $\Pi_1\big(C(I^2)\big)$.

\end{definition}

So points can ``slide'' sideways via a slide map but they cannot cross over each other (as that would require temporarily changing $y$ coordinate).  If two objects in $\Pi_1\big(C_n(I^2)\big)$ are isomorphic via a slide map then they are uniquely so, as all slide paths are homotopic to linear ones. So we can look at cliques where the connecting isomorphisms are given by these slide maps.  

\begin{definition}
We define the \emph{slide cliques} of the category $\Pi_1\big(C(I^2, \cO)\big)$ to be those cliques consisting of an entire slide-equivalence class of configurations, together with the unique slide maps between them.  
\end{definition}

As mentioned in Remark~\ref{repeatobjects} in this case objects of the clique are not repeated, so the indexing category can be suppressed.  Or, we can take $J$ to be a set of slide-equivalent unlabelled configurations of $n$ points in $I^2$, and then specifying a clique $X: \ol{J} \tra \Pi_1\big(C(I^2, \cO)\big)$ consists of picking $n$ objects of $\cO$ to use as labels for the points.  An object $X_j$ in the clique is then the configuration $j$ labelled by the chosen elements of $\cO$, and the connecting map $X_{ij}$ is the (unique) slide map connecting the configurations $i$ and $j$.  


\begin{proposition}
The (weak) vertical monoidal structure on ${\Pi_1\big(C(I^2, \cO)\big)}$ transfers to the slide cliques via $\wt{\otimes}$.
\end{proposition}

\begin{proof}
Consider slide cliques $X$ and $Y$. The clique $X \wt{\otimes} Y$ is defined componentwise so is not \emph{a priori} a slide clique. We need to check that it is in fact a slide clique. Recall that $X \wt{\otimes} Y$ is 
\[\{X_j \otimes Y_k \ | \ j \in J, k \in K \}\]
where $X$ and $Y$ are indexed by $J$ and $K$ respectively.  $X_j \otimes Y_k$ consists of the individual configurations, stacked vertically and reparametrised equally. So a configuration is slide equivalent to $X_j \otimes Y_k$ if and only if it can be expressed as two halves vertically, where the top half is slide equivalent to $X_j$ and the bottom half is slide equivalent to $Y_k$.  This is precisely the definition of $X \wt{\otimes}Y$. 
\end{proof}


We will use slide cliques to ensure that our horizontal tensor product is strict.  

\begin{proposition}
There is a strict monoidal structure on slide cliques produced by stacking configurations horizontally.
\end{proposition}

\begin{proof}
There is a weak monoidal structure on the configurations, just as for the vertical structure, but with boxes stacked horizontally instead. We write this as $A | B$. By abuse of notation will write the horizontal tensor product of slide cliques the same way.  Given slide cliques $X$ and $Y$ we define $X|Y$ to be the slide clique of $X_i | Y_j$ for any $i$ and $j$.  For associativity note that $(X|Y)|Z$ is represented by $(X_i|Y_j)|Z_k$ and $X|(Y|Z)$ is represented by $X_i|(Y_j|Z_k)$ but that these are in the same slide clique, with connecting map given by the associator, thus the associator represents the identity map, exhibiting this tensor product as strictly associative.  The unit laws follow similarly.
\end{proof}

\begin{remark}
Note that this tensor product is not an induced tensor product on cliques. The tensor product on cliques induced from horizontal stacking does not produce the slide clique in question, because if we stack $X_i$ and $Y_j$ horizontally, the slide clique of the resulting configuration includes configurations that cannot be decomposed into two equal halves.
\end{remark}


We are now ready to use slide cliques to make our main construction.


\section{The main construction}
\label{mainconstruction}

We will now start with a \bmc\ $B$ and show how to define a doubly-degenerate \bicats-category $\Sigma B$ from it.  This construction is analogous to the construction given in \cite{jk1}, but with slide cliques instead of train tracks.  By Proposition~\ref{dd-char} we need to construct a category with two monoidal structures, one weak, one strict, satisfying strict interchange. We begin with the underlying category.

\subsection{The underlying category of $\Sigma B$}
\label{sigmacategory}

The idea of this construction is as follows.  First we take the objects to be configurations of points in $I^2$ labelled by objects of $B$, subject to slide-equivalence to ensure horizontal strictness.  We then want to ``pull back'' morphisms from $B$, so we want to interpret each configuration of points as a tensor product of objects of $B$, but there is no canonical way to do so, and no way that will interact coherently with both horizontal and vertical stacking of boxes.  If we just pick a way we will get weak interchange as an artefact.  Instead, we look at all possible ways of interpeting the configuration as a word, each equipped with a braid relating the original configuration to the word.  These will form a clique in $B$, and taking clique maps between them will deal with the coherence and absorb the weakness of the interchange.   

\begin{definition}
Let $B$ be a \bmc\ with object set $\cO$.  We define a category $\Sigma B$ as follows.  The objects are the slide cliques of $\Pi_1\big(C(I^2, \cO)\big)$.

For morphisms, first note that we have the following canonical functors, induced from the universal property of $\cF \cO$, where $F$ is as constructed in Section~\ref{equivfree}, and $G$ is induced by mapping each object of $\cO$ to itself in $B$:
\[\Pi_1\big(C(I^2, \cO)\big) \mtlmap{F} \cF\cO \mtmap{G} B.\]
We can then pass to cliques, and, since $F$ is an equivalence, we can produce the following composite functor (using the constructions of Section~\ref{inducedfunctors}), which we call $\theta$:
\[\theta := \hs{-0.3}\raisebox{0.1em}{\makebox[0.6em][l]{\scalebox{2.3}[1.3]{$\widetilde{\hs{3}}$}}}\Pi_1\big(C(I^2, \cO)\big) \mtmap{\hs{0.3}F^*} \wt{\cF\cO} \mtmap{G_!} \wt{B}.\]

\noi Now, objects of $\Sigma B$ are by definition certain objects of \hs{-0.3}$\raisebox{0.1em}{\makebox[0.6em][l]{\scalebox{2.3}[1.3]{$\widetilde{\hs{3}}$}}}\Pi_1\big(C(I^2, \cO)\big)$, so we can apply $\theta$ to them. So, given objects objects $X, Y \in \Sigma B$ we define 
\[\Sigma B(X, Y) := \wt{B}(\theta X, \theta Y).\]

\noi This completes the definition of the category $\Sigma B$. 

\end{definition}

\subsection{Unravelling the definition of $\Sigma B$}

It is worth unravelling this definition to get a better idea of the construction, which will in turn allow us to set up system of notation for calculating with the morphisms.

%

An object of $\Sigma B$ is a slide clique of configurations of points in the interior of $I^2$ labelled by elements of $\cO$, with connecting maps given by the slide maps. So two configurations are in the same clique if one can be obtained from the other by just ``sliding'' points horizontally without ever changing any vertical coordinates.  

For example if we start with the configuration $X$ below, then:

\begin{itemize}

\item $Y$ represents the same equivalence class.
\item $Z$ is not equivalent as the vertical coordinate of $x_3$ is different.

\item $K$ is not equivalent, as the vertical coordinate of $x_1$ or $x_2$ would have to change in order for them to move ``past'' each other.
\end{itemize}
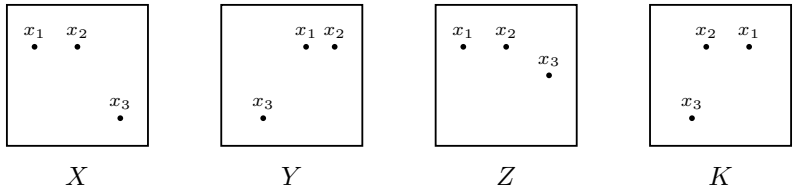
\[\psset{unit=1.9mm}
\pspicture(-5,-8)(50,5)

\rput(0,0){
\psframe(-5,-5)(5,5)
\rput(0,-7){$X$}

\rput(0,2){
\pscircle*(0,0){0.2}
\rput(0,1){\scr $x_2$}
}

\rput(-3,2){
\pscircle*(0,0){0.2}
\rput(0,1){\scr $x_1$}
}

\rput(3,-3){
\pscircle*(0,0){0.2}
\rput(0,1){\scr $x_3$}
}

}

\rput(15,0){
\psframe(-5,-5)(5,5)
\rput(0,-7){$Y$}

\rput(3,2){
\pscircle*(0,0){0.2}
\rput(0,1){\scr $x_2$}
}

\rput(1,2){
\pscircle*(0,0){0.2}
\rput(0,1){\scr $x_1$}
}

\rput(-2,-3){
\pscircle*(0,0){0.2}
\rput(0,1){\scr $x_3$}
}

}

\rput(30,0){
\psframe(-5,-5)(5,5)
\rput(0,-7){$Z$}

\rput(0,2){
\pscircle*(0,0){0.2}
\rput(0,1){\scr $x_2$}
}

\rput(-3,2){
\pscircle*(0,0){0.2}
\rput(0,1){\scr $x_1$}
}

\rput(3,0){
\pscircle*(0,0){0.2}
\rput(0,1){\scr $x_3$}
}

}

\rput(45,0){
\psframe(-5,-5)(5,5)
\rput(0,-7){$K$}

\rput(-1,2){
\pscircle*(0,0){0.2}
\rput(0,1){\scr $x_2$}
}

\rput(2,2){
\pscircle*(0,0){0.2}
\rput(0,1){\scr $x_1$}
}

\rput(-2,-3){
\pscircle*(0,0){0.2}
\rput(0,1){\scr $x_3$}
}

}

\endpspicture\]

\noi To describe the morphisms of $\Sigma B$ we need to understand the functors $F^*$ and $G_!$.  We start with a slide clique ${X} \: \ol{J} \tra \Pi_1\big(C(I^2, \cO)\big)$. Then $F^*{X}$ is a clique in $\cF\cO$ indexed by the 2-fibred product $\cC \dtimes_{\hs{-0.2}\cD} \hs{0.1}\ol{J}$, where
\[\begin{array}{ccl}
\cC &=& \cF\cO \\
\cD &=& \Pi_1\big(C(I^2, \cO)\big)
\end{array}\]

\noi This 2-fibred product has objects of the form $(c, j, \gamma)$, where
\begin{itemize}
\item $c\in \cF\cO$, so is a parenthesised word in objects of $\cO$, for example $\big( b_1, ( b_2 , b_3)\big)$,

\item $j \in \ol{J}$, so is an unlabelled configuration of points in $I^2$, 

\item $\gamma\: X_j \tmap{\sim} Fc$ is an isomorphism in $\Pi_1\big(\cC(I^2, \cO)\big)$, which we will further elucidate below. 
\end{itemize}

\noi To understand what $\gamma$ is, note that $Fc$ consists of some points in a vertical line, labelled according to the objects making up $c$, and spaced according to its parenthesisation.  Then $\gamma$ is a braid connecting the labelled points in $I^2$ (the configuration $X_j$) to the labelled points on the vertical line (the configuration $Fc$), with labels matching along strands.  Following the analogy with the construction in \cite{jk1}, we will refer to this as a ``linearising braid'' as it ``linearises'' the configuration of points into a configuration all in a straight line.  

Here is an example:

\begin{itemize}

\item $c = \big( b_1 , (b_2 , b_3)\big)$

\item $j$ is this configuration of points 
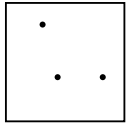
\[\renewcommand{\nowdot}{\psset{unit=1mm}\pscircle*(0,0){0.4}}
\ps(16,16)

\rput(0,0){

\psframe(0,0)(16,16)

\rput(5,13){\rnode{b1}{\nowdot}
\rput(-2,0){}
}

\rput(7,6){\rnode{b2}{\nowdot}
\rput(0,-2){}
}
\rput(13,6){\rnode{b3}{\nowdot}
\rput(0,-2){}}

}
\eps
\]

\item $\gamma$ is the following linearising braid; the points in the configuration of $Fc$ are technically on a vertical line in the middle of a square, but we will depict them on an interval for emphasis:

\end{itemize}

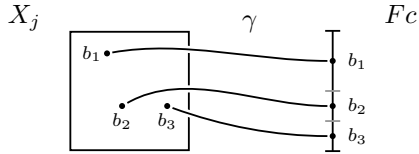
\[\renewcommand{\nowdot}{\psset{unit=1mm}\pscircle*(0,0){0.4}}
\ps(0,34)(40,51)

\rput(0,35){

\psframe(0,0)(16,16)
\pnode(8,8){x0}

\rput(5,13){\rnode{b1}{\nowdot}
\rput(-2,0){\scr $b_1$}
}

\rput(7,6){\rnode{b2}{\nowdot}
\rput(0,-2){\scr $b_2$}
}
\rput(13,6){\rnode{b3}{\nowdot}
\rput(0,-2){\scr $b_3$}}

\rput(24,17){$\gamma$}

\rput(35,0){
\psline(0,0)(0,16)
\pnode(0,8){Fc0}
\rput(0,16){\psline(-1,0)(1,0)}
\rput(0,0){\psline(-1,0)(1,0)}

\rput(0,8){\psline[linecolor=gray!80!white](-1,0)(1,0)}
\rput(0,4){\psline[linecolor=gray!80!white](-1,0)(1,0)}

\rput(0,12){\rnode{c1}{\nowdot}
\rput[l](2,0){\scr $b_1$}}

\rput(0,6){\rnode{c2}{\nowdot}
\rput[l](2,0){\scr $b_2$}}

\rput(0,2){\rnode{c3}{\nowdot}
\rput[l](2,0){\scr $b_3$}}

\rput[B](9,17){$Fc$}

}

\psset{border=1pt}

\nccurve[angleA=10,angleB=180]{b1}{c1}
\nccurve[angleA=30,angleB=180]{b2}{c2}
\nccurve[angleA=-20,angleB=180]{b3}{c3}

\rput[B](-6,17){$X_j$}

}

\eps\]

Note that all the information of $(c, j, \gamma)$ is encapsulated in the diagram of the linearising braid, at least in theory, as the placement of the objects $b_i$ on the vertical line uniquely specifies their parenthesisation. 



A map $(c_0, j_0, \gamma_0)  \tra (c_1, j_1, \gamma_1)$ in the 2-fibred product is then a map $w\:c_0 \tra c_1$ in $\cF\cO$ making this diagram commute in $\Pi_1\big(C(I^2, \cO)\big)$:

\[
\psset{unit=0.1cm,labelsep=3pt,nodesep=3pt}
\pspicture(0,-3)(20,21)



\rput(20,18){\rnode{a1}{$Fc_0$}}  
\rput(0,18){\rnode{a2}{$X_{j_0}$}}  
\rput(20,0){\rnode{a3}{$Fc_1$}}  
\rput(0,0){\rnode{a4}{$X_{j_1}$}}  

\ncline{->}{a2}{a1} \naput{{\scriptsize $\gamma_0$}} 
\ncline{->}{a4}{a3} \nbput{{\scriptsize $\gamma_1$}} 
\ncline{->}{a1}{a3} \naput{{\scriptsize $Fw$}} 
\ncline{->}{a2}{a4} \nbput{{\scriptsize $X_{j_0j_1}$}} 

\endpspicture
\]
Now $w$ is a morphism of $\cF\cO$ so is an abstract braid and/or re-association from the word $c_0$ to the word $c_1$.  $Fw$ is this abstract braid/re-association realised as a braid from the points of $Fc_0$ to the points of $Fc_1$.  $X_{j_0j_1}$ is the connecting isomorphism between $X_{j_0}$ and $X_{j_1}$, that is, the unique slide map between them.  The diagram that needs to commute can be pictured as below. As three sides of the square are isomorphisms we see that there is only one possible braid $Fc_0 \tra Fc_1$ that will make the diagram commute, and so we refer to this as a ``mediating braid'', as it mediates between the two linearising braids.

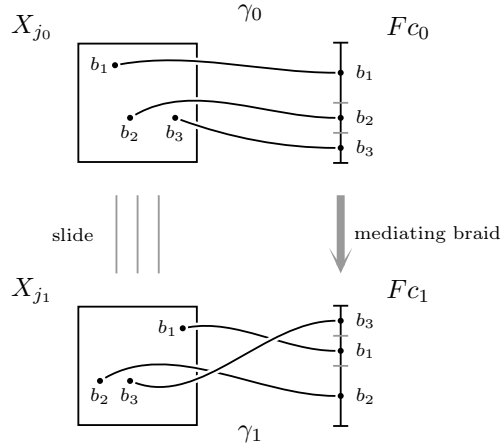
\[\renewcommand{\nowdot}{\psset{unit=1mm}\pscircle*(0,0){0.4}}
\ps(0,-3)(40,60)

\rput(0,35){

\rput(23,20){$\gamma_0$}

\psframe(0,0)(16,16)
\pnode(8,8){x0}

\rput(5,13){\rnode{b1}{\nowdot}
\rput(-2,0){\scr $b_1$}
}

\rput(7,6){\rnode{b2}{\nowdot}
\rput(0,-2){\scr $b_2$}
}
\rput(13,6){\rnode{b3}{\nowdot}
\rput(0,-2){\scr $b_3$}}

\rput(35,0){
\psline(0,0)(0,16)
\pnode(0,8){Fc0}
\rput(0,16){\psline(-1,0)(1,0)}
\rput(0,0){\psline(-1,0)(1,0)}

\rput(0,8){\psline[linecolor=gray!80!white](-1,0)(1,0)}
\rput(0,4){\psline[linecolor=gray!80!white](-1,0)(1,0)}

\rput(0,12){\rnode{c1}{\nowdot}
\rput[l](2,0){\scr $b_1$}}

\rput(0,6){\rnode{c2}{\nowdot}
\rput[l](2,0){\scr $b_2$}}

\rput(0,2){\rnode{c3}{\nowdot}
\rput[l](2,0){\scr $b_3$}}

\rput(9,18){$Fc_0$}

}

\psset{border=1pt}

\nccurve[angleA=10,angleB=180]{b1}{c1}
\nccurve[angleA=30,angleB=180]{b2}{c2}
\nccurve[angleA=-20,angleB=180]{b3}{c3}

\rput(-6,18){$X_{j_0}$}

}


\rput(0,0){

\rput(23,-1){$\gamma_1$}

\psframe(0,0)(16,16)
\pnode(8,8){x1}

\rput(14,13){\rnode{b1}{\nowdot}
\rput(-2,0){\scr $b_1$}
}

\rput(3,6){\rnode{b2}{\nowdot}
\rput(0,-2){\scr $b_2$}
}
\rput(7,6){\rnode{b3}{\nowdot}
\rput(0,-2){\scr $b_3$}}

\rput(35,0){
\psline(0,0)(0,16)
\pnode(0,8){Fc1}
\rput(0,16){\psline(-1,0)(1,0)}
\rput(0,0){\psline(-1,0)(1,0)}

\rput(0,12){\psline[linecolor=gray!80!white](-1,0)(1,0)}
\rput(0,8){\psline[linecolor=gray!80!white](-1,0)(1,0)}

\rput(0,10){\rnode{c1}{\nowdot}
\rput[l](2,0){\scr $b_1$}
}

\rput(0,4){\rnode{c2}{\nowdot}
\rput[l](2,0){\scr $b_2$}}

\rput(0,14){\rnode{c3}{\nowdot}
\rput[l](2,0){\scr $b_3$}}

\rput(9,18){$Fc_1$}

}

\psset{border=1pt}

\nccurve[angleA=10,angleB=180]{b1}{c1}
\nccurve[angleA=30,angleB=180]{b2}{c2}
\nccurve[angleA=-20,angleB=180]{b3}{c3}

\rput(-6,18){$X_{j_1}$}

}



{\psset{nodesep=35pt, linecolor=gray!80!white}

\ncline[offset=8pt]{x0}{x1}
\ncline{x0}{x1}
\ncline[offset=-8pt]{x0}{x1}

\nbput[labelsep=8pt]{\scr slide}

\ncline[linewidth=3pt]{->}{Fc0}{Fc1}
\naput{\scr mediating braid}
}

\eps\]

\noi We can then depict the mediating braid as below, and we can see that the composite braid of $\gamma_0$ followed by the mediating braid is the braid $\gamma_1$:

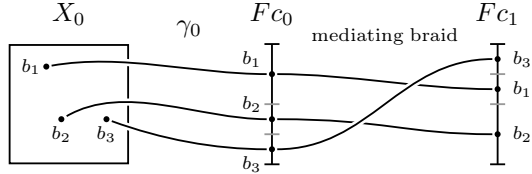
\[\renewcommand{\nowdot}{\psset{unit=1mm}\pscircle*(0,0){0.4}}
\ps(0,30)(70,60)

\rput(0,35){

\rput(24,18){$\gamma_0$}

\psframe(0,0)(16,16)
\pnode(8,8){x0}

\rput(5,13){\rnode{b1}{\nowdot}
\rput(-2,0){\scr $b_1$}
}

\rput(7,6){\rnode{b2}{\nowdot}
\rput(0,-2){\scr $b_2$}
}
\rput(13,6){\rnode{b3}{\nowdot}
\rput(0,-2){\scr $b_3$}}

\rput(35,0){
\psline(0,0)(0,16)
\pnode(0,8){Fc0}
\rput(0,16){\psline(-1,0)(1,0)}
\rput(0,0){\psline(-1,0)(1,0)}

\rput(0,8){\psline[linecolor=gray!80!white](-1,0)(1,0)}
\rput(0,4){\psline[linecolor=gray!80!white](-1,0)(1,0)}

\rput(0,12){\rnode{c1}{\nowdot}
\rput[r](-1.4,2){\scr $b_1$}}

\rput(0,6){\rnode{c2}{\nowdot}
\rput[r](-1.4,2){\scr $b_2$}}

\rput(0,2){\rnode{c3}{\nowdot}
\rput[r](-1.4,-2){\scr $b_3$}}

\rput(0,20){$Fc_0$}

}

\psset{border=1pt}

\nccurve[angleA=10,angleB=180]{b1}{c1}
\nccurve[angleA=30,angleB=180]{b2}{c2}
\nccurve[angleA=-20,angleB=180]{b3}{c3}

\rput(8,20){$X_0$}

}

\rput(65,35){
\psline(0,0)(0,16)
\pnode(0,8){Fc1}
\rput(0,16){\psline(-1,0)(1,0)}
\rput(0,0){\psline(-1,0)(1,0)}

\rput(0,12){\psline[linecolor=gray!80!white](-1,0)(1,0)}
\rput(0,8){\psline[linecolor=gray!80!white](-1,0)(1,0)}

\rput(0,10){\rnode{d1}{\nowdot}
\rput[l](2,0){\scr $b_1$}
}

\rput(0,4){\rnode{d2}{\nowdot}
\rput[l](2,0){\scr $b_2$}}

\rput(0,14){\rnode{d3}{\nowdot}
\rput[l](2,0){\scr $b_3$}}

\rput(0,20){$Fc_1$}

}

\psset{border=1pt}

\nccurve[angleA=0,angleB=180]{c1}{d1}
\nccurve[angleA=0,angleB=180,ncurvA=1,ncurvB=0.8]{c2}{d2}
\nccurve[angleA=0,angleB=180,ncurvA=0.8,ncurvB=0.8]{c3}{d3}

\rput(50,52){\scr mediating braid}

\eps\]

\noi Furthermore, as $F$ is full and faithful we know that there is a unique morphism $w\:c_0 \tra c_1 \in \cF\cO$ whose image under $F$ is the mediating braid.  

This completes the description of the 2-fibred product indexing the clique $F^*{X}$.  The clique itself is then the clique in $\cF\cO$ produced from the projection from the 2-fibred product onto the $\cF\cO$ component, that is
\[\begin{array}{ccc}
\cF\cO \dtimes_{\hs{-0.2}\cD} \hs{0.1}\ol{J} & \mtra & \cF\cO \\
(c, j, \gamma) & \tmapsto & c
\end{array}\]

\noi So the objects of the clique are words $c$, but as each word appears multiple times indexed by different $(j, \gamma)$ we find it best to think of the word together with the configuration in $I^2$ and the linearising braid.

For example, consider the word $c = \big(b_1 , (b_2 , b_3)\big)$. Here are two distinct ways in which it arises in the clique $F^*X$, indexed by different configurations and linearising braids:

%

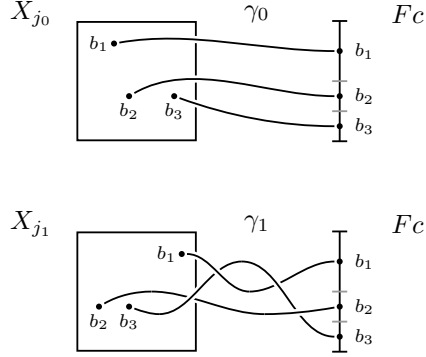
\[\renewcommand{\nowdot}{\psset{unit=1mm}\pscircle*(0,0){0.4}}
\ps(0,3)(40,55)

\rput(0,35){

\psframe(0,0)(16,16)
\pnode(8,8){x0}

\rput(5,13){\rnode{b1}{\nowdot}
\rput(-2,0){\scr $b_1$}
}

\rput(7,6){\rnode{b2}{\nowdot}
\rput(0,-2){\scr $b_2$}
}
\rput(13,6){\rnode{b3}{\nowdot}
\rput(0,-2){\scr $b_3$}}

\rput(35,0){
\psline(0,0)(0,16)
\pnode(0,8){Fc0}
\rput(0,16){\psline(-1,0)(1,0)}
\rput(0,0){\psline(-1,0)(1,0)}

\rput(0,8){\psline[linecolor=gray!80!white](-1,0)(1,0)}
\rput(0,4){\psline[linecolor=gray!80!white](-1,0)(1,0)}

\rput(0,12){\rnode{c1}{\nowdot}
\rput[l](2,0){\scr $b_1$}}

\rput(0,6){\rnode{c2}{\nowdot}
\rput[l](2,0){\scr $b_2$}}

\rput(0,2){\rnode{c3}{\nowdot}
\rput[l](2,0){\scr $b_3$}}

\rput(9,17){$Fc$}

}

\psset{border=1pt}

\nccurve[angleA=10,angleB=180]{b1}{c1}
\nccurve[angleA=30,angleB=180]{b2}{c2}
\nccurve[angleA=-20,angleB=180]{b3}{c3}

\rput(-6,17){$X_{j_0}$}
\rput(24,17){$\gamma_0$}

}


\rput(0,7){

\psframe(0,0)(16,16)
\pnode(8,8){x1}

\rput(14,13){\rnode{b1}{\nowdot}
\rput(-2,0){\scr $b_1$}
}

\rput(3,6){\rnode{b2}{\nowdot}
\rput(0,-2){\scr $b_2$}
}
\rput(7,6){\rnode{b3}{\nowdot}
\rput(0,-2){\scr $b_3$}}

\rput(35,0){
\psline(0,0)(0,16)
\pnode(0,8){Fc1}
\rput(0,16){\psline(-1,0)(1,0)}
\rput(0,0){\psline(-1,0)(1,0)}

\rput(0,8){\psline[linecolor=gray!80!white](-1,0)(1,0)}
\rput(0,4){\psline[linecolor=gray!80!white](-1,0)(1,0)}

\rput(0,12){\rnode{c1}{\nowdot}
\rput[l](2,0){\scr $b_1$}}

\rput(0,6){\rnode{c2}{\nowdot}
\rput[l](2,0){\scr $b_2$}}

\rput(0,2){\rnode{c3}{\nowdot}
\rput[l](2,0){\scr $b_3$}}

\rput(9,17){$Fc$}

}

\pnode(23,8){b11}
\pnode(10,8){b21}
\pnode(11,5){b31}

\pnode(25,5){b22}
\pnode(22,12){b32}

\psset{border=1pt}


\nccurve[angleA=30,angleB=180,nodesepB=0pt]{b2}{b21}
\nccurve[angleA=-20,angleB=180,nodesepB=0pt]{b3}{b31}

\nccurve[angleA=0,angleB=180,nodesep=0pt]{b21}{b22}
\nccurve[angleA=0,angleB=180,nodesep=0pt]{b31}{b32}

\nccurve[angleA=0,angleB=180,nodesepB=0pt]{b1}{b11}

\nccurve[angleA=0,angleB=180,nodesepA=0pt]{b11}{c1}
\nccurve[angleA=0,angleB=180,nodesepA=0pt]{b32}{c3}
\nccurve[angleA=0,angleB=190,nodesepA=0pt]{b22}{c2}

\rput(-6,17){$X_{j_1}$}
\rput(24,17){$\gamma_1$}

}



%
%
%
%

\eps\]

\noi The connecting isomorphisms in $F^*{X}$ are the mediating braids.  For the above example this is a non-trivial braid connecting $Fc$ with itself, mediating between the braids $\gamma_0$ and $\gamma_1$.

This completes our characterisation of the functor
\[\raisebox{0.1em}{\makebox[0.6em][l]{\scalebox{2.3}[1.3]{$\widetilde{\hs{3}}$}}}\Pi_1\big(C(I^2, \cO)\big) \mtmap{\hs{0.3}F^*} \wt{\cF\cO}.\]

\noi We then apply $G_!$, which takes the clique $F^* X$ and evaluates the objects and the connecting isomorphisms in $B$.  The words of $\cF\cO$ are evaluated in $B$ via the tensor product, so for example the word $\big(b_1, ( b_2, b_3) \big)$ is evaluated as 
\[b_1 \otimes (b_2 \otimes b_3).\]  
The connecting isomorphisms are all braids/re-associations in $\cF\cO$ so are evaluated via braiding and coherence constraints in $B$, as depicted below:

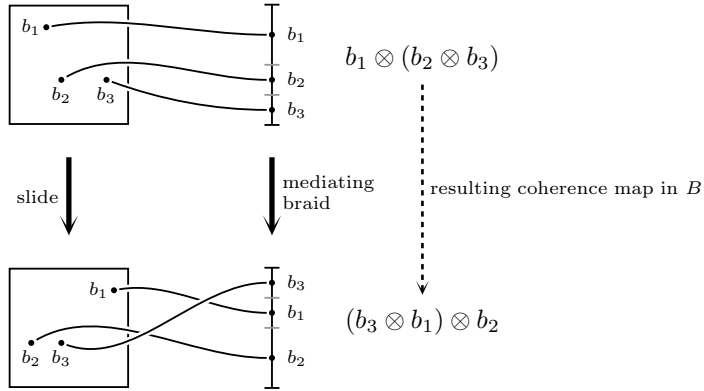
\[\renewcommand{\nowdot}{\psset{unit=1mm}\pscircle*(0,0){0.4}}
\ps(0,-2)(70,50)

\rput(0,35){

\psframe(0,0)(16,16)
\pnode(8,8){y1}

\rput(5,13){\rnode{b1}{\nowdot}
\rput(-2,0){\scr $b_1$}
}

\rput(7,6){\rnode{b2}{\nowdot}
\rput(0,-2){\scr $b_2$}
}
\rput(13,6){\rnode{b3}{\nowdot}
\rput(0,-2){\scr $b_3$}}

\rput(35,0){
\psline(0,0)(0,16)
\rput[B](20,8){\Rnode{xx}{$b_1 \otimes (b_2 \otimes b_3)$}}

\rput(0,16){\psline(-1,0)(1,0)}
\rput(0,0){\psline(-1,0)(1,0)}

\rput(0,8){\psline[linecolor=gray!80!white](-1,0)(1,0)}
\rput(0,4){\psline[linecolor=gray!80!white](-1,0)(1,0)}

\rput(0,12){\rnode{c1}{\nowdot}
\rput[l](2,0){\scr $b_1$}}

\rput(0,6){\rnode{c2}{\nowdot}
\rput[l](2,0){\scr $b_2$}}

\rput(0,2){\rnode{c3}{\nowdot}
\rput[l](2,0){\scr $b_3$}}


\pnode(0,8){z1}
}

\psset{border=1pt}

\nccurve[angleA=10,angleB=180]{b1}{c1}
\nccurve[angleA=30,angleB=180]{b2}{c2}
\nccurve[angleA=-20,angleB=180]{b3}{c3}


}


\rput(0,0){

\psframe(0,0)(16,16)
\pnode(8,8){y2}

\rput(14,13){\rnode{b1}{\nowdot}
\rput(-2,0){\scr $b_1$}
}

\rput(3,6){\rnode{b2}{\nowdot}
\rput(0,-2){\scr $b_2$}
}
\rput(7,6){\rnode{b3}{\nowdot}
\rput(0,-2){\scr $b_3$}}

\rput(35,0){
\psline(0,0)(0,16)
\rput[B](20,8){\Rnode{yy}{$(b_3 \otimes b_1) \otimes b_2$}}

\rput(0,16){\psline(-1,0)(1,0)}
\rput(0,0){\psline(-1,0)(1,0)}

\rput(0,12){\psline[linecolor=gray!80!white](-1,0)(1,0)}
\rput(0,8){\psline[linecolor=gray!80!white](-1,0)(1,0)}

\rput(0,10){\rnode{c1}{\nowdot}
\rput[l](2,0){\scr $b_1$}
}

\rput(0,4){\rnode{c2}{\nowdot}
\rput[l](2,0){\scr $b_2$}}

\rput(0,14){\rnode{c3}{\nowdot}
\rput[l](2,0){\scr $b_3$}}


\pnode(0,8){z2}
}

\psset{border=1pt}

\nccurve[angleA=10,angleB=180]{b1}{c1}
\nccurve[angleA=30,angleB=180]{b2}{c2}
\nccurve[angleA=-20,angleB=180]{b3}{c3}


}


\ncline[linewidth=2pt,nodesep=35pt]{->}{y1}{y2}
\nbput{\scr slide}

\ncline[linewidth=2pt,nodesep=35pt]{->}{z1}{z2}
\naput[npos=0.45]{\scr \parbox[l]{8em}{mediating\\ braid}}

\ncline[linewidth=1pt,nodesep=5pt, linestyle=dashed]{->}{xx}{yy}
\naput{\scr resulting coherence map in $B$}

\eps\]

\noi The evaluated words together with the evaluated connecting maps give the clique $\theta X$ in $B$.

%
%
%
%
%
%

We have been investigating the composite functor:
\[\theta := \hs{-0.3}\raisebox{0.1em}{\makebox[0.6em][l]{\scalebox{2.3}[1.3]{$\widetilde{\hs{3}}$}}}\Pi_1\big(C(I^2, \cO)\big) \mtmap{\hs{0.3}F^*} \wt{\cF\cO} \mtmap{G_!} \wt{B}\]
applied to a slide clique $X$, and we have finally arrived at the clique $\theta X$, which we are using to define morphisms in $\Sigma B$.  Recall that we defined
\[\Sigma B({X}, {Y}) = \wt{B}(\theta X, \theta Y)\]
\noi where ${X}$ and ${Y}$ are slide cliques of points in $I^2$ labelled by objects of $B$.   A morphism ${X} \tra {Y}$ in $\Sigma B$ is then a clique map in $\wt{B}$
\[\theta X \tra \theta Y\]
\noi thus it is represented by a morphism in $B$ between any chosen representatives of the cliques $\theta X$ and $\theta Y$. Picking a representative object of $\theta X$ consists of picking 

\begin{itemize}

\item one of the slide-equivalent configurations of $X$

\item a parenthesisation (and ordering) of the objects labelling $X$, and

\item a linearising braid.

\end{itemize}

\noi We do the same for $Y$, evaluate the parenthesisations in $B$, and take  any morphism $f$ between those resulting objects in $B$.  

The subtlety is that two representatives $f, f'$ are giving the same clique map if the square involving connecting isomorphisms commute.  To make this precise we need to introduce some more notation.  Write $\alpha(X)$ and $\alpha'(X)$ for the two parenthesisations of objects labelling $X$, with a connecting isomorphism that we will refer to as ``mediating braid'' although it involves a mediating braid as well as coherence constraints of $B$.  Similarly $\beta(Y), \beta'(Y)$.  Now suppose we have clique maps represented by the two maps below.
\[\ps(-10,0)(10,9)

\rput[B](-8,7){\Rnode{a1}{$\alpha(X)$}}
\rput[B](8,7){\Rnode{a2}{$\beta(Y)$}}
\rput[B](-8,0){\Rnode{b1}{$\alpha'(X)$}}
\rput[B](8,0){\Rnode{b2}{$\beta'(Y)$}}

\ncline{->}{a1}{a2} \naput{\scr $f$}
\ncline{->}{b1}{b2} \naput{\scr $f'$}

\eps\]
%
Then $f$ and $f'$ represent the same clique map $\theta X \tra \theta Y$ if the following square commutes
\[
\psset{unit=0.1cm,labelsep=3pt,nodesep=3pt}
\pspicture(0,-4)(20,22)


\rput(0,18){\rnode{a2}{$\alpha(X)$}}  
\rput(20,18){\rnode{a1}{$\beta(Y)$}}  
\rput(0,0){\rnode{a4}{$\alpha'(X)$}}  

\rput(20,0){\rnode{a3}{$\beta'(Y)$}}  

\ncline{->}{a2}{a1} \naput{{\scriptsize $f$}} 
\nbput[labelsep=2pt]{\scr $\sim$}
\ncline{->}{a4}{a3} \nbput{{\scriptsize $f'$}} 
\naput[labelsep=1pt]{\scr $\sim$}
\ncline{->}{a1}{a3} \naput{{\scriptsize \sf{mediating braid}}} 
\nbput{\rotatebox{90}{\scr $\sim$}}

\ncline{->}{a2}{a4} \nbput{{\scriptsize \sf{mediating braid}}}
\naput{\rotatebox{90}{\scr $\sim$}}
\endpspicture
\]

For example, the following are three representatives of the same clique $\theta X$, as the configurations of points only differ by a slide.  The maps in $B$ as shown are the connecting isomorphisms for the clique, and thus each one also represents the identity map on this clique.  Here $\sigma$ is the braiding in $B$, and we see that it can represent the identity map on this clique, where it is just ``compensating'' for a crossing in the linearising braid.


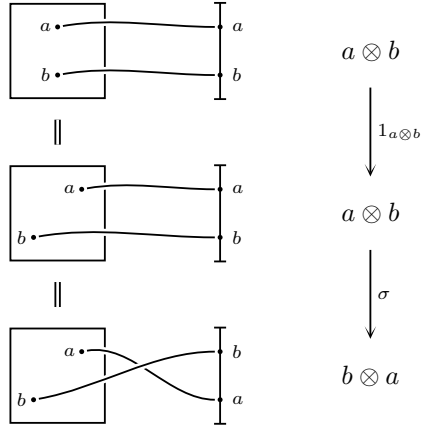
\[\renewcommand{\nowdot}{\psset{unit=1.1mm}\pscircle*(0,0){0.3}}
\psset{unit=0.8mm}
\ps(0,90)(60,159)


\rput(0,144){

\psframe(0,0)(16,16)
\pnode(8,8){x1}

\rput(8,12){\rnode{a1}{\nowdot}
\rput(-2,0){\scr $a$}
}

\rput(8,4){\rnode{b1}{\nowdot}
\rput(-2,0){\scr $b$}
}


\rput(35,0){
\psline(0,0)(0,16)
\rput(0,16){\psline(-1,0)(1,0)}
\rput(0,0){\psline(-1,0)(1,0)}

\rput(25,8){\rnode{z1}{$a \otimes b$}}

\rput(0,12){\rnode{a2}{\nowdot}
\rput[l](2,0){\scr $a$}}

\rput(0,4){\rnode{b2}{\nowdot}
\rput[l](2,0){\scr $b$}}
}

\psset{border=1pt}

\nccurve[angleA=10,angleB=180]{a1}{a2}
\nccurve[angleA=10,angleB=180]{b1}{b2}


}


\rput(0,117){

\psframe(0,0)(16,16)
\pnode(8,8){x2}

\rput(12,12){\rnode{a1}{\nowdot}
\rput(-2,0){\scr $a$}
}

\rput(4,4){\rnode{b1}{\nowdot}
\rput(-2,0){\scr $b$}
}


\rput(35,0){
\psline(0,0)(0,16)
\rput(0,16){\psline(-1,0)(1,0)}
\rput(0,0){\psline(-1,0)(1,0)}

\rput(25,8){\rnode{z2}{$a \otimes b$}}


\rput(0,12){\rnode{a2}{\nowdot}
\rput[l](2,0){\scr $a$}}

\rput(0,4){\rnode{b2}{\nowdot}
\rput[l](2,0){\scr $b$}}
}

\psset{border=1pt}

\nccurve[angleA=10,angleB=180]{a1}{a2}
\nccurve[angleA=10,angleB=180]{b1}{b2}


}


\rput(0,90){
\psframe(0,0)(16,16)
\pnode(8,8){x3}

\rput(12,12){\rnode{a1}{\nowdot}
\rput(-2,0){\scr $a$}
}

\rput(4,4){\rnode{b1}{\nowdot}
\rput(-2,0){\scr $b$}
}


\rput(35,0){
\psline(0,0)(0,16)
\rput(0,16){\psline(-1,0)(1,0)}
\rput(0,0){\psline(-1,0)(1,0)}

\rput(25,8){\rnode{z3}{$b \otimes a$}}


\rput(0,4){\rnode{a2}{\nowdot}
\rput[l](2,0){\scr $a$}}

\rput(0,12){\rnode{b2}{\nowdot}
\rput[l](2,0){\scr $b$}}
}

\psset{border=1pt}

\nccurve[angleA=10,angleB=180]{a1}{a2}
\nccurve[angleA=10,angleB=180]{b1}{b2}


}

\ncline[nodesep=10pt,offset=0pt]{->}{z1}{z2}\naput{\scr $1_{a \otimes b}$}
\ncline[nodesep=10pt,offset=0pt]{->}{z2}{z3}\naput{\scr $\sigma$}

\ncline[doubleline=true,nodesep=26pt]{x1}{x2}
\ncline[doubleline=true,nodesep=26pt]{x2}{x3} 

\eps\]

By contrast the following are not in the same clique as the configurations differ by vertical coordinates, not just horizontal ones.  Thus the identity $1_{b\otimes a}$ as shown does not represent an identity clique map, although it is an identity in $B$.


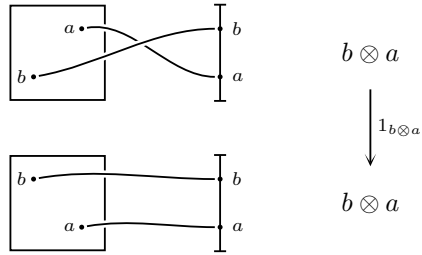
\[\renewcommand{\nowdot}{\psset{unit=1.1mm}\pscircle*(0,0){0.3}}
\psset{unit=0.8mm}
\ps(0,58)(60,100)

\rput(0,85){
\psframe(0,0)(16,16)
\pnode(8,8){x3}

\rput(12,12){\rnode{a1}{\nowdot}
\rput(-2,0){\scr $a$}
}

\rput(4,4){\rnode{b1}{\nowdot}
\rput(-2,0){\scr $b$}
}


\rput(35,0){
\psline(0,0)(0,16)
\rput(0,16){\psline(-1,0)(1,0)}
\rput(0,0){\psline(-1,0)(1,0)}

\rput(25,8){\rnode{z3}{$b \otimes a$}}


\rput(0,4){\rnode{a2}{\nowdot}
\rput[l](2,0){\scr $a$}}

\rput(0,12){\rnode{b2}{\nowdot}
\rput[l](2,0){\scr $b$}}
}

\psset{border=1pt}

\nccurve[angleA=10,angleB=180]{a1}{a2}
\nccurve[angleA=10,angleB=180]{b1}{b2}


}

\rput(0,60){

\psframe(0,0)(16,16)
\pnode(8,8){x4}

\rput(12,4){\rnode{a1}{\nowdot}
\rput(-2,0){\scr $a$}
}

\rput(4,12){\rnode{b1}{\nowdot}
\rput(-2,0){\scr $b$}
}


\rput(35,0){
\psline(0,0)(0,16)
\rput(0,16){\psline(-1,0)(1,0)}
\rput(0,0){\psline(-1,0)(1,0)}

\rput(25,8){\rnode{z4}{$b \otimes a$}}

\rput(0,4){\rnode{a2}{\nowdot}
\rput[l](2,0){\scr $a$}}

\rput(0,12){\rnode{b2}{\nowdot}
\rput[l](2,0){\scr $b$}}
}

\psset{border=1pt}

\nccurve[angleA=10,angleB=180]{a1}{a2}
\nccurve[angleA=10,angleB=180]{b1}{b2}


}

\ncline[nodesep=10pt,offset=0pt]{->}{z3}{z4}\naput{\scr $1_{b \otimes a}$}


\eps\]

\noi Both of these scenarios will be key later.  At this point it is evident that we need a consistent notation for calculating with these morphisms, which we will now establish.

\subsection{Notational conventions for $\Sigma B$}

In this section we will lay out our method and notation for working with $\Sigma B$.  We will write objects as below; in the following diagram we have drawn dashed horizontal lines to remind us that this is a slide clique, but most of the time we will not include those lines.

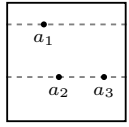
\[\renewcommand{\nowdot}{\psset{unit=1mm}\pscircle*(0,0){0.4}}
\ps(3,-1)(19,16)

\rput(0,0){

\psline[linestyle=dashed,linecolor=gray](0,13)(16,13)
\psline[linestyle=dashed,linecolor=gray](0,6)(16,6)

\psframe(0,0)(16,16)
\pnode(8,8){x0}

\rput(5,13){\rnode{b1}{\nowdot}
\rput(0,-2){\scr $a_1$}
}

\rput(7,6){\rnode{b2}{\nowdot}
\rput(0,-2){\scr $a_2$}
}
\rput(13,6){\rnode{b3}{\nowdot}
\rput(0,-2){\scr $a_3$}}

}

\eps
\]

\noi A morphism in $\Sigma B$ is a clique map, and we will depict a representative of it between particular objects as below; here for the source and target in $\Sigma B$ we depict the representing configuration $X$, the linearising braid $\gamma$, and the configuration $Fc$ together with its realisation as a parenthesised word in $B$.

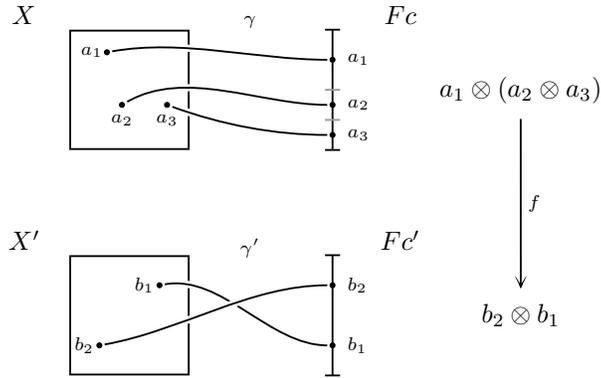
\[\renewcommand{\nowdot}{\psset{unit=1mm}\pscircle*(0,0){0.4}}
\ps(0,-3)(60,50)

\rput(0,30){

\rput(24,17){\scr $\gamma$}

\psframe(0,0)(16,16)
\pnode(8,8){x0}

\rput(5,13){\rnode{b1}{\nowdot}
\rput(-2,0){\scr $a_1$}
}

\rput(7,6){\rnode{b2}{\nowdot}
\rput(0,-2){\scr $a_2$}
}
\rput(13,6){\rnode{b3}{\nowdot}
\rput(0,-2){\scr $a_3$}}

\rput(35,0){
\psline(0,0)(0,16)
\pnode(0,8){Fc0}
\rput(0,16){\psline(-1,0)(1,0)}
\rput(0,0){\psline(-1,0)(1,0)}

\rput(0,8){\psline[linecolor=gray!80!white](-1,0)(1,0)}
\rput(0,4){\psline[linecolor=gray!80!white](-1,0)(1,0)}

\rput(0,12){\rnode{c1}{\nowdot}
\rput[l](2,0){\scr $a_1$}}

\rput(0,6){\rnode{c2}{\nowdot}
\rput[l](2,0){\scr $a_2$}}

\rput(0,2){\rnode{c3}{\nowdot}
\rput[l](2,0){\scr $a_3$}}

\rput(9,18){$Fc$}

\rput(25,8){\rnode{za}{$a_1 \otimes (a_2 \otimes a_3)$}}

}

\psset{border=1pt}

\nccurve[angleA=10,angleB=180]{b1}{c1}
\nccurve[angleA=30,angleB=180]{b2}{c2}
\nccurve[angleA=-20,angleB=180]{b3}{c3}

\rput(-6,18){$X$}

}


\rput(0,0){

\rput(24,17){\scr $\gamma'$}

\psframe(0,0)(16,16)
\pnode(8,8){x3}

\rput(12,12){\rnode{a1}{\nowdot}
\rput(-2,0){\scr $b_1$}
}

\rput(4,4){\rnode{b1}{\nowdot}
\rput(-2,0){\scr $b_2$}
}


\rput(35,0){
\psline(0,0)(0,16)
\rput(0,16){\psline(-1,0)(1,0)}
\rput(0,0){\psline(-1,0)(1,0)}

\rput(25,8){\rnode{zb}{$b_2 \otimes b_1$}}


\rput(0,4){\rnode{a2}{\nowdot}
\rput[l](2,0){\scr $b_1$}}

\rput(9,18){$Fc'$}

\rput(0,12){\rnode{b2}{\nowdot}
\rput[l](2,0){\scr $b_2$}}
}

\psset{border=1pt}

\nccurve[angleA=10,angleB=180]{a1}{a2}
\nccurve[angleA=10,angleB=180]{b1}{b2}

\rput(-6,18){$X'$}


\ncline[nodesep=6pt]{->}{za}{zb} \naput{\scr $f$}

}



{\psset{nodesep=35pt, linecolor=gray!80!white}

\ncline[offset=8pt]{x0}{x1}
\ncline{x0}{x1}
\ncline[offset=-8pt]{x0}{x1}


}

\eps\]


We will tend to suppress the distinction between different members of a slide clique, as those play less role in the ensuing calculations on morphisms of $\Sigma B$; their main role is to ensure that the horizontal tensor product is strict on objects. 

When less specificity is needed we will represent general configurations and parenthesisations as shown below. Here $X$ represents a configuration of points labelled by objects of $B$, and $\langle X \rangle$  represents a particular parenthesisation of the objects in question. Where we need to refer to two different parenthesisations of the same objects, we will use more specific notation such as $\alpha_1(X)$, $\alpha_2(X)$. The single (thicker) strand represents a chosen linearising braid; this notation will later allow us to indicate a crossing of whole braids (as in \cite{jk1}):

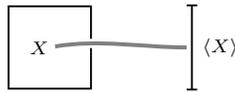
\[\renewcommand{\nowdot}{\psset{unit=1mm}\pscircle*(0,0){0.4}}
\psset{unit=0.7mm}
\ps(0,0)(40,18)

\rput(0,0){

\psframe(0,0)(16,16)

\rput(8,8){\rnode{a1}{}
\rput(-2,0){\scr $X$}
}

\rput(35,0){
\psline(0,0)(0,16)
\rput(0,16){\psline(-1,0)(1,0)}
\rput(0,0){\psline(-1,0)(1,0)}

\rput(0,8){\rnode{a2}{}
\rput[l](2,0){\scr $\langle X \rangle$}}

}

\psset{border=1pt}

\nccurve[angleA=10,angleB=180, linecolor=gray, linewidth=1.5pt]{a1}{a2}

}

\eps\]


\noi A morphism will then be depicted via a particular representative of the clique map as follows:

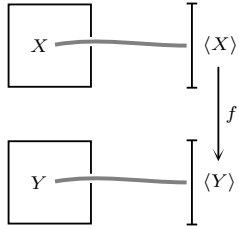
\[\renewcommand{\nowdot}{\psset{unit=1mm}\pscircle*(0,0){0.4}}
\psset{unit=0.7mm}
\ps(0,0)(40,40)

\rput(0,26){

\psframe(0,0)(16,16)

\rput(8,8){\rnode{a1}{}
\rput(-2,0){\scr $X$}
}

\rput(35,0){
\psline(0,0)(0,16)
\rput(0,16){\psline(-1,0)(1,0)}
\rput(0,0){\psline(-1,0)(1,0)}

\rput(0,8){\rnode{aa2}{}
\rput[l](2,0){\scr $\langle X \rangle$}}

}

\psset{border=1pt}

\nccurve[angleA=10,angleB=180,linecolor=gray, linewidth=1.5pt]{a1}{aa2}

}


\rput(0,0){

\psframe(0,0)(16,16)

\rput(8,8){\rnode{a1}{}
\rput(-2,0){\scr $Y$}
}

\rput(35,0){
\psline(0,0)(0,16)
\rput(0,16){\psline(-1,0)(1,0)}
\rput(0,0){\psline(-1,0)(1,0)}

\rput(0,8){\rnode{ba2}{}
\rput[l](2,0){\scr $\langle Y \rangle$}}

}

\psset{border=1pt}

\nccurve[angleA=10,angleB=180,linecolor=gray, linewidth=1.5pt]{a1}{ba2}

}

\ncline[nodesep=8pt,offset=10pt]{->}{aa2}{ba2}\naput{\scr $f$}

%
%


\eps\]

When checking that two morphisms are the same we may need to take into account that different representing objects of the cliques in question have been used.  For example we may want to check that the following are two representatives of the same morphism:

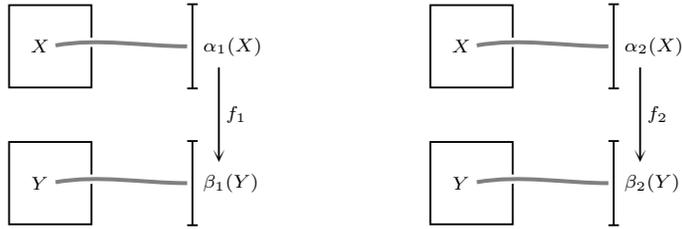
\[\renewcommand{\nowdot}{\psset{unit=1mm}\pscircle*(0,0){0.4}}
\psset{unit=0.7mm}
\ps(0,-1)(120,42)



\rput(0,0){

\rput(0,26){

\psframe(0,0)(16,16)

\rput(8,8){\rnode{a1}{}
\rput(-2,0){\scr $X$}
}

\rput(35,0){
\psline(0,0)(0,16)
\rput(0,16){\psline(-1,0)(1,0)}
\rput(0,0){\psline(-1,0)(1,0)}

\rput(0,8){\rnode{aa2}{}
\rput[l](2,0){\scr $\alpha_1(X)$}}

}

\psset{border=1pt}

\nccurve[angleA=10,angleB=180,linecolor=gray, linewidth=1.5pt]{a1}{aa2}

}


\rput(0,0){

\psframe(0,0)(16,16)

\rput(8,8){\rnode{a1}{}
\rput(-2,0){\scr $Y$}
}

\rput(35,0){
\psline(0,0)(0,16)
\rput(0,16){\psline(-1,0)(1,0)}
\rput(0,0){\psline(-1,0)(1,0)}

\rput(0,8){\rnode{ba2}{}
\rput[l](2,0){\scr $\beta_1(Y)$}}

}

\psset{border=1pt}

\nccurve[angleA=10,angleB=180,linecolor=gray, linewidth=1.5pt]{a1}{ba2}

}

\ncline[nodesep=8pt,offset=10pt]{->}{aa2}{ba2}\naput{\scr $f_1$}

}



\rput(80,0){

\rput(0,26){

\psframe(0,0)(16,16)

\rput(8,8){\rnode{a1}{}
\rput(-2,0){\scr $X$}
}

\rput(35,0){
\psline(0,0)(0,16)
\rput(0,16){\psline(-1,0)(1,0)}
\rput(0,0){\psline(-1,0)(1,0)}

\rput(0,8){\rnode{aa2}{}
\rput[l](2,0){\scr $\alpha_2(X)$}}

}

\psset{border=1pt}

\nccurve[angleA=10,angleB=180,linecolor=gray, linewidth=1.5pt]{a1}{aa2}

}


\rput(0,0){

\psframe(0,0)(16,16)

\rput(8,8){\rnode{a1}{}
\rput(-2,0){\scr $Y$}
}

\rput(35,0){
\psline(0,0)(0,16)
\rput(0,16){\psline(-1,0)(1,0)}
\rput(0,0){\psline(-1,0)(1,0)}

\rput(0,8){\rnode{ba2}{}
\rput[l](2,0){\scr $\beta_2(Y)$}}

}

\psset{border=1pt}

\nccurve[angleA=10,angleB=180,linecolor=gray, linewidth=1.5pt]{a1}{ba2}

}

\ncline[nodesep=8pt,offset=10pt]{->}{aa2}{ba2}\naput{\scr $f_2$}

}

\eps\]

\noi To check this, we need to check that the following diagram commutes in $B$, where the top and bottom morphisms are the unique connecting isomorphisms induced from the respective linearising braids:

\[
\psset{unit=0.1cm,labelsep=3pt,nodesep=3pt}
\pspicture(-10,-4)(10,24)



\rput(-13,20){\rnode{a1}{$\alpha_1(X)$}}  
\rput(13,20){\rnode{a2}{$\alpha_2(X)$}}  
\rput(-13,0){\rnode{a3}{$\beta_1(Y)$}}  
\rput(13,0){\rnode{a4}{$\beta_2(Y)$}}  

\ncline{->}{a1}{a2} \naput{{\scriptsize \sf{coherence}}} 
\ncline{->}{a3}{a4} \nbput{{\scriptsize \sf{coherence}}} 
\ncline{->}{a1}{a3} \nbput{{\scriptsize $f_1$}} 
\ncline{->}{a2}{a4} \naput{{\scriptsize $f_2$}} 

\endpspicture
\]

\noi In practice where there is no ambiguity we will abuse notation and write a morphism $f$ in $\Sigma B$ and a chosen representative of it in the same way.

\subsection{Tensor products on $\Sigma B$}
\label{sigmatensor}

We now show that $\Sigma B$ has the structure of a doubly-degenerate \bicats-category. By Proposition~\ref{dd-char} we need to exhibit a strict ``horizontal'' tensor product, a weak ``vertical'' tensor product, and strict interchange.


We define the tensor products on objects of $\Sigma B$ as follows. 

\begin{itemize}

\item Weak vertical tensor product is given by vertical stacking of boxes, and then equal reparametrisation, as described in Section~\ref{vertmonconfigs}.

\item Strict horizontal tensor product is given by horizontal stacking of boxes as described in Section~\ref{slidemaps}; reparametrisation then doesn't matter as it is absorbed into the slide cliques.

\item The unit for both tensor products is the empty box.


\end{itemize}

\noi Now on morphisms the idea is very similar to \cite{jk1}: we set both the vertical and horizontal tensor products of morphisms to be the tensor product in $\wt{B}$.  We will now make this more precise.

Consider morphisms in $\Sigma B$
\[X_1 \mtmap{f_1} Y_1\]
\[X_2 \mtmap{f_2} Y_2\]
We will define the tensor products by picking representatives of $f_1$ and $f_2$ and using them to specify a representative of the clique map that is the tensor product. So consider any representatives

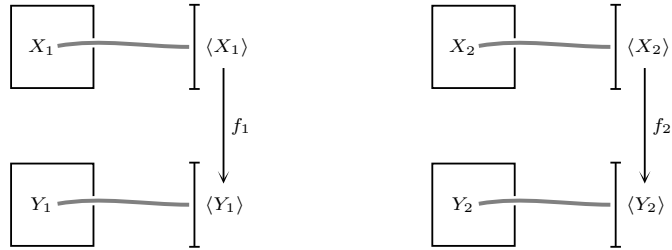
\[\renewcommand{\nowdot}{\psset{unit=1mm}\pscircle*(0,0){0.4}}
\psset{unit=0.7mm}
\ps(0,0)(120,50)



\rput(0,0){

\rput(0,30){

\psframe(0,0)(16,16)

\rput(8,8){\rnode{a1}{}
\rput(-2,0){\scr $X_1$}
}

\rput(35,0){
\psline(0,0)(0,16)
\rput(0,16){\psline(-1,0)(1,0)}
\rput(0,0){\psline(-1,0)(1,0)}

\rput(0,8){\rnode{aa2}{}
\rput[l](2,0){\scr $\langle X_1 \rangle$}}

}

\psset{border=1pt}

\nccurve[angleA=10,angleB=180,linecolor=gray, linewidth=1.5pt]{a1}{aa2}

}


\rput(0,0){

\psframe(0,0)(16,16)

\rput(8,8){\rnode{a1}{}
\rput(-2,0){\scr $Y_1$}
}

\rput(35,0){
\psline(0,0)(0,16)
\rput(0,16){\psline(-1,0)(1,0)}
\rput(0,0){\psline(-1,0)(1,0)}

\rput(0,8){\rnode{ba2}{}
\rput[l](2,0){\scr $\langle Y_1 \rangle$}}

}

\psset{border=1pt}

\nccurve[angleA=10,angleB=180,linecolor=gray, linewidth=1.5pt]{a1}{ba2}

}

\ncline[nodesep=8pt,offset=11pt]{->}{aa2}{ba2}\naput{\scr $f_1$}

}



\rput(80,0){

\rput(0,30){

\psframe(0,0)(16,16)

\rput(8,8){\rnode{a1}{}
\rput(-2,0){\scr $X_2$}
}

\rput(35,0){
\psline(0,0)(0,16)
\rput(0,16){\psline(-1,0)(1,0)}
\rput(0,0){\psline(-1,0)(1,0)}

\rput(0,8){\rnode{aa2}{}
\rput[l](2,0){\scr $\langle X_2 \rangle$}}

}

\psset{border=1pt}

\nccurve[angleA=10,angleB=180,linecolor=gray, linewidth=1.5pt]{a1}{aa2}

}


\rput(0,0){

\psframe(0,0)(16,16)

\rput(8,8){\rnode{a1}{}
\rput(-2,0){\scr $Y_2$}
}

\rput(35,0){
\psline(0,0)(0,16)
\rput(0,16){\psline(-1,0)(1,0)}
\rput(0,0){\psline(-1,0)(1,0)}

\rput(0,8){\rnode{ba2}{}
\rput[l](2,0){\scr $\langle Y_2 \rangle$}}

}

\psset{border=1pt}

\nccurve[angleA=10,angleB=180,linecolor=gray, linewidth=1.5pt]{a1}{ba2}

}

\ncline[nodesep=8pt,offset=11pt]{->}{aa2}{ba2}\naput{\scr $f_2$}

}

\eps\]

\noi \noi First we define the vertical tensor product
\[\frac{X_1}{X_2} \mtmap{\frac{f_1}{f_2}} \frac{Y_1}{Y_2}\]

\noi to be the map represented by $f_1 \otimes f_2$ with the following linearising braids:

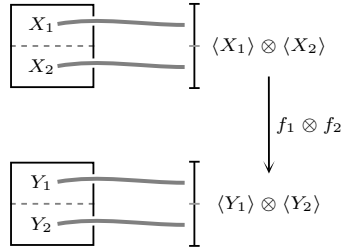
\[\renewcommand{\nowdot}{\psset{unit=1mm}\pscircle*(0,0){0.4}}
\psset{unit=0.7mm}
\ps(0,0)(60,48)

\rput(0,30){

\psline[linestyle=dashed, linecolor=gray](0,8)(16,8)
\psframe(0,0)(16,16)

\rput(8,12){\rnode{a1}{}
\rput(-2,0){\scr $X_1$}
}

\rput(8,4){\rnode{b1}{}
\rput(-2,0){\scr $X_2$}
}

\rput(35,0){
\psline(0,0)(0,16)
\rput(0,16){\psline(-1,0)(1,0)}
\rput(0,0){\psline(-1,0)(1,0)}
\rput(0,8){\psline[linecolor=gray!80!white](-1,0)(1,0)}

\rput(0,12){\rnode{aa2}{}
}

\rput(0,4){\rnode{bb2}{}
}

\rput(15,8){\rnode{aaa2}{\scr $\langle X_1 \rangle \otimes \langle X_2 \rangle$}
}

}

\psset{border=1pt}

\nccurve[angleA=10,angleB=180,linecolor=gray, linewidth=1.5pt]{a1}{aa2}
\nccurve[angleA=10,angleB=180,linecolor=gray, linewidth=1.5pt]{b1}{bb2}

}


\rput(0,0){

\psline[linestyle=dashed, linecolor=gray](0,8)(16,8)
\psframe(0,0)(16,16)

\rput(8,12){\rnode{a1}{}
\rput(-2,0){\scr $Y_1$}
}

\rput(8,4){\rnode{b1}{}
\rput(-2,0){\scr $Y_2$}
}

\rput(35,0){
\psline(0,0)(0,16)
\rput(0,16){\psline(-1,0)(1,0)}
\rput(0,0){\psline(-1,0)(1,0)}
\rput(0,8){\psline[linecolor=gray!80!white](-1,0)(1,0)}

\rput(0,12){\rnode{aa2}{}
}

\rput(0,4){\rnode{bb2}{}
}

\rput(15,8){\rnode{bbb2}{\scr $\langle Y_1 \rangle \otimes \langle Y_2 \rangle$}
}

}

\psset{border=1pt}

\nccurve[angleA=10,angleB=180,linecolor=gray, linewidth=1.5pt]{a1}{aa2}
\nccurve[angleA=10,angleB=180,linecolor=gray, linewidth=1.5pt]{b1}{bb2}

}

\ncline[nodesep=8pt,offset=0pt]{->}{aaa2}{bbb2}\naput{\scr $f_1 \otimes f_2$}

\eps\]

\noi This is well-defined: if we start with different representatives of $f_1$ and $f_2$ the resulting tensor product gives the same clique map as the square we need to check is a tensor product of individual squares for the clique maps $f_1$ and $f_2$.

Next we define the horizontal tensor product
\[{X_1}|{X_2} \mtmap{{f_1}|{f_2}} {Y_1}|{Y_2}\]

\noi to be the map represented by $f_1 \otimes f_2$, now with respect to the following linearising braids:

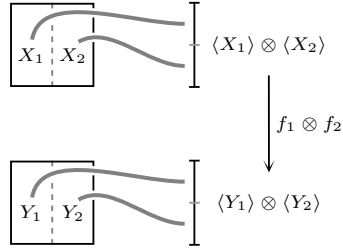
\[\renewcommand{\nowdot}{\psset{unit=1mm}\pscircle*(0,0){0.4}}
\psset{unit=0.7mm}
\ps(0,0)(60,46)

\rput(0,30){

\psline[linestyle=dashed, linecolor=gray](8,0)(8,16)
\psframe(0,0)(16,16)

\rput(4,8){\rnode{a1}{}
\rput(0,-2){\scr $X_1$}
}

\rput(12,8){\rnode{b1}{}
\rput(0,-2){\scr $X_2$}
}

\rput(35,0){
\psline(0,0)(0,16)
\rput(0,16){\psline(-1,0)(1,0)}
\rput(0,0){\psline(-1,0)(1,0)}
\rput(0,8){\psline[linecolor=gray!80!white](-1,0)(1,0)}

\rput(0,12){\rnode{aa2}{}
}

\rput(0,4){\rnode{bb2}{}
}

\rput(15,8){\rnode{aaa2}{\scr $\langle X_1 \rangle \otimes \langle X_2 \rangle$}
}

}

\psset{border=1pt}

\nccurve[angleA=80,angleB=180,linecolor=gray, linewidth=1.5pt]{a1}{aa2}
\nccurve[angleA=30,angleB=180,linecolor=gray, linewidth=1.5pt]{b1}{bb2}

}


\rput(0,0){

\psline[linestyle=dashed, linecolor=gray](8,0)(8,16)
\psframe(0,0)(16,16)

\rput(4,8){\rnode{a1}{}
\rput(0,-2){\scr $Y_1$}
}

\rput(12,8){\rnode{b1}{}
\rput(0,-2){\scr $Y_2$}
}

\rput(35,0){
\psline(0,0)(0,16)
\rput(0,16){\psline(-1,0)(1,0)}
\rput(0,0){\psline(-1,0)(1,0)}
\rput(0,8){\psline[linecolor=gray!80!white](-1,0)(1,0)}

\rput(0,12){\rnode{aa2}{}
}

\rput(0,4){\rnode{bb2}{}
}

\rput(15,8){\rnode{bbb2}{\scr $\langle Y_1 \rangle \otimes \langle Y_2 \rangle$}
}

}

\psset{border=1pt}

\nccurve[angleA=80,angleB=180,linecolor=gray, linewidth=1.5pt]{a1}{aa2}
\nccurve[angleA=30,angleB=180,linecolor=gray, linewidth=1.5pt]{b1}{bb2}

}

\ncline[nodesep=8pt,offset=0pt]{->}{aaa2}{bbb2}\naput{\scr $f_1 \otimes f_2$}

\eps\]

\noi Note that these linearising braids can be expressed formally in terms of units and interchange: the original linearising braids are stacked horizontally and then post-composed with the following maps.

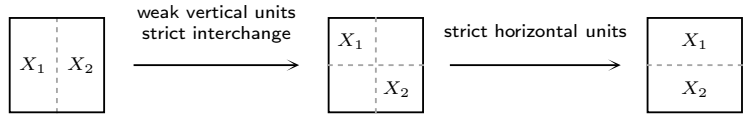
\[\psset{unit=0.8mm}
\pspicture(0,-10)(90,12)

\rput(-8,0){\rnode{a1}{

\psline[linestyle=dashed, linecolor=gray!70!white](0,-8)(0,8)
\psframe(-8,-8)(8,8)

\rput(-4,0){\rnode{aa1}{}
\rput(0,0){\scr $X_1$}
}

\rput(4,0){\rnode{bb1}{}
\rput(0,0){\scr $X_2$}}

}}


\rput(45,0){\rnode{a2}{\psframe(-8,-8)(8,8)

\rput(-3.4,4){\rnode{aa2}{\scr $X_1$}
}

\rput(4.2,-4){\rnode{bb2}{\scr $X_2$}
}

\psline[linestyle=dashed, linecolor=gray!70!white](-8,0)(8,0)
\psline[linestyle=dashed, linecolor=gray!70!white](0,-8)(0,8)}}




\rput(98,0){\rnode{a3}{\psframe(-8,-8)(8,8)
\rput(0,4){\rnode{aa3}{}
\rput(0,0){\scr $X_1$}
}
\rput(0,-4){\rnode{bb3}{}
\rput(0,0){\scr $X_2$}}
\psline[linestyle=dashed, linecolor=gray!70!white](-8,0)(8,0)
%
}}

\psset{nodesep=29pt, labelsep=4pt}

\ncline[arrows=->]{a1}{a2} \naput[labelsep=-2pt]{\scr\sf \parbox{10em}{\bc weak vertical units\\ strict interchange\ec}}
\ncline[arrows=->,nodesep=28pt]{a2}{a3} \naput[labelsep=10pt]{\scr\sf strict horizontal units}

\psset{border=1pt, linecolor=gray, linewidth=1.5pt, nodesep=2pt}

%
%

\endpspicture\]

\noi This is well-defined for the same reason as the vertical tensor product.

We now need to check associativity and unit constraints. For the weak vertical tensor product everything is immediate as the constraints come from those in $B$.  For the horizontal tensor product the constraints also come from those in $B$ but we now need the tensor product to be strict; this follows because those constraints in $B$ actually represent identity clique maps.  In more detail, first note that for objects $X, Y, Z$ in $\Sigma B$, the objects
$(X|Y)|Z$ and $X|(Y|Z)$ are in the same slide clique, as the different parenthesisation just means that each labelled point is repositioned horizontally, as indicated in the following diagram:

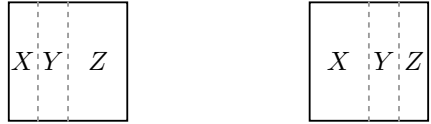
\[\renewcommand{\nowdot}{\psset{unit=1mm}\pscircle*(0,0){0.4}}
\ps(0,24)(50,42)

\rput(0,25){

\psframe(0,0)(16,16)
\rput(4,0){\psline[linestyle=dashed, linecolor=gray!80!white](0,0)(0,16)}
\rput(8,0){\psline[linestyle=dashed, linecolor=gray!80!white](0,0)(0,16)}

\rput(2, 8){$X$}
\rput(6, 8){$Y$}
\rput(12, 8){$Z$}

}

\rput(40,25){

\psframe(0,0)(16,16)
\rput(8,0){\psline[linestyle=dashed, linecolor=gray!80!white](0,0)(0,16)}
\rput(12,0){\psline[linestyle=dashed, linecolor=gray!80!white](0,0)(0,16)}

\rput(4, 8){$X$}
\rput(10, 8){$Y$}
\rput(14, 8){$Z$}

}

\eps\]

Now consider morphisms
\[\ps(-10,-8)(10,8)

\rput[B](-8,6){\Rnode{a1}{$X_1$}}
\rput[B](8,6){\Rnode{a2}{$X_2$}}
\rput[B](-8,0){\Rnode{b1}{$Y_1$}}
\rput[B](8,0){\Rnode{b2}{$Y_2$}}
\rput[B](-8,-6){\Rnode{c1}{$Z_1$}}
\rput[B](8,-6){\Rnode{c2}{$Z_2$}}

\ncline{->}{a1}{a2} \naput{\scr $f$}
\ncline{->}{b1}{b2} \naput{\scr $g$}
\ncline{->}{c1}{c2} \naput{\scr $h$}

\eps\]

\noi According to our above definition $(f|g)|h$ is represented by
\[(X_1 \otimes Y_1) \otimes Z_1 \vltmap{(f \otimes g) \otimes h} (X_2 \otimes Y_2) \otimes Z_2\]
whereas $f|(g|h)$ is represented by
\[X_1 \otimes (Y_1 \otimes Z_1) \vltmap{f \otimes (g \otimes h)} X_2 \otimes (Y_2 \otimes Z_2).\]
To show that these represent the same clique map we need to check that the square involving connecting isomorphisms commutes, but in this case the connecting isomorphisms are the associators in $B$, so the square does indeed commute.  The argument for the unit constraints is analogous.

\subsection{Interchange for $\Sigma B$}
\label{sigmainterchange}

Now we show interchange is strict.  On objects it is clear as, either way round, we get


\renewcommand{\ul}{}

\[\psset{unit=0.9mm}
\ps(-10,-10)(10,7)

\psframe(-8,-8)(8,8)
\rput(0,-8){\psline[linestyle=dashed, linecolor=gray!80!white](0,0)(0,16)}
\rput(-8,0){\psline[linestyle=dashed, linecolor=gray!80!white](0,0)(16,0)}

\rput[B](-4,3.5){\scr $a$}
\rput[B](4,3.5){\scr $b$}
\rput[B](-4,-4.5){\scr $c$}
\rput[B](4,-4.5){\scr $d$}

\eps\]

On morphisms this proceeds exactly as in \cite{jk1}.  We need to compare the following two possible ways of tensoring morphisms $f, g, h, j$:

\[\setlength{\arraycolsep}{1.7pt}
\renewcommand{\arraystretch}{1.3}
\begin{array}{ccc}
f & \psline(0,-1)(0,3) & g \\
\hline 
h & \psline(0,-1)(0,3) & j
\end{array}
\hs{1}
\mbox{and}
\hs{1}
\setlength{\arraycolsep}{4pt}
\begin{array}{c|c}
f & g \\[-8pt]
\psline(-2.0,0.5)(1.5,0.5) & \psline(-1.5,0.5)(1.8,0.5) \\[-6pt]
h & j 
\end{array}
\]
each of which gives a morphism in $\Sigma B$ as shown below.


\[\ps(0,-8)(100,6)

%
%

\rput(30,0){\rnode{a}{

\psframe(-8,-8)(8,8)
\rput(0,-8){\psline[linestyle=dashed, linecolor=gray!80!white](0,0)(0,16)}
\rput(-8,0){\psline[linestyle=dashed, linecolor=gray!80!white](0,0)(16,0)}

\rput[B](-4,3.5){\scr $a_1$}
\rput[B](4,3.5){\scr $b_1$}
\rput[B](-4,-4.5){\scr $c_1$}
\rput[B](4,-4.5){\scr $d_1$}
}
}

\rput(70,0){\rnode{b}{

\psframe(-8,-8)(8,8)
\rput(0,-8){\psline[linestyle=dashed, linecolor=gray!80!white](0,0)(0,16)}
\rput(-8,0){\psline[linestyle=dashed, linecolor=gray!80!white](0,0)(16,0)}

\rput[B](-4,3.5){\scr $a_2$}
\rput[B](4,3.5){\scr $b_2$}
\rput[B](-4,-4.5){\scr $c_2$}
\rput[B](4,-4.5){\scr $d_2$}
}
}


\ncline[nodesep=30pt]{->}{a}{b}

\eps\]

\noi According to our definitions of horizontal and vertical tensor product, the first is a clique map represented by
\[(f \otimes g) \otimes (h \otimes j)\]
and the second is represented by
\[(f \otimes h) \otimes (g \otimes j).\]
As usual, to check that these represent the same clique map we need to check that the square involving connecting isomorphisms commute.  Now, the linearising braids in question are:

%


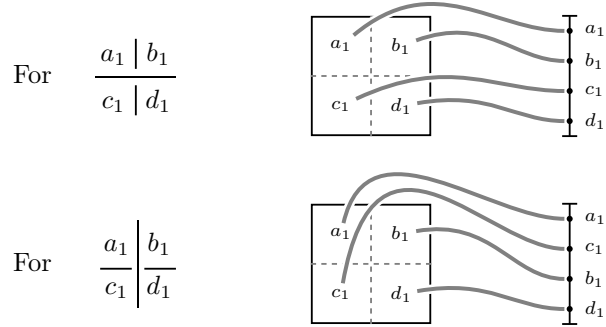
\[\renewcommand{\nowdot}{\psset{unit=1mm}\pscircle*(0,0){0.4}}
\ps(-30,4)(30,50)

\rput(-30,40){\bmp{10em}
For \hs{1}
$\setlength{\arraycolsep}{1.8pt}
\renewcommand{\arraystretch}{1.35}
\begin{array}{ccc}
\ul{a}_1 & \psline(0,-1)(0,3) & b_1 \\[1pt]
\hline 
c_1 & \psline(0,-1)(0,3) & d_1
\end{array}
$
\emp}

\rput(0,40){\rnode{a}{

\rput(-8,-8){
\psframe(0,0)(16,16)
\psline[linestyle=dashed, linecolor=gray](8,0)(8,16)
\psline[linestyle=dashed, linecolor=gray](0,8)(16,8)

\rput(4,12){\rnode{a1}{\scr $\ul{a}_1$}}
\rput(12,12){\rnode{a2}{\scr $b_1$}}
\rput(4,4){\rnode{b1}{\scr $c_1$}}
\rput(12,4){\rnode{b2}{\scr $d_1$}}

}
}

\rput(27,0){\rnode{a}{

\rput(0,-8){
\psline(0,0)(0,16)
\rput(0,16){\psline(-1,0)(1,0)}
\rput(0,0){\psline(-1,0)(1,0)}

\rput(0,0){
\rput(0,14){\rnode{aa1}{\nowdot}
\rput[l](2,0){\scr $\ul{a}_1$}}

\rput(0,10){\rnode{aa2}{\nowdot}
\rput[l](2,0){\scr $b_1$}}

\rput(0,6){\rnode{bb1}{\nowdot}
\rput[l](2,0){\scr $c_1$}}

\rput(0,2){\rnode{bb2}{\nowdot}
\rput[l](2,0){\scr $d_1$}}

}

%

\psset{border=1pt, linecolor=gray, linewidth=1.5pt}

\nccurve[angleA=40,angleB=180,ncurvA=1]{a1}{aa1}
\nccurve[angleA=20,angleB=180,ncurvA=1]{a2}{aa2}

\nccurve[angleA=25,angleB=180,ncurvA=1]{b1}{bb1}
\nccurve[angleA=10,angleB=180,ncurvA=1]{b2}{bb2}

}

}
}
}


\rput(-30,15){\bmp{10em}
For \hs{1}
$\setlength{\arraycolsep}{3pt}
\renewcommand{\arraystretch}{1.35}
\begin{array}{c|c}
a_1 & b_1 \\[-8pt]
\psline(-2.2,0.5)(2,0.5) & \psline(-2,0.5)(2.2,0.5) \\[-8pt]
c_1 & d_1 
\end{array}
$
\emp}

\rput(0,15){\rnode{a}{

\rput(-8,-8){
\psframe(0,0)(16,16)
\psline[linestyle=dashed, linecolor=gray](8,0)(8,16)
\psline[linestyle=dashed, linecolor=gray](0,8)(16,8)

\rput(4,12){\rnode{a1}{\scr $a_1$}}
\rput(12,12){\rnode{a2}{\scr $b_1$}}
\rput(4,4){\rnode{b1}{\scr $c_1$}}
\rput(12,4){\rnode{b2}{\scr $d_1$}}

}
}

\rput(27,0){\rnode{a}{

\rput(0,-8){
\psline(0,0)(0,16)
\rput(0,16){\psline(-1,0)(1,0)}
\rput(0,0){\psline(-1,0)(1,0)}

\rput(0,0){
\rput(0,14){\rnode{aa1}{\nowdot}
\rput[l](2,0){\scr $a_1$}}

\rput(0,6){\rnode{aa2}{\nowdot}
\rput[l](2,0){\scr $b_1$}}

\rput(0,10){\rnode{bb1}{\nowdot}
\rput[l](2,0){\scr $c_1$}}

\rput(0,2){\rnode{bb2}{\nowdot}
\rput[l](2,0){\scr $d_1$}}

}

%

\psset{border=1pt, linecolor=gray, linewidth=1.5pt}

\nccurve[angleA=77,angleB=180,ncurvA=1]{a1}{aa1}
\nccurve[angleA=10,angleB=180,ncurvA=1]{a2}{aa2}

\nccurve[angleA=80,angleB=180,ncurvA=1.7]{b1}{bb1}
\nccurve[angleA=10,angleB=180,ncurvA=1]{b2}{bb2}

}

}
}
}

\eps\]


\noi This latter linearising braid may equivalently be pictured as below (by moving the $c_1$ strand down):

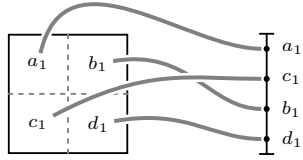
\[\renewcommand{\nowdot}{\psset{unit=1mm}\pscircle*(0,0){0.4}}
\ps(-30,2)(30,23)


\rput(0,12){\rnode{a}{

\rput(-8,-8){
\psframe(0,0)(16,16)
\psline[linestyle=dashed, linecolor=gray](8,0)(8,16)
\psline[linestyle=dashed, linecolor=gray](0,8)(16,8)

\rput(4,12){\rnode{a1}{\scr $a_1$}}
\rput(12,12){\rnode{a2}{\scr $b_1$}}
\rput(4,4){\rnode{b1}{\scr $c_1$}}
\rput(12,4){\rnode{b2}{\scr $d_1$}}

}
}

\rput(27,0){\rnode{a}{

\rput(0,-8){
\psline(0,0)(0,16)
\rput(0,16){\psline(-1,0)(1,0)}
\rput(0,0){\psline(-1,0)(1,0)}

\rput(0,0){
\rput(0,14){\rnode{aa1}{\nowdot}
\rput[l](2,0){\scr $a_1$}}

\rput(0,6){\rnode{aa2}{\nowdot}
\rput[l](2,0){\scr $b_1$}}

\rput(0,10){\rnode{bb1}{\nowdot}
\rput[l](2,0){\scr $c_1$}}

\rput(0,2){\rnode{bb2}{\nowdot}
\rput[l](2,0){\scr $d_1$}}

}

%

\psset{border=1pt, linecolor=gray, linewidth=1.5pt}

\nccurve[angleA=77,angleB=180,ncurvA=1]{a1}{aa1}
\nccurve[angleA=10,angleB=180,ncurvA=1]{a2}{aa2}

\nccurve[angleA=28,angleB=180,ncurvA=1]{b1}{bb1}
\nccurve[angleA=10,angleB=180,ncurvA=1]{b2}{bb2}

}

}
}
}

\eps\]

\noi Thus we see that the connecting isomorphism is given by the braiding in $B$
\[b_1 \otimes c_1 \tmap{\sigma} c_1 \otimes b_1\]
together with the necessary coherence isomorphisms to make an isomorphism
\[(a_1 \otimes b_1) \otimes (c_1 \otimes d_1) \ltra (a_1 \otimes c_1) \otimes (b_1 \otimes c_1)\]
and similarly for the target objects $a_2, b_2, c_2, d_2$.

So the square we need to check is the following, which does indeed commute.
\[
\psset{unit=0.1cm,labelsep=3pt,nodesep=3pt}
\pspicture(-10,-14)(10,16)



\rput(-30,10){\rnode{a1}{$(a_1 \otimes b_1) \otimes (c_1 \otimes d_1)$}}  
\rput(30,10){\rnode{a2}{$(a_2 \otimes b_2) \otimes (c_2 \otimes d_2)$}}  
\rput(-30,-10){\rnode{a3}{$(a_1 \otimes c_1) \otimes (b_1 \otimes d_1)$}}  
\rput(30,-10){\rnode{a4}{$(a_2 \otimes c_2) \otimes (b_2 \otimes d_2)$}}  

\ncline{->}{a1}{a2} \naput{{\scriptsize $(f \otimes g) \otimes (h \otimes j)$}} 
\ncline{->}{a3}{a4} \nbput{{\scriptsize $(f \otimes h) \otimes (g \otimes j)$}} 
\ncline{->}{a1}{a3} \nbput{{\scr \parbox{6em}{\bfr \sf{braiding}\\ \sf{and coherence}\efr}}} 
\ncline{->}{a2}{a4} \naput{{\scr \parbox{6em}{\sf{braiding}\\ \sf{and coherence}} }}

\endpspicture
\]

\noi Thus the two maps represent the same clique map as required, showing that interchange is strict.

We have now proved the following theorem.

\begin{theorem}
Given any braided monoidal category $B$, the category $\Sigma B$ has the structure of a doubly-degenerate \bicats-category.  
\end{theorem}

Note that when we use interchange to construct braidings for the main theorem, we will be dealing with two special cases where two of the objects are identities, that is, empty configurations.  The following linearising braid then simplifies: 

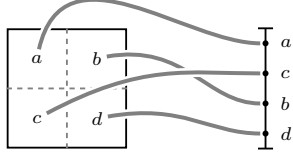
\[\renewcommand{\nowdot}{\psset{unit=1mm}\pscircle*(0,0){0.4}}
\ps(-10,3)(30,21)


\rput(0,12){\rnode{a}{

\rput(-8,-8){
\psframe(0,0)(16,16)
\psline[linestyle=dashed, linecolor=gray](8,0)(8,16)
\psline[linestyle=dashed, linecolor=gray](0,8)(16,8)

\rput(4,12){\rnode{a1}{\scr $a$}}
\rput(12,12){\rnode{a2}{\scr $b$}}
\rput(4,4){\rnode{b1}{\scr $c$}}
\rput(12,4){\rnode{b2}{\scr $d$}}

}
}

\rput(27,0){\rnode{a}{

\rput(0,-8){
\psline(0,0)(0,16)
\rput(0,16){\psline(-1,0)(1,0)}
\rput(0,0){\psline(-1,0)(1,0)}

\rput(0,0){
\rput(0,14){\rnode{aa1}{\nowdot}
\rput[l](2,0){\scr $a$}}

\rput(0,6){\rnode{aa2}{\nowdot}
\rput[l](2,0){\scr $b$}}

\rput(0,10){\rnode{bb1}{\nowdot}
\rput[l](2,0){\scr $c$}}

\rput(0,2){\rnode{bb2}{\nowdot}
\rput[l](2,0){\scr $d$}}

}

%

\psset{border=1pt, linecolor=gray, linewidth=1.5pt}

\nccurve[angleA=77,angleB=180,ncurvA=1]{a1}{aa1}
\nccurve[angleA=10,angleB=180,ncurvA=1]{a2}{aa2}

\nccurve[angleA=28,angleB=180,ncurvA=1]{b1}{bb1}
\nccurve[angleA=10,angleB=180,ncurvA=1]{b2}{bb2}

}

}
}
}

\eps\]

\begin{itemize}

\item If $a$ and $d$ are identities then we just have a braiding of $c$ past $b$.

\item If $c$ and $b$ are identities then the braid becomes trivial.

\end{itemize}

%
%
%
%
%
%
%
%
%
%
%
%
%
%
%
%
%

\section{The main theorem}
\label{maintheorem}

We now build up to our main theorem, giving a braided monoidal equivalence between $U \Sigma B$ and $B$.  Recall that $U \Sigma B$ is the underlying braided monoidal category of $\Sigma B$, with respect to the vertical tensor product.  We will start with just the underlying categories before addressing the braided monoidal structure.  Note that when a doubly-degenerate \bicats-category $A$ is expressed as a category with two monoidal structures, the underlying category of $UA$ is just the underlying category of $A$. Thus $U\Sigma B$ has the same objects and morphisms as $\Sigma B$.

\begin{proposition}
There is an equivalence of the underlying categories
\[B \tra U\Sigma B.\]
\end{proposition}

\begin{proof}
First we construct the functor, which we will call $W$.

\begin{itemize}
\item On objects: we send an object $b \in B$ to the slide clique of the singleton $b$ in a box:
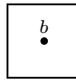
\[\pspicture(-5,-5)(5,5)

\psframe(-5,-5)(5,5)
\pscircle*(0,0){0.5}
\rput(0,2){\scr $b$}

\endpspicture\]
We will write this clique as $\ol{b}$. Note that this clique also contains all configurations where the dot is on the same vertical level but might have a different horizontal position, such as: 
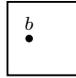
\[\pspicture(-5,-5)(5,5)

\psframe(-5,-5)(5,5)

\rput(-2,0){
\pscircle*(0,0){0.5}
\rput(0,2){\scr $b$}
}
\endpspicture\]

\item On morphisms, $a \tmap{f} b$ is sent to the clique map represented by $f$, as depicted below.

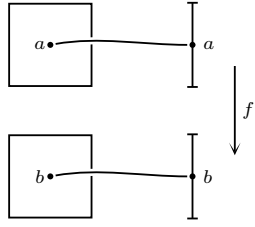
\[\renewcommand{\nowdot}{\psset{unit=1mm}\pscircle*(0,0){0.4}}
\psset{unit=0.7mm}
\ps(0,0)(60,40)

\rput(0,25){

\psframe(0,0)(16,16)

\rput(8,8){\rnode{a1}{\nowdot}
\rput(-2,0){\scr $a$}
}

\rput(35,0){
\psline(0,0)(0,16)
\rput(0,16){\psline(-1,0)(1,0)}
\rput(0,0){\psline(-1,0)(1,0)}

\rput(0,8){\rnode{a2}{\nowdot}
\rput[l](2,0){\scr $a$}}

}

\psset{border=1pt}

\nccurve[angleA=10,angleB=180]{a1}{a2}

}


\rput(0,0){

\psframe(0,0)(16,16)

\rput(8,8){\rnode{a1}{\nowdot}
\rput(-2,0){\scr $b$}
}

\rput(35,0){
\psline(0,0)(0,16)
\rput(0,16){\psline(-1,0)(1,0)}
\rput(0,0){\psline(-1,0)(1,0)}

\rput(0,8){\rnode{aa2}{\nowdot}
\rput[l](2,0){\scr $b$}}

}

\psset{border=1pt}

\nccurve[angleA=10,angleB=180]{a1}{aa2}

}

\ncline[arrows=->,offset=16pt, nodesep=8pt]{a2}{aa2} \naput{\scr $f$}


%



%

\eps\]

\end{itemize}

\noi Functoriality is immediate.

We now show that $W \: B \tra U \Sigma B$ is full and faithful.  Consider objects $a, b$ in $B$; we need to show that the function on homs
\[B(a,b) \ltra U\Sigma B(Wa, Wb)\]
is an isomorphism.  We know that $Wa = \ol{a}$ and  $Wb = \ol{b}$, so 
\[\renewcommand{\arraystretch}{1.2}  
\begin{array}{rcll}
U\Sigma B(Wa, Wb) &=& U\Sigma B (\ol{a}, \ol{b}) \\
&=& \Sigma B (\ol{a}, \ol{b})  & \mbox{by definition of $U$}\\
&=& \wt{B}\big(\theta\ol{a}, \theta\ol{b}\big) & \mbox{by definition of $\Sigma$}
\end{array}\] 
where we recall that the functor $\theta$ is the composite
\[\theta := \hs{-0.3}\raisebox{0.1em}{\makebox[0.6em][l]{\scalebox{2.3}[1.3]{$\widetilde{\hs{3}}$}}}\Pi_1\big(C(I^2, \cO)\big) \mtmap{\hs{0.3}F^*} \wt{\cF\cO} \mtmap{G_!} \wt{B}.\]

We now examine the action of $\theta$ on the clique $\ol{a}$. First note that every object of the clique $\ol{a}$ is a singleton dot, so every possible linearising braid is trivial (a single strand), and thus the clique $F^*\ol{a}$ has only trivial mediating braids.   The clique $G_!F^*\ol{a}$ (also written $\theta\ol{a}$) can then be represented by $a \in B$; other objects in the clique will include copies of the unit $I_B$, and the connecting maps are built from unit constraints.  Similarly the clique $\theta\ol{b}$ can then be represented by $b$, so
\[\wt{B}\big(\theta\ol{a}, \theta\ol{b}\big) \iso B(a,b)\]
thus
\[U \Sigma B(Wa, Wb) = \wt{B}(\theta\ol{a}, \theta\ol{b}) \iso B(a,b).\]
Moreover, via this isomorphism the action of the functor $W$
\[W\: B(a,b) \ltra U \Sigma B(Wa, Wb) \iso B(a,b)\]
sends a morphism $f$ to itself, so we see that $B \tmap{W} U\Sigma B$ is full and faithful as claimed.

We now show that the functor $W$ is essentially surjective.  Consider any $X \in U\Sigma B$.  We need to exhibit an object $b \in B$ such that $Wb \iso X$ in $U\Sigma B$.  Recall that $X$ is by definition a slide clique of a configuration of points labelled by objects of $B$. Set $b = \langle X \rangle$, any parenthesisation of the objects of $B$ labelling $X$.  Then $Wb$ is the clique of the singleton


%
%
%
%
%
%
%
%
%
%
%
%
%
%
%
%
%

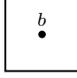
\[\pspicture(-5,-5)(5,5)

\psframe(-5,-5)(5,5)
\pscircle*(0,0){0.5}
\rput(0,2){\scr $b$}

\endpspicture\]
This is isomorphic to $X$ in $U\Sigma B$ via a clique map represented by the identity map $1_b$ in $B$.  (Note that it is represented by an identity map in $B$ but is not the identity as a clique map, so it is not the identity in $U\Sigma B$.)

This is perhaps elucidated by a specific example.  Suppose $X$ is the slide clique of the following labelled configuration:

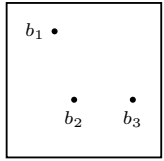
\[\renewcommand{\nowdot}{\psset{unit=1mm}\pscircle*(0,0){0.4}}
\psset{unit=1.3mm}
\ps(0,0)(16,15)

\rput(0,0){


\psframe(0,0)(16,16)
\pnode(8,8){x0}

\rput(5,13){\rnode{b1}{\nowdot}
\rput(-2,0){\scr $b_1$}
}

\rput(7,6){\rnode{b2}{\nowdot}
\rput(0,-2){\scr $b_2$}
}
\rput(13,6){\rnode{b3}{\nowdot}
\rput(0,-2){\scr $b_3$}}

}

\eps\]
We choose $b = b_1 \otimes (b_2 \otimes b_3)$. Then $Wb$ is the slide clique of the following singleton configuration
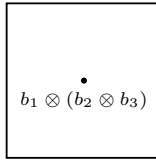
\[\renewcommand{\nowdot}{\psset{unit=1mm}\pscircle*(0,0){0.4}}
\psset{unit=1.3mm}
\ps(0,-1)(16,16)

\rput(0,0){


\psframe(0,0)(16,16)
\pnode(8,8){x0}

\rput(8,8){\rnode{b1}{\nowdot}
\rput(0,-2){\scr $b_1 \otimes (b_2 \otimes b_3)$}
}

%
}

\eps\]
We can then exhibit an isomorphism $X \tmap{\sim} Wb$  in $\Sigma B$ (and hence in $U\Sigma B$).  Such a morphism in $\Sigma B$ is by definition a clique map

\[\theta X \tra \theta (Wb) \in \wt{B}\]
Each of these cliques can be represented by the object $b = b_1 \otimes (b_2 \otimes b_3) \in B$ so $1_b$ is a valid morphism; moreover as it is invertible it is an isomorphism of the cliques in question. 

Note that $\theta (Wb)$ has a trivial linearising braid whereas $\theta X$ in general has a non-trivial one.  The map in $\Sigma B$ is depicted below.

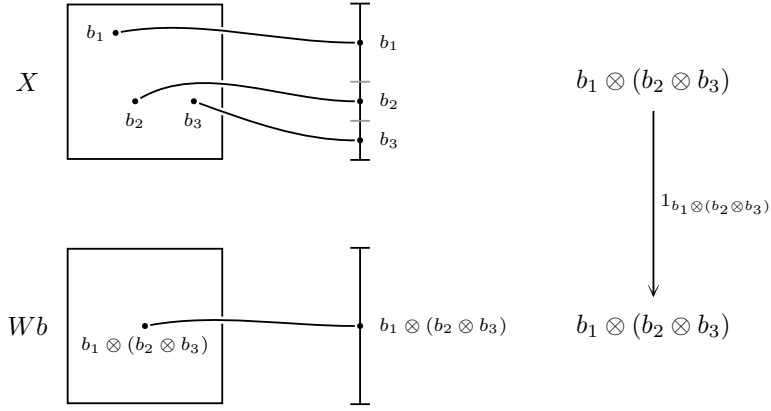
\[\renewcommand{\nowdot}{\psset{unit=1mm}\pscircle*(0,0){0.4}}
\psset{unit=1.3mm}
\ps(0,-2)(70,43)

\rput(0,25){


\psframe(0,0)(16,16)
\pnode(8,8){x0}

\rput(5,13){\rnode{b1}{\nowdot}
\rput(-2,0){\scr $b_1$}
}

\rput(7,6){\rnode{b2}{\nowdot}
\rput(0,-2){\scr $b_2$}
}
\rput(13,6){\rnode{b3}{\nowdot}
\rput(0,-2){\scr $b_3$}}

\rput(-4,8){$X$}

\rput(30,0){
\psline(0,0)(0,16)
\pnode(0,8){Fc0}
\rput(0,16){\psline(-1,0)(1,0)}
\rput(0,0){\psline(-1,0)(1,0)}

\rput(0,8){\psline[linecolor=gray!80!white](-1,0)(1,0)}
\rput(0,4){\psline[linecolor=gray!80!white](-1,0)(1,0)}

\rput(0,12){\rnode{c1}{\nowdot}
\rput[l](2,0){\scr $b_1$}}

\rput(0,6){\rnode{c2}{\nowdot}
\rput[l](2,0){\scr $b_2$}}

\rput(0,2){\rnode{c3}{\nowdot}
\rput[l](2,0){\scr $b_3$}}


\rput(30,8){\rnode{za}{$b_1 \otimes (b_2 \otimes b_3)$}}

}

\psset{border=1pt}

\nccurve[angleA=10,angleB=180]{b1}{c1}
\nccurve[angleA=30,angleB=180]{b2}{c2}
\nccurve[angleA=-20,angleB=180]{b3}{c3}


}


\rput(0,0){

\psframe(0,0)(16,16)
\pnode(8,8){x3}

\rput(-4,8){$Wb$}

\rput(8,8){\rnode{a1}{\nowdot}
\rput(0,-2){\scr $b_1 \otimes (b_2 \otimes b_3)$}
}



\rput(30,0){
\psline(0,0)(0,16)
\rput(0,16){\psline(-1,0)(1,0)}
\rput(0,0){\psline(-1,0)(1,0)}

\rput(30,8){\rnode{zb}{$b_1 \otimes (b_2 \otimes b_3)$}}


\rput(0,8){\rnode{a2}{\nowdot}
\rput[l](2,0){\scr $b_1 \otimes (b_2 \otimes b_3)$}}

}

\psset{border=1pt}

\nccurve[angleA=10,angleB=180]{a1}{a2}


\ncline[nodesep=6pt]{->}{za}{zb} \naput{\scr $1_{b_1 \otimes (b_2 \otimes b_3)}$}

}



{\psset{nodesep=35pt, linecolor=gray!80!white}

\ncline[offset=8pt]{x0}{x1}
\ncline{x0}{x1}
\ncline[offset=-8pt]{x0}{x1}


}

\eps\]

\noi We have exhibited $b \in B$ and an isomorphism $X \iso Wb$, so $B \tra \Sigma B$ is essentially surjective as claimed.  This completes the proof that $W\: B \lra U \Sigma B$ is an equivalence of categories.

\end{proof}

\begin{proposition}
The functor $B \tmap{W} U \Sigma B$ is monoidal. 
\end{proposition}

\begin{proof}

First we need to construct a constraint isomorphism
\[ Wa \otimes Wb \ltmap{\phi_{ab}} W(a \otimes b) \ \in  U\Sigma B\]
for any objects $a, b \in B$. The configurations for $Wa \otimes Wb$ and $W(a \otimes b)$ are shown below:

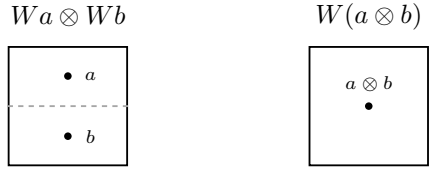
\[\pspicture(0,-8)(40,13)

\rput(0,0){
\pspicture(-8,-8)(8,8)

\rput(0,12){$Wa \otimes Wb$}

\psframe(-8,-8)(8,8)
\pscircle*(0,4){0.5}
\rput(3,4){\scr $a$}

\pscircle*(0,-4){0.5}
\rput(3,-4){\scr $b$}

\psline[linestyle=dashed, linecolor=gray!70!white](-8,0)(8,0)

\endpspicture
}

\rput(40,0){
\pspicture(-8,-8)(8,8)

\rput(0,12){$W(a \otimes b)$}

\psframe(-8,-8)(8,8)
\pscircle*(0,0){0.5}
\rput(0,3){\scr $a \otimes b$}

\eps
}

\eps
\]

%
%
%
%
%
%
%
%
%
%
%
%
%
%
%
%
%
%
%
%
%
%
%
%
%
%
%
%
%
%
%
%
%
%
%
%

\noi We can pick the constraint isomorphism to be the clique map represented by $1_{a \otimes b}$ as shown below.



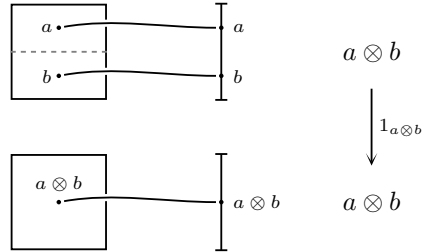
\[\psset{unit=0.8mm}
\ps(40,43)

\rput(0,25){

\psframe(0,0)(16,16)

\rput(8,12){\rnode{a1}{\nowdot}
\rput(-2,0){\scr $a$}
}

\rput(8,4){\rnode{b1}{\nowdot}
\rput(-2,0){\scr $b$}
}

\psline[linestyle=dashed, linecolor=gray](0,8)(16,8)

\rput(35,0){
\psline(0,0)(0,16)
\rput(0,16){\psline(-1,0)(1,0)}
\rput(0,0){\psline(-1,0)(1,0)}

\rput(25,8){\rnode{aa}{$a \otimes b$}}

\rput(0,12){\rnode{a2}{\nowdot}
\rput[l](2,0){\scr $a$}}

\rput(0,4){\rnode{b2}{\nowdot}
\rput[l](2,0){\scr $b$}}
}

\psset{border=1pt}

\nccurve[angleA=10,angleB=180]{a1}{a2}
\nccurve[angleA=10,angleB=180]{b1}{b2}


}

\rput(0,0){

\psframe(0,0)(16,16)

\rput(8,8){\rnode{b1}{\nowdot}
\rput[c](0,3){\scr $a \otimes b$}
}

\rput(35,0){
\psline(0,0)(0,16)
\rput(0,16){\psline(-1,0)(1,0)}
\rput(0,0){\psline(-1,0)(1,0)}

\rput(0,8){\rnode{b2}{\nowdot}
\rput[l](2,0){\scr $a \otimes b$}}

\rput(25,8){\rnode{bb}{$a \otimes b$}}

}

\psset{border=1pt}

\nccurve[angleA=10,angleB=180]{b1}{b2}


}

\ncline[nodesep=10pt,offset=0pt]{->}{aa}{bb}\naput{\scr $1_{a \otimes b}$}

\eps\]

For the unit constraint
\[I_{U\Sigma B} \tmap{\sim} WI_B\]
note that $I_{U\Sigma B}$ is the empty box, and $FI_B$ is the singleton labelled by $I_B$ so we can also choose this map to be a clique map represented by the identity, as shown below:


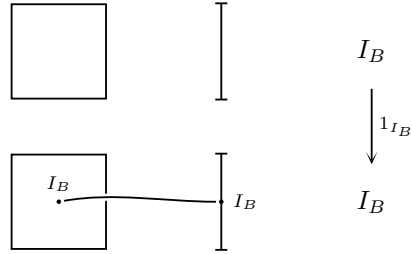
\[\psset{unit=0.8mm}
\ps(40,42)

\rput(0,25){

\psframe(0,0)(16,16)

%
%

\rput(35,0){
\psline(0,0)(0,16)
\rput(0,16){\psline(-1,0)(1,0)}
\rput(0,0){\psline(-1,0)(1,0)}

\rput(25,8){\rnode{aa}{$I_B$}}

%
%
%
}

\psset{border=1pt}

\nccurve[angleA=10,angleB=180]{a1}{a2}
\nccurve[angleA=10,angleB=180]{b1}{b2}


}

\rput(0,0){

\psframe(0,0)(16,16)

\rput(8,8){\rnode{b1}{\nowdot}
\rput[c](0,3){\scr $I_B$}
}

\rput(35,0){
\psline(0,0)(0,16)
\rput(0,16){\psline(-1,0)(1,0)}
\rput(0,0){\psline(-1,0)(1,0)}

\rput(0,8){\rnode{b2}{\nowdot}
\rput[l](2,0){\scr $I_B$}}

\rput(25,8){\rnode{bb}{$I_B$}}

}

\psset{border=1pt}

\nccurve[angleA=10,angleB=180]{b1}{b2}


}

\ncline[nodesep=10pt,offset=0pt]{->}{aa}{bb}\naput{\scr $1_{I_B}$}

\eps\]

\noi Note that neither of these maps is the identity in $U\Sigma B$ but both are clique maps represented by the identity in $B$, so are isomorphisms in the category of cliques.  

We now check the axioms for a monoidal functor. For interaction with the associator we need to check that the following diagram commutes.

\[
\psset{unit=0.1cm,labelsep=3pt,nodesep=3pt}
\pspicture(-20,-3)(20,33)



\rput[B](-25,30){\Rnode{tl}{$(Wa \otimes Wb) \otimes Wc$}}  
\rput[B](25,30){\Rnode{tr}{$Wa \otimes (Wb \otimes Wc)$}}  
\rput[B](-25,15){\Rnode{ml}{$W(a \otimes b) \otimes Wc$}}  
\rput[B](25,15){\Rnode{mr}{$Wa \otimes W(b \otimes c)$}}  
\rput[B](-25,0){\Rnode{bl}{$W\big((a \otimes b) \otimes c \big)$}}  
\rput[B](25,0){\Rnode{br}{$W\big(a \otimes (b \otimes c)\big)$}}  

\ncline{->}{tl}{tr} \naput{{\scriptsize \sf assoc}} 
\ncline{->}{bl}{br} \nbput{{\scriptsize $W(\mbox{\sf assoc})$}} 
\ncline{->}{tl}{ml} \nbput{{\scriptsize $\phi_{ab} \otimes 1_{Wc}$}} 
\ncline{->}{tr}{mr} \naput{{\scriptsize $1_{Wa} \otimes \phi_{bc}$}} 
\ncline{->}{ml}{bl} \nbput{{\scriptsize $\phi_{a\otimes b, c}$}} 
\ncline{->}{mr}{br} \naput{{\scriptsize $\phi_{a, b\otimes c}$}} 

\endpspicture
\]

\noi It is hard to fit the expression of this diagram into the page, so we will take it in parts. The left-hand side is a clique map that can be represented by the following composite of identities:

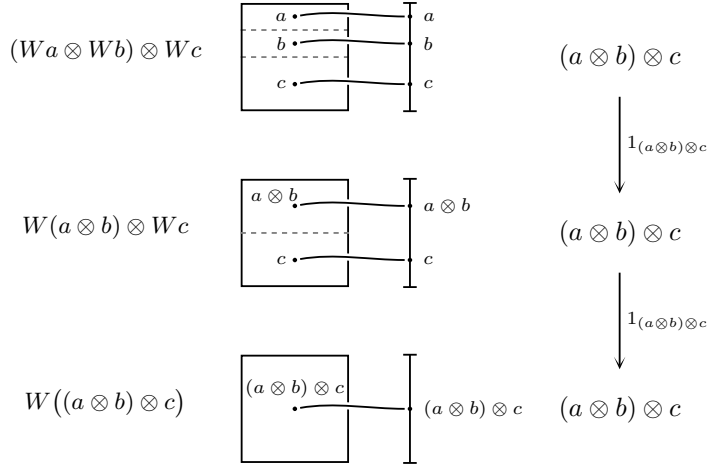
\[\psset{unit=0.9mm}
\ps(40,68)

\rput(0,52){

\rput[B](-20,8){\Rnode{tl}{\small $(Wa \otimes Wb) \otimes Wc$}}

\psline[linestyle=dashed, linecolor=gray](0,12)(16,12)
\psline[linestyle=dashed, linecolor=gray](0,8)(16,8)

\psframe(0,0)(16,16)

\rput(8,14){\rnode{a1}{\nowdot}
\rput(-2,0){\scr $a$}
}

\rput(8,10){\rnode{b1}{\nowdot}
\rput(-2,0){\scr $b$}
}

\rput(8,4){\rnode{c1}{\nowdot}
\rput(-2,0){\scr $c$}
}

\rput(35,0){

\rput(-10,0){

\psline(0,0)(0,16)
\rput(0,16){\psline(-1,0)(1,0)}
\rput(0,0){\psline(-1,0)(1,0)}

\rput(0,14){\rnode{a2}{\nowdot}
\rput[l](2,0){\scr $a$}}

\rput(0,10){\rnode{b2}{\nowdot}
\rput[l](2,0){\scr $b$}}

\rput(0,4){\rnode{c2}{\nowdot}
\rput[l](2,0){\scr $c$}}
}

\rput(21,8){\rnode{aa}{$(a \otimes b) \otimes c$}}

}

\psset{border=1pt}

\nccurve[angleA=10,angleB=180]{a1}{a2}
\nccurve[angleA=10,angleB=180]{b1}{b2}
\nccurve[angleA=10,angleB=180]{c1}{c2}


}


\rput(0,26){

\rput[B](-20,8){\Rnode{ml}{\small $W (a \otimes b) \otimes Wc$}}

\psframe(0,0)(16,16)

\rput(8,12){\rnode{a1}{\nowdot}
\rput(-3,1.5){\scr $a \otimes b$}
}

\rput(8,4){\rnode{b1}{\nowdot}
\rput(-2,0){\scr $c$}
}

\psline[linestyle=dashed, linecolor=gray](0,8)(16,8)

\rput(35,0){

\rput(-10,0){
\psline(0,0)(0,16)
\rput(0,16){\psline(-1,0)(1,0)}
\rput(0,0){\psline(-1,0)(1,0)}

\rput(0,12){\rnode{a2}{\nowdot}
\rput[l](2,0){\scr $a\otimes b$}}

\rput(0,4){\rnode{b2}{\nowdot}
\rput[l](2,0){\scr $c$}}
}

\rput(21,8){\rnode{bb}{$(a \otimes b) \otimes c$}}

}

\psset{border=1pt}

\nccurve[angleA=10,angleB=180]{a1}{a2}
\nccurve[angleA=10,angleB=180]{b1}{b2}


}


\rput(0,0){

\rput[B](-20,8){\Rnode{bl}{\small $W \big((a \otimes b) \otimes c\big)$}}

\psframe(0,0)(16,16)

\rput(8,8){\rnode{b1}{\nowdot}
\rput[c](0,3){\scr $(a \otimes b)\otimes c$}
}

\rput(35,0){

\rput(-10,0){
\psline(0,0)(0,16)
\rput(0,16){\psline(-1,0)(1,0)}
\rput(0,0){\psline(-1,0)(1,0)}

\rput(0,8){\rnode{b2}{\nowdot}
\rput[l](2,0){\scr $(a \otimes b) \otimes c$}}
}

\rput(21,8){\rnode{cc}{$(a \otimes b) \otimes c$}}

}

\psset{border=1pt}

\nccurve[angleA=10,angleB=180]{b1}{b2}


}

\ncline[nodesep=10pt,offset=0pt]{->}{aa}{bb}\naput{\scr $1_{(a \otimes b)\otimes c}$}
\ncline[nodesep=10pt,offset=0pt]{->}{bb}{cc}\naput{\scr $1_{(a \otimes b)\otimes c}$}

\eps\]

\noi Similarly, the right-hand side can be represented by the following composite of identities:

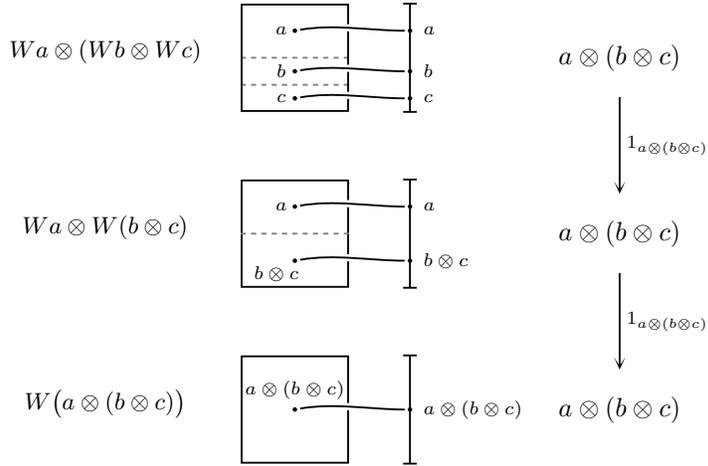
\[\psset{unit=0.9mm}
\ps(40,68)

\rput(0,52){

\rput[B](-20,8){\Rnode{tl}{\small $Wa \otimes (Wb \otimes Wc)$}}

\psline[linestyle=dashed, linecolor=gray](0,8)(16,8)
\psline[linestyle=dashed, linecolor=gray](0,4)(16,4)

\psframe(0,0)(16,16)

\rput(8,12){\rnode{a1}{\nowdot}
\rput(-2,0){\scr $a$}
}

\rput(8,6){\rnode{b1}{\nowdot}
\rput(-2,0){\scr $b$}
}

\rput(8,2){\rnode{c1}{\nowdot}
\rput(-2,0){\scr $c$}
}

\rput(35,0){

\rput(-10,0){

\psline(0,0)(0,16)
\rput(0,16){\psline(-1,0)(1,0)}
\rput(0,0){\psline(-1,0)(1,0)}

\rput(0,12){\rnode{a2}{\nowdot}
\rput[l](2,0){\scr $a$}}

\rput(0,6){\rnode{b2}{\nowdot}
\rput[l](2,0){\scr $b$}}

\rput(0,2){\rnode{c2}{\nowdot}
\rput[l](2,0){\scr $c$}}
}

\rput(21,8){\rnode{aa}{$a \otimes (b \otimes c)$}}

}

\psset{border=1pt}

\nccurve[angleA=10,angleB=180]{a1}{a2}
\nccurve[angleA=10,angleB=180]{b1}{b2}
\nccurve[angleA=10,angleB=180]{c1}{c2}


}


\rput(0,26){

\rput[B](-20,8){\Rnode{ml}{\small $W a \otimes W(b \otimes c)$}}

\psframe(0,0)(16,16)

\rput(8,12){\rnode{a1}{\nowdot}
\rput(-2,0){\scr $a$}
}

\rput(8,4){\rnode{b1}{\nowdot}
\rput(-2.7,-2){\scr $b \otimes c$}
}

\psline[linestyle=dashed, linecolor=gray](0,8)(16,8)

\rput(35,0){

\rput(-10,0){
\psline(0,0)(0,16)
\rput(0,16){\psline(-1,0)(1,0)}
\rput(0,0){\psline(-1,0)(1,0)}

\rput(0,12){\rnode{a2}{\nowdot}
\rput[l](2,0){\scr $a$}}

\rput(0,4){\rnode{b2}{\nowdot}
\rput[l](2,0){\scr $b \otimes c$}}
}

\rput(21,8){\rnode{bb}{$a \otimes (b \otimes c)$}}

}

\psset{border=1pt}

\nccurve[angleA=10,angleB=180]{a1}{a2}
\nccurve[angleA=10,angleB=180]{b1}{b2}


}


\rput(0,0){

\rput[B](-20,8){\Rnode{bl}{\small $W \big(a \otimes (b \otimes c)\big)$}}

\psframe(0,0)(16,16)

\rput(8,8){\rnode{b1}{\nowdot}
\rput[c](0,3){\scr $a \otimes (b\otimes c)$}
}

\rput(35,0){

\rput(-10,0){
\psline(0,0)(0,16)
\rput(0,16){\psline(-1,0)(1,0)}
\rput(0,0){\psline(-1,0)(1,0)}

\rput(0,8){\rnode{b2}{\nowdot}
\rput[l](2,0){\scr $a \otimes (b \otimes c)$}}
}

\rput(21,8){\rnode{cc}{$a \otimes (b \otimes c)$}}

}

\psset{border=1pt}

\nccurve[angleA=10,angleB=180]{b1}{b2}


}

\ncline[nodesep=10pt,offset=0pt]{->}{aa}{bb}\naput{\scr $1_{a \otimes (b\otimes c)}$}
\ncline[nodesep=10pt,offset=0pt]{->}{bb}{cc}\naput{\scr $1_{a \otimes (b\otimes c)}$}

\eps\]

\noi The top and bottom (horizontal) arrows are each represented by the associator 
\[(a \otimes b) \otimes c \mtra a \otimes (b \otimes c)\]
and so the diagram does indeed commute.  The unit diagrams follow similarly. 
\end{proof}

\begin{proposition}
The monoidal functor $B \tmap{W} U\Sigma B$ is braided.
\end{proposition}

\begin{proof}

To show that the monoidal functor $W$ is braided we need to show that the following diagram commutes (where we are writing $\sigma$ for the braiding in $B$ and also for the one in $U \Sigma B$).

\[
\psset{unit=0.1cm,labelsep=3pt,nodesep=3pt}
\pspicture(0,-4)(30,24)



\rput[B](0,20){\Rnode{a1}{$Wa \otimes Wb$}}  
\rput[B](30,20){\Rnode{a2}{$W(a \otimes b)$}}  
\rput[B](0,0){\Rnode{a3}{$Wb \otimes Wa$}}  
\rput[B](30,0){\Rnode{a4}{$W(b \otimes a)$}}  

\ncline{->}{a1}{a2} \naput{{\scriptsize $\phi_{ab}$}} 
\ncline{->}{a3}{a4} \nbput{{\scriptsize $\phi_{ba}$}} 
\ncline{->}{a1}{a3} \nbput{{\scriptsize $\sigma$}} 
\ncline{->}{a2}{a4} \naput{{\scriptsize $W(\sigma)$}} 

\endpspicture
\]

We need to examine what the braiding $\sigma$ in $U\Sigma B$ is.  It comes from the weak \eh\ argument, which is the following combination of unit constraints and interchange:

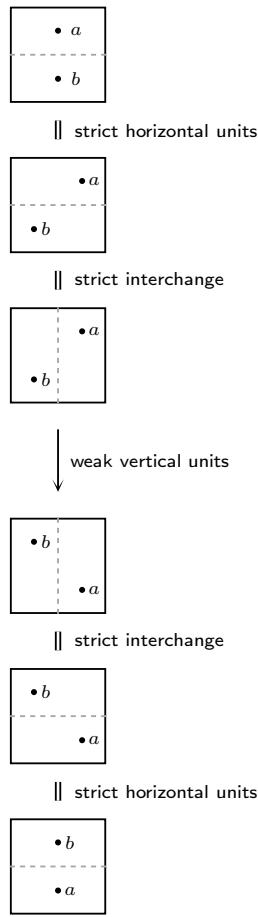
\[\psset{unit=0.8mm}
\pspicture(-8,-8)(8,148)


\rput(0,140){\rnode{a1}{
\psframe(-8,-8)(8,8)

\rput(0,4){\pscircle*(0,0){0.5}
\rput(3,0){\scr $a$}
}

\rput(0,-4){\pscircle*(0,0){0.5}
\rput(3,0){\scr $b$}}

\psline[linestyle=dashed, linecolor=gray!70!white](-8,0)(8,0)
}}


\rput(0,115){\rnode{a2}{\psframe(-8,-8)(8,8)

\rput(4,4){\pscircle*(0,0){0.5}
\rput(2,0){\scr $a$}
}

\rput(-4,-4){\pscircle*(0,0){0.5}
\rput(2,0){\scr $b$}}

\psline[linestyle=dashed, linecolor=gray!70!white](-8,0)(8,0)
}}


\rput(0,90){\rnode{a3}{\psframe(-8,-8)(8,8)

\rput(4,4){\pscircle*(0,0){0.5}
\rput(2,0){\scr $a$}
}

\rput(-4,-4){\pscircle*(0,0){0.5}
\rput(2,0){\scr $b$}}

\psline[linestyle=dashed, linecolor=gray!70!white](0,-8)(0,8)}}


\rput(0,55){\rnode{a4}{\psframe(-8,-8)(8,8)
\rput(4,-4){\pscircle*(0,0){0.5}
\rput(2,0){\scr $a$}
}
\rput(-4,4){\pscircle*(0,0){0.5}
\rput(2,0){\scr $b$}}
%
\psline[linestyle=dashed, linecolor=gray!70!white](0,-8)(0,8)
}}


\rput(0,30){\rnode{a5}{\psframe(-8,-8)(8,8)
\rput(4,-4){\pscircle*(0,0){0.5}
\rput(2,0){\scr $a$}
}
\rput(-4,4){\pscircle*(0,0){0.5}
\rput(2,0){\scr $b$}}
\psline[linestyle=dashed, linecolor=gray!70!white](-8,0)(8,0)
%
}}


\rput(0,5){\rnode{a6}{\psframe(-8,-8)(8,8)
\rput(0,-4){\pscircle*(0,0){0.5}
\rput(2,0){\scr $a$}
}
\rput(0,4){\pscircle*(0,0){0.5}
\rput(2,0){\scr $b$}}
\psline[linestyle=dashed, linecolor=gray!70!white](-8,0)(8,0)
%
}}

\psset{nodesep=32pt, labelsep=4pt}

\ncline[doubleline=true,arrows=-]{a1}{a2} \nbput{\scr\sf strict horizontal units}
\ncline[doubleline=true,arrows=-]{a2}{a3} \nbput{\scr\sf strict interchange}

\ncline[arrows=->,nodesep=28pt]{a3}{a4} \naput{\scr\sf weak vertical units}
\ncline[doubleline=true,arrows=-]{a4}{a5} \nbput{\scr\sf  strict interchange}
\ncline[doubleline=true,arrows=-]{a5}{a6} \nbput{\scr\sf strict horizontal units}

\endpspicture\]

\noi By definition, the morphisms of $U\Sigma B$ are those of $\Sigma B$, so these morphisms are all clique maps in $\wt{B}$.  

To build up this composite we begin with the unit constraints in $\Sigma B$. The horizontal unit constraints are identities; the vertical unit constraints are not identities but we can choose the representatives to be identities, as shown below.

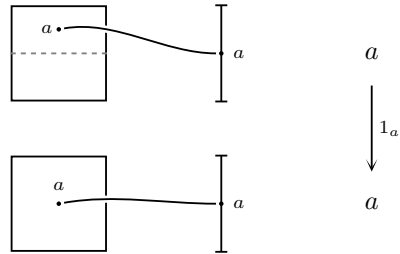
\[\psset{unit=0.8mm}
\ps(40,42)

\rput(0,25){

\psframe(0,0)(16,16)

\rput(8,12){\rnode{a1}{\nowdot}
\rput(-2,0){\scr $a$}
}

\rput(8,4){\rnode{b1}{}
\rput(-2,0){\scr $$}
}

\psline[linestyle=dashed, linecolor=gray](0,8)(16,8)

\rput(35,0){
\psline(0,0)(0,16)
\rput(0,16){\psline(-1,0)(1,0)}
\rput(0,0){\psline(-1,0)(1,0)}

\rput(25,8){\rnode{aa}{$a$}}

\rput(0,8){\rnode{a2}{\nowdot}
\rput[l](2,0){\scr $a$}}

\rput(0,4){\rnode{b2}{}
\rput[l](2,0){\scr $$}}
}

\psset{border=1pt}

\nccurve[angleA=10,angleB=180]{a1}{a2}


}

\rput(0,0){

\psframe(0,0)(16,16)

\rput(8,8){\rnode{b1}{\nowdot}
\rput[c](0,3){\scr $a$}
}

\rput(35,0){
\psline(0,0)(0,16)
\rput(0,16){\psline(-1,0)(1,0)}
\rput(0,0){\psline(-1,0)(1,0)}

\rput(0,8){\rnode{b2}{\nowdot}
\rput[l](2,0){\scr $a$}}

\rput(25,8){\rnode{bb}{$a$}}

}

\psset{border=1pt}

\nccurve[angleA=10,angleB=180]{b1}{b2}


}

\ncline[nodesep=10pt,offset=0pt]{->}{aa}{bb}\naput{\scr $1_{a}$}

\eps\]

For each step we need to take a vertical or horizontal tensor product of the constraint maps in question using the constructions defined in Section~\ref{sigmatensor}, or perform interchange on them.  


Note that interchange is strict, but our expression of it involves a change of representing object, so although it is the identity clique map it will be represented by the appropriate connecting isomorphism, as given in Section~\ref{sigmainterchange}.

%

Thus the braiding in $U \Sigma B$ is the clique map represented by the following composite:



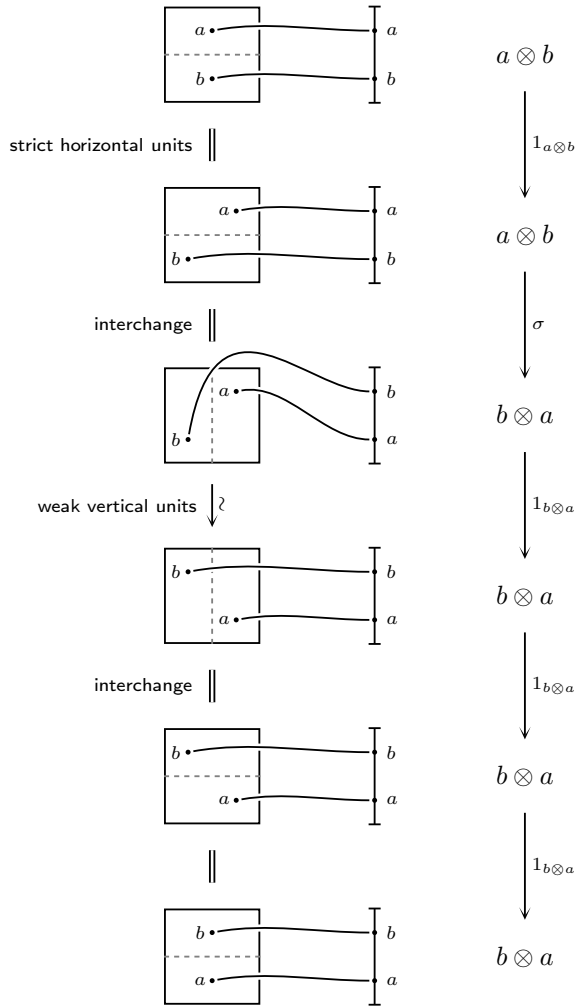
\[\renewcommand{\nowdot}{\psset{unit=1.1mm}\pscircle*(0,0){0.3}}
\psset{unit=0.8mm}
\ps(130,170)

\rput(0,150){

\psframe(0,0)(16,16)
\pnode(8,8){x1}

\rput(8,12){\rnode{a1}{\nowdot}
\rput(-2,0){\scr $a$}
}

\rput(8,4){\rnode{b1}{\nowdot}
\rput(-2,0){\scr $b$}
}

\psline[linestyle=dashed, linecolor=gray](0,8)(16,8)

\rput(35,0){
\psline(0,0)(0,16)
\rput(0,16){\psline(-1,0)(1,0)}
\rput(0,0){\psline(-1,0)(1,0)}

\rput(25,8){\rnode{z1}{$a \otimes b$}}

\rput(0,12){\rnode{a2}{\nowdot}
\rput[l](2,0){\scr $a$}}

\rput(0,4){\rnode{b2}{\nowdot}
\rput[l](2,0){\scr $b$}}
}

\psset{border=1pt}

\nccurve[angleA=10,angleB=180]{a1}{a2}
\nccurve[angleA=10,angleB=180]{b1}{b2}


}

\rput(0,120){

\psframe(0,0)(16,16)
\pnode(8,8){x2}

\rput(12,12){\rnode{a1}{\nowdot}
\rput(-2,0){\scr $a$}
}

\rput(4,4){\rnode{b1}{\nowdot}
\rput(-2,0){\scr $b$}
}

\psline[linestyle=dashed, linecolor=gray](0,8)(16,8)

\rput(35,0){
\psline(0,0)(0,16)
\rput(0,16){\psline(-1,0)(1,0)}
\rput(0,0){\psline(-1,0)(1,0)}

\rput(25,8){\rnode{z2}{$a \otimes b$}}

\rput[l](45,-5){\parbox{15em}{This is the identity as a clique map, but with a change of representing objects, so we represent the clique map by the connecting isomorphism $\sigma$.}}

\rput(0,12){\rnode{a2}{\nowdot}
\rput[l](2,0){\scr $a$}}

\rput(0,4){\rnode{b2}{\nowdot}
\rput[l](2,0){\scr $b$}}
}

\psset{border=1pt}

\nccurve[angleA=10,angleB=180]{a1}{a2}
\nccurve[angleA=10,angleB=180]{b1}{b2}


}

\rput(0,90){
\psframe(0,0)(16,16)
\pnode(8,8){x3}

\rput(12,12){\rnode{a1}{\nowdot}
\rput(-2,0){\scr $a$}
}

\rput(4,4){\rnode{b1}{\nowdot}
\rput(-2,0){\scr $b$}
}

\psline[linestyle=dashed, linecolor=gray](8,0)(8,16)

\rput(35,0){
\psline(0,0)(0,16)
\rput(0,16){\psline(-1,0)(1,0)}
\rput(0,0){\psline(-1,0)(1,0)}

\rput(25,8){\rnode{z3}{$b \otimes a$}}

\rput[l](45,-5){\parbox{15em}{This is not the identity as the objects in $U \Sigma B$ are different, but the clique map is represented by the identity.}}

\rput(0,4){\rnode{a2}{\nowdot}
\rput[l](2,0){\scr $a$}}

\rput(0,12){\rnode{b2}{\nowdot}
\rput[l](2,0){\scr $b$}}
}

\psset{border=1pt}

\nccurve[angleA=10,angleB=180]{a1}{a2}
\nccurve[angleA=80,angleB=180,ncurvA=1.7]{b1}{b2}


}

\rput(0,60){

\psframe(0,0)(16,16)
\pnode(8,8){x4}

\rput(12,4){\rnode{a1}{\nowdot}
\rput(-2,0){\scr $a$}
}

\rput(4,12){\rnode{b1}{\nowdot}
\rput(-2,0){\scr $b$}
}

\psline[linestyle=dashed, linecolor=gray](8,0)(8,16)

\rput(35,0){
\psline(0,0)(0,16)
\rput(0,16){\psline(-1,0)(1,0)}
\rput(0,0){\psline(-1,0)(1,0)}

\rput(25,8){\rnode{z4}{$b \otimes a$}}

\rput[l](45,-8){\parbox{15em}{This interchange has identities in the positions that generally require a braiding in the connecting isomorphism, so can in fact be represented by an identity morphism.}}

\rput(0,4){\rnode{a2}{\nowdot}
\rput[l](2,0){\scr $a$}}

\rput(0,12){\rnode{b2}{\nowdot}
\rput[l](2,0){\scr $b$}}
}

\psset{border=1pt}

\nccurve[angleA=10,angleB=180]{a1}{a2}
\nccurve[angleA=10,angleB=180]{b1}{b2}


}

\rput(0,30){

\psframe(0,0)(16,16)
\pnode(8,8){x5}

\rput(12,4){\rnode{a1}{\nowdot}
\rput(-2,0){\scr $a$}
}

\rput(4,12){\rnode{b1}{\nowdot}
\rput(-2,0){\scr $b$}
}

\psline[linestyle=dashed, linecolor=gray](0,8)(16,8)

\rput(35,0){
\psline(0,0)(0,16)
\rput(0,16){\psline(-1,0)(1,0)}
\rput(0,0){\psline(-1,0)(1,0)}

\rput(25,8){\rnode{z5}{$b \otimes a$}}

\rput(0,4){\rnode{a2}{\nowdot}
\rput[l](2,0){\scr $a$}}

\rput(0,12){\rnode{b2}{\nowdot}
\rput[l](2,0){\scr $b$}}
}

\psset{border=1pt}

\nccurve[angleA=10,angleB=180]{a1}{a2}
\nccurve[angleA=10,angleB=180]{b1}{b2}


}

\rput(0,0){
\psframe(0,0)(16,16)
\pnode(8,8){x6}

\rput(8,4){\rnode{a1}{\nowdot}
\rput(-2,0){\scr $a$}
}

\rput(8,12){\rnode{b1}{\nowdot}
\rput(-2,0){\scr $b$}
}

\psline[linestyle=dashed, linecolor=gray](0,8)(16,8)

\rput(35,0){
\psline(0,0)(0,16)
\rput(0,16){\psline(-1,0)(1,0)}
\rput(0,0){\psline(-1,0)(1,0)}

\rput(25,8){\rnode{z6}{$b \otimes a$}}

\rput(0,4){\rnode{a2}{\nowdot}
\rput[l](2,0){\scr $a$}}

\rput(0,12){\rnode{b2}{\nowdot}
\rput[l](2,0){\scr $b$}}
}

\psset{border=1pt}

\nccurve[angleA=10,angleB=180]{a1}{a2}
\nccurve[angleA=10,angleB=180]{b1}{b2}


}

\ncline[nodesep=10pt,offset=0pt]{->}{z1}{z2}\naput{\scr $1_{a \otimes b}$}
\ncline[nodesep=10pt,offset=0pt]{->}{z2}{z3}\naput{\scr $\sigma$}
\ncline[nodesep=10pt,offset=0pt]{->}{z3}{z4}\naput{\scr $1_{b \otimes a}$}
\ncline[nodesep=10pt,offset=0pt]{->}{z4}{z5}\naput{\scr $1_{b \otimes a}$}
\ncline[nodesep=10pt,offset=0pt]{->}{z5}{z6}\naput{\scr $1_{b \otimes a}$}

\ncline[doubleline=true,nodesep=28pt]{x1}{x2}\nbput[labelsep=5pt]{\sf\scr strict horizontal units}
\ncline[doubleline=true,nodesep=28pt]{x2}{x3} \nbput[labelsep=5pt]{\sf\scr interchange}
\ncline[nodesep=26pt,arrows=->]{x3}{x4} \naput[npos=0.45]{\scr\rotatebox{90}{$\sim$}}\nbput[labelsep=5pt]{\sf\scr weak vertical units}
\ncline[doubleline=true,nodesep=28pt]{x4}{x5} \nbput[labelsep=5pt]{\sf\scr interchange}
\ncline[doubleline=true,nodesep=28pt]{x5}{x6}

\eps\]

\noi As this composite is just $\sigma$, we see that both paths round the square in question are $\sigma$, so the diagram commutes.  So $B \tmap{W} U \Sigma B$ is a braided monoidal equivalence as claimed.

\end{proof}

We have proved the main theorem.

\begin{theorem}[Main Theorem]\label{theoremmain}
Given any braided monoidal category $B$ there is a \ddbicatscat\ $\Sigma B$ whose underlying braided monoidal category $U\Sigma B$ is braided monoidal equivalent to $B$.
\end{theorem}

\begin{remark}
It is worth noting why this proof does not work for structures with a strict vertical tensor product as well as a strict horizontal one.  Such structures would be doubly-degenerate strict 3-categories, so we know that not all braided monoidal categories arise in this way, thus an analogous proof of the main theorem should fail for such structures.  The issue here would be in the $\Sigma$ construction.  We defined the objects of $\Sigma B$ to be the slide cliques of configurations of points labelled by the objects of $B$; the horizontal slide cliques ensure that horizontal composition is strict. If we try to use horizontal and vertical slide maps at the same time, to make vertical composition also strict, we will not get a clique as there will not be uniquely specified isomorphisms between configurations that are now considered equivalent.  Thus the $\Sigma$ construction cannot be made in the first place.   
\end{remark}




\section{Future work}
\label{future}

Evidently this result is only an object-level result and not a result on totalities, but we judge this to be a worthwhile beginning, as in \cite{jk1}.  In future work we assemble doubly-degenerate \bicats-categories into a bicategory and show that $U$ and $\Sigma$ as defined in the present work extend to a biequivalence of this bicategory with the bicategory of braided monoidal categories, braided monoidal functors, and braided monoidal transformations.  The result in this paper provides the biessential surjectivity.  The subtlety required for the totality is a notion of weak functor between doubly-degenerate \bicats-categories.

In future work we will also prove the analogous results for doubly-degenerate Trimble 3-categories.  The ideas are essentially the same, except that we don't need the concept of slide cliques as we don't need to make the structures strict horizontally; rather, we need a way to extract a plain monoidal category structure from one parametrised by an operad, and that is where the cliques come in.


\end{document}